\documentclass[leqno]{amsart}
\usepackage{amssymb}
\usepackage{mathrsfs}
\usepackage{amsmath, amsfonts, vmargin, enumerate}
\usepackage{graphics}
\usepackage{color}
\usepackage{verbatim}

\usepackage{amsthm}
\usepackage{latexsym, bm}
\usepackage{euscript}
\usepackage{dsfont}
\usepackage{bbm}

\input xypic
\xyoption {all}

 \makeatletter

\newtheorem{thm}{Theorem}[section]
\newtheorem{prop}[thm]{Proposition}
\newtheorem{cor}[thm]{Corollary}
\newtheorem{lem}[thm]{Lemma}
\newtheorem{defn}[thm]{Definition}
\newtheorem{rem}[thm]{Remark}
\newtheorem{example}[thm]{Example}
\newtheorem{pb}[thm]{Problem}
\newtheorem{conj}[thm]{Conjecture}

\newenvironment{rmk}{\begin{rem}\rm}{\end{rem}}

\numberwithin{equation}{section}

\newcommand{\supp}{{\rm supp\,}}
\newcommand{\bmo}{{\rm bmo}}
\newcommand{\BMO}{{\rm BMO}}
\newcommand{\h}{{\rm h}}
\newcommand{\M}{\mathcal M}
\newcommand{\N}{\mathcal N}

\newcommand{\D}{\mathcal D}
\newcommand{\R}{\mathbb{R}^d}

\newcommand{\e}{\varepsilon}

\newcommand{\Z}{\mathbb{Z}^d}
\newcommand{\fk}{\mathsf{k} }
\newcommand{\fE}{\mathbf{E} }
\newcommand{\fF}{\mathbf{F} }

\begin{document}

\title{Operator-valued local Hardy spaces}

\author{Runlian  XIA}

\address{Laboratoire de Math{\'e}matiques, Universit{\'e} de Franche-Comt{\'e},
25030 Besan\c{c}on Cedex, France, and Instituto de Ciencias Matem{\'a}ticas, 28049 Madrid, Spain}
\email{runlian91@gmail.com}



\author{Xiao XIONG}

\address{Department of Mathematics and Statistics, University of Saskatchewan, Saskatoon, Saskatchewan, S7N 5E6, Canada}
\email{ufcxx56@gmail.com}

\maketitle

\markboth{R. Xia and X. Xiong}%
{Operator-valued local Hardy spaces}

\begin{abstract}
This paper gives a systematic study of operator-valued local Hardy spaces. These spaces are localizations of the Hardy spaces defined by Tao Mei, and share many properties with Mei's Hardy spaces. We prove the $\h_1$-$\bmo$ duality, as well as the $\h_p$-$\h_q$ duality for any conjugate pair $(p,q)$ when $1<p< \infty$. We show that $\h_1(\R, \M)$ and $\bmo(\R, \M)$ are also good endpoints of $L_p(L_\infty(\R) \overline{\otimes} \M)$ for interpolation. We obtain the local version of Calder\'on-Zygmund theory, and then deduce that the Poisson kernel in our definition of the local Hardy norms can be replaced by any reasonable test function. Finally, we establish the atomic decomposition of the local Hardy space $\h_1^c(\R,\M)$.
\end{abstract}

\tableofcontents

 \setcounter{section}{-1}

%%%%%%%%%%%%%%%%%%%%%%%%
\section{Introduction and Preliminaries}
%%%%%%%%%%%%%%%%%%%%%%%%

This paper is devoted to the study of operator-valued local Hardy spaces. It follows the current line of investigation of noncommutative harmonic analysis.  This field arose from the noncommutative integration theory developed by Murray and von Neumann, in order to provide a mathematical foundation for quantum mechanics. The objective was to construct and study a linear functional on an operator algebra which plays the role of the classical integral. In \cite{PiXu97}, Pisier and Xu developed a pioneering work on noncommutative martingale theory; since then, many classical results have been successfully transferred to the noncommutative setting, see for instance, \cite{Junge, Junge-Musat, Junge-Xu-03, Junge-Xu-08, Ran2002, PR2006, Ran2007}.

Inspired by the above mentioned developments and the Littlewood-Paley-Stein theory of quantum Markov semigroups (cf. \cite{JLX2006, JM2011, JM2010}), Mei \cite{Mei2007}  studied operator-valued Hardy spaces, which are defined by the Littlewood-Paley $g$-function and Lusin area integral function associated to the Poisson kernel.
 These spaces are shown to be very useful for many aspects of noncommutative harmonic analysis.
In \cite{XXX17}, we obtain general characterizations of Mei's Hardy spaces, which state that the Poisson kernel can be replaced by any reasonable test function. This is done mainly by using the operator-valued Calder\'on-Zygmund theory.

\medskip

In the classical setting, the theory of Hardy spaces is one of the most important topics in harmoic analysis. The local Hardy spaces $\h_p(\R)$ were first introduced by Goldberg  \cite{Goldberg1979}. These spaces are viewed as local or inhomogeneous counterparts of the classical real Hardy spaces $\mathcal{H}_p(\mathbb{R}^d)$.  Goldberg's motivation of introducing these local spaces was the study of pseudo-differential operators. It is known that pseudo-differential operators are not necessarily bounded on the classical Hardy space $\mathcal{H}_1(\mathbb{R}^d)$, but bounded on $\h_1(\R)$ under some appropriate assumptions. Afterwards, many other inhomogeneous spaces have also been studied. Our references for the classical theory are \cite{Goldberg1979, Torres1991, Fr-To-Wei-88}. However, they have not been investigated so far in the operator-valued case.

\medskip

Motivated by \cite{XXY2015, XXX17, Mei2007}, we provide a localization of Mei's operator-valued Hardy spaces on $\mathbb{R}^d$ in this paper. The norms of these spaces are partly given by the truncated versions of the Littlewood-Paley $g$-function and Lusin area integral function. Some techniques that we use to deal with our local Hardy spaces are modelled after those of \cite{XXX17}; however, some highly non-trivial modifications are needed. Since with the truncation, we only know the $L_p$-norms of the Poisson integrals of functions on the strip $\R\times (0,1)$, and lose information when the time is large. This brings some substantial difficulties that the non-local case does not have,  for example,  the duality problem. 
 Moreover, the noncommutative maximal function method is still unavailable in this setting, while in the classical case it is efficiently and frequently employed. However, based on tools developed recently, for instance, in \cite{PiXu97, Junge, Junge-Xu-03, Ran2002, Ran2007, JLX2006, Mei2007, Mei-tent}, we can overcome these difficulties.

\medskip

 Let us present here the four main results of this paper. The first family of results concerns the operator-valued local Hardy spaces $\h_{p}^{c}(\mathbb{R}^d,\mathcal{M})$ and $\bmo^c(\R,\M)$. The first major result of this part is the $\h_p^c$-$\bmo_q^c$ duality for $1\leq p<2$, where $q$ denotes the conjugate index of $p$. In particular, when $p=1$, we obtain the operator-valued local analogue of the classical Fefferman-Stein theorem. The pattern of the proof of this theorem is similar to that of Mei's non-local case. We also show that $\h_q^c(\R,\M)=\bmo_q^c(\R,\M)$ for $2<q<\infty$ like in the martingale and non-local settings. Thus the dual of $\h_p^c(\R,\M)$ agrees with $\h_q^c(\R,\M)$ when $1<p\leq 2$.
 
The second major result shows that the local Hardy spaces behave well with both complex and real interpolations. In particular, we have 
 \[
 \big(\bmo^c(\R,\M),\h_{1}^{c}(\mathbb{R}^d,\mathcal{M}) \big)_\frac{1}{p}=\h_{p}^{c}(\mathbb{R}^d,\mathcal{M})
 \]
 for $1<p<\infty$. We reduce this interpolation problem to the corresponding one on the non-local Hardy spaces in order to use Mei's interpolation result in \cite{Mei2007}. This proof is quite simple.
 
The third major result concerns the Calder\'on-Zygmund  theory. The usual $\M$-valued Calder\'on-Zygmund operators which satisfy the H\"{o}rmander condition are in general not bounded on inhomogeneous spaces. Thus, in order to guarantee the boundedness of a Calder\'on-Zygmund operator on $\h_{p}^{c}(\mathbb{R}^d,\mathcal{M})$, we need to impose an extra decay at infinity to the kernel. 

The Calder\'on-Zygmund  theory mentioned above will be applied to the general characterization of $\h_{p}^{c}(\R,\M)$ with the Poisson kernel replaced by any reasonable test function. This characterization will play an important role in our recent study of (inhomogeneous) Triebel-Lizorkin spaces on $\R$, see \cite{XX18}.

\medskip

\subsection{Notation}
In the following,we collect some notation which will be frequently used in this paper. Throughout, we will use the notation $A\lesssim B$, which is an inequality up to a constant: $A\leq cB$ for some constant
$c>0$. The relevant constants in all such inequalities may depend
on the dimension $d$, the test function $\Phi$ or $p$, etc, but
never on the function $f$ in consideration. The equivalence $A\approx B$
will mean $A\lesssim B$ and $B\lesssim A$ simultaneously.

\medskip
The Bessel and Riesz potentials are $J^\alpha =(1-(2\pi)^{-2}\Delta  )^{\frac \alpha  2}$ and  $I^\alpha =(-(2\pi)^{-2}\Delta  )^{\frac \alpha  2}$, respectively. If $\alpha=1$, we will abbreviate $J^1$ as $J$ and $I^1$ as $I$. We denote also $J_\alpha (\xi)=(1+|\xi|^2 )^{\frac \alpha  2}$ on $\mathbb{R}^d$ and $I_\alpha (\xi)=|\xi|^ \alpha  $ on $\mathbb{R}^d \setminus \{0\}$. Then $J_\alpha $ and $I_\alpha    $ are the symbols of the Fourier multipliers $J^\alpha $ and $I^\alpha $, respectively.

We denote by $H_2^\sigma(\R)$ the potential Sobolev space, consisting of all tempered distributions $f$ such that $J^\sigma(f)\in L_2(\R)$. If $\sigma> \frac d 2 $, the elements in $H_2^\sigma(\R)$ will serve as important convolution kernels in the sequel.

\subsection{Noncommutative $L_{p}$-spaces}
We also recall some preliminaries on noncommutative $L_p$-spaces and operator-valued Hardy spaces.
We start with a brief introduction of noncommutative $L_p$-spaces.
Let $\M$ be a von Neumann algebra equipped with a
normal semifinite faithful trace $\tau$ and
$S^+_{\M}$ be the set of all positive elements $x$
in $\M$ with $\tau(s(x))<\infty$, where $s(x)$ denotes the support of $x$, i.e.,
the smallest projection $e$ such that $exe=x$. Let
$S_{\M}$ be the linear span of
$S^+_{\M}$. Then every $x\in
S_{\M}$ has finite trace, and
$S_{\M}$ is a w*-dense $*$-subalgebra of
$\M$.

Let $1\leq p<\infty$. For any $x\in S_{\M}$, the
operator $|x|^p$ belongs to $S^+_{\M}$
(recalling $|x|=(x^*x)^{\frac{1}{2}}$). We define
$$\|x\|_p=\big(\tau(|x|^p)\big)^{\frac{1}{p}}.$$
One can prove that $\|\cdot\|_p$ is a norm on
$S_{\M}$. The completion of
$(S_{\M},\|\cdot\|_p)$ is denoted by $L_p(\M)$,
which is the usual noncommutative $L_p$-space associated to
$(\M,\tau)$. In this paper, the norm of $L_{p}  (\mathcal{M} )$ will be often denoted
simply by $  \|  \cdot \|  _{p}$ if there is no confusion. But if different  $L_p$-spaces appear in a same context, we will precise their norms in order to avoid possible ambiguity. We refer the reader  to \cite{Xu2007} and \cite{PX2003}
for further information on noncommutative $L_{p}$-spaces.

Now we introduce noncommutative Hilbert space-valued $L_{p}$-spaces
$L_{p}  (\mathcal{M};H^{c} )$ and $L_{p}  (\mathcal{M};H^{r} )$,
which are studied at length in \cite{JLX2006}. Let $H$ be a Hilbert
space and $v\in H$ with $  \|  v \|  =1$, and $p_{v}$ be the orthogonal projection onto the one-dimensional
subspace generated by $v$. Then define the following row and column
noncommutative $L_{p}$-spaces:
$$
L_{p}  (\mathcal{M};H^{r} )=(p_{v}\otimes1_{\mathcal{M}})L_{p}  (B(H)\overline{\otimes}\mathcal{M} )\mbox{ and }L_{p}  (\mathcal{M};H^{c} )=L_{p}  (B(H)\overline{\otimes}\mathcal{M} )(p_{v}\otimes1_{\mathcal{M}}),
$$
where the tensor product $B(H)\overline{\otimes}\mathcal{M}$ is equipped
with the tensor trace while $B(H)$ is equipped with the usual trace, and where $1_{\mathcal{M}}$ denotes the unit of $\M$. For $f\in L_{p}  (\mathcal{M};H^{c} )$,
$$
  \|  f \|  _{L_{p}  (\mathcal{M};H^{c} )}=\lVert   (f^{*}f )^{\frac{1}{2}}\rVert _p.
$$
 A similar formula holds for the row space by passing to adjoint:
$f\in L_{p}  (\mathcal{M};H^{r} )$  if and only if $f^{*}\in L_{p}  (\mathcal{M};H^{c} )$,
and $  \|  f \|  _{L_{p}  (\mathcal{M};H^{r} )}=  \|  f^{*} \|  _{L_{p}  (\mathcal{M};H^{c} )}$.
It is clear that $L_{p}  (\mathcal{M};H^{c} )$ and $L_{p}  (\mathcal{M};H^{r} )$
are 1-complemented  subspaces of $L_{p}  (B(H)\overline{\otimes}\mathcal{M} )$
for any $p$.

\subsection{Operator-valued Hardy spaces}

Throughout the remainder of the paper, unless explicitly stated otherwise,   $(\M,\tau)$ will be fixed as before  and $\N=L_\infty(\R)\overline\otimes\M$, equipped with the tensor trace. In this subsection, we introduce Mei's operator-valued Hardy spaces. Contrary to the custom, we will use letters $s, t$ to denote variables of $\R$ since letters $x, y$ are reserved for operators in noncommutative $L_p$-spaces. Accordingly, a generic element of the upper half-space $\mathbb{R}^{d+1}_+$ will be denoted by $(s,\e)$ with $\e>0$, where $\mathbb{R}^{d+1}_+=\{(s,\e): s\in\R,\, \e>0\}$.

Let $\mathrm{P}$ be the Poisson kernel on $\R$:
 $$\mathrm{P}(s)=c_d\,\frac{1}{(|s|^2+1)^{\frac{d+1}2}}$$
with $c_d$ the usual normalizing constant and $|s|$ the Euclidean norm of $s$. Let
  $$\mathrm{P}_\e(s)=\frac1{\e^d}\, \mathrm{P}(\frac s\e)=c_d\,\frac{\e}{(|s|^2+\e^2)^{\frac{d+1}2}}\,.$$
For any function $f$ on $\R$ with values in $L_1(\M)+\M$,  its Poisson integral, whenever it exists, will be denoted by $\mathrm{P}_\e(f)$:
 $$\mathrm{P}_\e(f)(s)=\int_{\mathbb{R}^d}\mathrm{P}_\e(s-t)f(t)dt, \quad (s,\e)\in \mathbb{R}^{d+1}_+.$$
Note that the Poisson integral of $f$ exists if
 $$f\in L_1\big(\M; L^c_2(\R,\frac{dt}{1+|t|^{d+1}})\big)+L_{\infty}\big(\M;L^c_2(\R,\frac{dt}{1+|t|^{d+1}})\big).$$
This space is the right space in which all functions considered in this paper live as far as only column spaces are involved. As it will appear frequently later, to simplify notation, we will denote the Hilbert space $L_2(\R,\frac{dt}{1+|t|^{d+1}})$ by $\mathrm{R}_d$:
 \begin{equation}\label{eq: Rd}
  \mathrm{R}_d=L_2(\R,\frac{dt}{1+|t|^{d+1}}).
 \end{equation}
The Lusin area square function of $f$ is defined by
 \begin{equation}\label{Lusin}
 S^c(f) (s) = \Big(\int_{\Gamma} \big|\frac{\partial}{\partial\e} \mathrm{P}_\e(f)(s+t)\big|^2\,\frac{dt\,d\e}{\e^{d-1}}\Big)^{\frac 1 2}, \quad s\in \R,
 \end{equation}
where $\Gamma$ is the cone $\{(t,\e)\in \mathbb{R}^{d+1}_+: |t|<\e\}$.
For $1\leq p<\infty$ define the column Hardy space $\mathcal{H}_p^c(\mathbb{R}^d, \M)$ to be
 $$\mathcal{H}_p^c(\mathbb{R}^d, \M)=\big\{f: \|f\|_{\mathcal{H}_p^c}
 =\|S^c(f)\|_p<\infty\big\}.$$
Note that  \cite{Mei2007} uses the gradient of $\mathrm{P}_\e(f)$ instead of the sole radial derivative in the definition of $S^c$ above,  but this does not affect  $\mathcal{H}_p^c(\mathbb{R}^d, \M)$ (up to equivalent norms). At the same time, it is proved in  \cite{Mei2007}  that $\mathcal{H}_p^c(\mathbb{R}^d, \M)$ can be equally defined by the Littlewood-Paley $g$-function:
\begin{equation}\label{LP}
 G^c(f) (s) = \Big(\int_{0}^\infty \e\, \big|\frac{\partial }{\partial\e}\mathrm{P}_\e(f)(s)\big|^2\,d\e \Big)^{\frac 1 2}, \quad s\in \mathbb{R}^d.
 \end{equation}
 Thus
  $$\|f\|_{\mathcal{H}_p^c}\approx \|G^c(f)\|_p,\quad f\in \mathcal{H}_p^c(\mathbb{R}^d, \M).$$
The row Hardy space $\mathcal{H}_p^r(\mathbb{R}^d, \M)$ is the space of all $f$ such that $f^*\in\mathcal{H}_p^c(\mathbb{R}^d, \M)$, equipped with the norm
 $\|f\|_{\mathcal{H}_p^r}= \|f^*\|_{\mathcal{H}_p^c}\,.$
Finally, we define the mixture space $\mathcal{H}_p(\mathbb{R}^d, \M)$ as
 $$\mathcal{H}_p(\mathbb{R}^d, \M)=\mathcal{H}_p^c(\mathbb{R}^d, \M)+\mathcal{H}_p^r(\mathbb{R}^d, \M)
 \;\textrm{ for }\; 1\leq p\le2$$
equipped with the sum norm
 $$\|f\|_{\mathcal{H}_p}=\inf\big\{\|f_1\|_{\mathcal{H}^c_p}+\|f_2\|_{\mathcal{H}^r_p}: f=f_1+f_2\big\},$$
and
 $$\mathcal{H}_p(\mathbb{R}^d, \M)=\mathcal{H}_p^c(\mathbb{R}^d, \M)\cap\mathcal{H}_p^r(\mathbb{R}^d, \M)
 \;\textrm{ for }\; 2< p<\infty$$
equipped with the intersection norm
 $$\|f\|_{\mathcal{H}_p}=\max\big(\|f\|_{\mathcal{H}^c_p}\,, \|f\|_{\mathcal{H}^r_p}\big).$$
 Observe that
 $$\mathcal{H}_2^c(\mathbb{R}^d, \M)=\mathcal{H}_2^r(\mathbb{R}^d, \M)
 =L_2(\N)\;\textrm{ with equivalent norms.}$$
It is proved in  \cite{Mei2007} that for $1<p<\infty$
 $$ \mathcal{H}_p(\mathbb{R}^d, \M)=L_p(\N)\;\textrm{ with equivalent norms.}$$

The operator-valued BMO spaces are also studied in  \cite{Mei2007}. Let $Q$ be a cube in $\R$ (with sides parallel to the axes) and $|Q|$ its volume. For a function $f$ with values in $\M$, $f_Q$ denotes its mean over $Q$:
 $$f_Q = \frac{1}{|Q|} \int_Q f(t)dt.$$
The column BMO norm of $f$ is defined to be
 \begin{equation}
 \|f\|_{\BMO^c}=\sup_{Q\subset\R}\Big\|\frac{1}{|Q|}\int_Q\big|f(t)-f_Q\big|^2dt\Big\|_{\M}^{\frac 12}.
 \end{equation}
Then
 $$\mathrm{BMO}^c(\R,\M)=\big\{ f \in
 L_{\infty}\big(\M; \mathrm{R}_d^c\big):\; \|f\|_{\mathrm{BMO}^c}<\infty\big\}.$$
Similarly, we define the row space $\mathrm{BMO}^r(\R,\M)$ as the space of $f$ such that $f^*$ lies in $\mathrm{BMO}^c(\R,\M)$,  and  $\mathrm{BMO}(\R,\M)=\mathrm{BMO}^c(\R,\M)\cap \mathrm{BMO}^r(\R,\M)$  with the intersection norm.

In \cite{Mei2007}, it is showed that the dual of  $\mathcal{H}_1^c(\mathbb{R}^d, \M)$ can be naturally identified with $\mathrm{BMO}^c(\R,\M)$. This is the operator-valued analogue of the celebrated Fefferman-Stein $H_1$-BMO duality theorem. 

On the other hand, one of the main results of \cite{XXX17} asserts that the Poisson kernel in the definition of Hardy spaces can be replaced by more general test functions.

Take any Schwartz function $\Phi$ with vanishing mean. We will assume that $\Phi$ is
nondegenerate in the following sense:
\begin{equation}\label{eq: nondegenerate}
  \forall\,\xi\in\R\setminus\{0\}\,\,\exists\, \varepsilon>0,\,\,\text{ s.t. }\,\widehat{\Phi}  (\varepsilon\xi )\neq0.
\end{equation}
Set $\Phi_\e(s) = \e^{-d}  \Phi(\frac{s}{\e})$ for $\e>0$. The radial and conic square functions of $f$ associated to $\Phi$ are defined by replacing the partial derivative of the Poisson kernel $\rm P$ in $S^c(f)$ and $G^c(f)$  by $\Phi$ :
\begin{equation}
 S_{\Phi}^c(f) (s) = \Big(\int_{\Gamma} | \Phi_\e *f (s+t)|^2\frac{dtd\e}{\e^{d+1}}\Big)^{\frac 1 2} \,, \quad s\in\R\label{eq: S}
\end{equation}
and 
\begin{equation}\label{eq: G}
 G_{\Phi}^c(f) (s) =  \Big(\int_0^\infty | \Phi_\e *f (s)|^2\frac{d\e}{\e}\Big)^{\frac 1 2}.
\end{equation}
The following two lemmas are taken from \cite{XXX17}. The first one says that the two square functions above define equivalent norms in $\mathcal{H}_p^c(\R,\M)$:
\begin{lem}\label{equivalence Hp}
Let $1\leq p<\infty$ and $f\in L_{1}(\M; \mathrm{R}_d^c)+L_{\infty}(\M; \mathrm{R}_d^c)$. Then $f\in\mathcal{H}_p^c(\R,\M)$  if and only if  $G_{\Phi}^c(f)\in L_{p}(\N)$  if and only if $S_{\Phi}^c(f)\in L_{p}(\N)$. If this is the case, then
 $$\|G_{\Phi}^c(f)\|_p  \approx \|S_{\Phi}^c(f)\|_p \approx \|f\|_{\mathcal{H}^c_p}$$
with the relevant constants depending only on $p, d$ and $\Phi$.
\end{lem}

The above square functions $G_{\Phi}^c$ and $S_{\Phi}^c$ can be discretized as follows:
\begin{equation}
\begin{split}
 G_{\Phi}^{c,D}(f)(s)  &=  \Big(\sum_{j=-\infty}^{\infty} |\Phi_{2^{-j}}*f (s)|^2\Big)^{\frac 1 2}\\
 S_{\Phi}^{c,D}(f)(s)  &=  \Big(\sum_{j=-\infty}^{\infty} 2^{dj}\int_{B(s, 2^{-j})} |\Phi_{2^{-j}}*f (t)|^2 dt\Big)^{\frac 1 2}\,.
 \end{split}
\end{equation}
Here $B(s, r)$ denotes the ball of $\R$ with center $s$ and radius $r$. To prove that these discrete square functions also describe our Hardy spaces, we need to impose the following  condition on the previous Schwartz function $\Phi$, which is stronger than \eqref{eq: nondegenerate}:
 \begin{equation}\label{schwartz D}
  \forall\,\xi\in\R\setminus\{0\}\,\,\exists\, 0<2a\le b<\infty\;\text { s.t. }\; \widehat{\Phi}(\e\xi)\neq0,\;\forall\; \e\in (a,\,b].
 \end{equation}
 The following is the discrete version of Lemma \ref{equivalence Hp}:
 \begin{lem}\label{equivalence HpD}
Let $1\leq p<\infty$ and $f\in L_{1}(\M; \mathrm{R}_d^c)+L_{\infty}(\M; \mathrm{R}_d^c)$. Then $f\in\mathcal{H}_p^c(\R,\M)$  if and only if  $G_{\Phi}^{c, D}(f)\in L_{p}(\N)$  if and only if $S_{\Phi}^{c, D}(f)\in L_{p}(\N)$. Moreover,
 $$\|G_{\Phi}^{c, D}(f)\|_p \approx \|S_{\Phi}^{c, D}(f)\|_p\approx \|f\|_{\mathcal{H}^c_p}$$
with the relevant constants depending only on $p, d$ and $\Phi$.
\end{lem}

\medskip

Finally, let us give some easy facts on operator-valued functions. The first one is the following Cauchy-Schwarz type inequality for the operator-valued square function,
\begin{equation}
  \big|\int_{\mathbb{R}^d}\phi (s)f(s)ds \big|^{2}\leq\int_{\mathbb{R}^d}  |\phi(s) |^{2}ds\int_{\mathbb{R}^d}  |f(s) |^{2}ds,\label{eq: 2-1}
\end{equation}
where $\phi:\mathbb{R}^d\rightarrow\mathbb{C}$ and $f:\mathbb{R}^d\rightarrow L_{1}(\mathcal{M} )+\mathcal{M}$
are functions such that all integrations of the above inequality make sense.
We also require the operator-valued version of the Plancherel formula. For sufficiently nice functions $f: \mathbb{R}^{d}\rightarrow L_{1}  (\mathcal{M} )+\M$,
for example, for $f \in L_{2}  (\mathbb{R}^{d} )\otimes L_{2}  (\mathcal{M} )$,
we have
\begin{equation}
\int_{\mathbb{R}^d} |f  (s )|^2 ds=\int_{\mathbb{R}^d}  |  \widehat{f}  (\xi )|^2  d\xi.\label{eq: Planchel}
\end{equation}
Given two nice functions $f$ and $g$, the polarized version of the above equality is
\begin{equation}
\int_{\mathbb{R}^d} f  (s )g^{*}  (s )ds=\int_{\mathbb{R}^d}\widehat{f}  (\xi )\widehat{g}  (\xi )^{*}d\xi.\label{eq: Planchel-1}
\end{equation}

\bigskip

The paper is organized as follows. In the next section, we give the definitions of operator-valued local Hardy and bmo spaces. Section \ref{section-dual} is devoted to the proofs of duality results, including the $\h_1 $-$ \bmo$ duality and the $\h_p $-$\h_q$ duality for $1<p<2$ and $\frac{1}{p} +\frac 1 q =1$. Section \ref{section-interp} gives the results of interpolation. In section \ref{section-CZ}, we develop Calder\'on-Zymund theory that is suitable for our local version of Hardy spaces. In section \ref{section-general charact}, we prove general characterizations of $\h_p^c(\R,\M)$, and then connect the local Hardy spaces $\h_p^c(\R,\M)$ with Mei's non-local Hardy spaces $\mathcal{H}_p^c(\R,\M)$. In the last section of this paper, we give the atomic decomposition of $\h_1^c(\R,\M)$.

%%%%%%%%%%%%%%%%%%%%%%%%%%%%
\section{Operator-valued local Hardy spaces}
%%%%%%%%%%%%%%%%%%%%%%%%%%%%%

\subsection{Operator-valued local Hardy spaces}
In this subsection, we give the definition of operator-valued local Hardy spaces as well as some basic facts of them. Let $f\in L_1(\M; \mathrm{R}_d^c)+L_{\infty}(\M;\mathrm{R}_d^c)$ (recalling that the Hilbert space $\mathrm{R}_d$ is defined by \eqref{eq: Rd}). Then the Poisson integral of $f$ is well-defined and takes values in $L_1(\M)+\M$. Now we define the local analogue of the Lusin area square function of $f$
by
\begin{equation*}
s^{c}  (f )  (s )  =  \Big(\int_{\widetilde{\Gamma}}  \big|\frac{\partial}{\partial\varepsilon}\mathrm{P}_{\varepsilon}  (f )  (s+t ) \big|^{2}\frac{dtd\varepsilon}{\varepsilon^{d-1}}\Big)^{\frac{1}{2}},\,s\in\mathbb{R}^{d},
\end{equation*}
where $\widetilde{\Gamma}$ is the truncated cone $  \{ (t,\varepsilon )\in\mathbb{R}_{+}^{d+1}:  |t |<\varepsilon<1 \} $. It is the intersection of the cone $  \{ (t,\varepsilon )\in\mathbb{R}_{+}^{d+1}:  |t |<\varepsilon  \} $ and the strip $S\subset\mathbb{R}_{+}^{d+1}$ defined by:
$$S=  \{   (s,\varepsilon ):s\in\mathbb{R}^{d},0<\varepsilon< 1 \}.
$$
For $1\leq p<\infty$ define the column local Hardy space $\h _p^c(\R,\M)$ to be
$$
\h _p^c(\R,\M)=\{f\in L_1(\M; \mathrm{R}_d^c)+L_{\infty}(\M;\mathrm{R}_d^c) : \|  f\| _{\h_p^c} <\infty\},
$$
where the $\h_{p}^{c}  (\mathbb{R}^{d},\mathcal{M} )$-norm of $f$
is defined by
\begin{equation*}
  \|  f \|  _{\h_{p}^{c} }  =    \|  s^{c}  (f )  \|  _{L_{p}  (\mathcal{N} )}+  \| \mathrm{P}*f \| _{L_{p}  (\mathcal{N} )}.
\end{equation*}
The row local Hardy space $\h_{p}^{r}(\R,\M)$ is the space of all $f$ such that $f^*\in \h_{p}^{c}(\R,\M)$, equipped with the norm $ \|  f\| _{\h_p^r}= \|  f^*\| _{\h_p^c}$.
Moreover, define the mixture space $\h_{p}  (\mathbb{R}^{d},\mathcal{M} )$
as follows:
$$
\h_{p}  (\mathbb{R}^{d},\mathcal{M} )=\h_{p}^{c}  (\mathbb{R}^{d},\mathcal{M} )+\h_{p}^{r}  (\mathbb{R}^{d},\mathcal{M} )\mbox{ for }1\leq p\leq2
$$
 equipped with the sum norm
$$
  \|  f \|  _{\h_{p}}=\inf  \{   \|  g \|  _{\h_{p}^{c}}+  \|  h \|  _{\h_{p}^{r}}:\, f=g+h,g\in \h_{p}^{c}  (\mathbb{R}^{d},\mathcal{M} ),h\in \h_{p}^{r}  (\mathbb{R}^{d},\mathcal{M} ) \} ,
$$
 and
$$
\h_{p}  (\mathbb{R}^{d},\mathcal{M} )=\h_{p}^{c}  (\mathbb{R}^{d},\mathcal{M} )\cap \h_{p}^{r}  (\mathbb{R}^{d},\mathcal{M} )\mbox{ for }2<p<\infty
$$
equipped with the intersection norm
$$
  \|  f \|  _{\h_{p} }=\max  \{   \|  f \|  _{\h_{p}^{c}},  \|  f \|  _{\h_{p}^{r}} \} .
$$
The local analogue of the Littlewood-Paley $g$-function of $f$ is
defined by
\begin{equation*}
g^{c}  (f )  (s )  = \Big(\int_{0}^1  |\frac{\partial}{\partial\varepsilon} \mathrm{P}_{\varepsilon}  (f )  (s ) |^{2}\varepsilon d\varepsilon \Big)^{\frac{1}{2}},\,s\in\mathbb{R}^{d}.
\end{equation*}
We will see in section \ref{section-general charact} that
$$  \|  s^{c}  (f ) \|  _p+ \| \mathrm{P}*f \| _p \approx  \|  g^{c}  (f ) \|  _p +\| \mathrm{P}*f \| _p$$
for all $1\leq p <\infty$.

\medskip

In the following, we give some easy facts that will be frequently used later. Firstly, we have
\begin{equation}\label{eq: 1}
  \|  s^{c}  (f ) \|  _2^{2}+     \|  \mathrm{P}*f \|  _2^{2}\approx    \|  f \|  _2^{2}.
\end{equation}
Indeed, by \eqref{eq: Planchel}, we have
\begin{equation*}
\begin{split}
\int_{\mathbb{R}^{d}}  \big|\frac{\partial}{\partial\varepsilon}\mathrm{P}_{\varepsilon}  (f )  (s ) \big|^{2}ds   &= \int_{\mathbb{R}^{d}}  \big|\widehat{\frac{\partial}{\partial\varepsilon}\mathrm{P}_{\varepsilon}}  (\xi ) \big|^{2}  |\widehat{f}  (\xi ) |^{2}d\xi\\
 &   = \int_{\mathbb{R}^{d}}4\pi^{2}  |\xi |^{2}  |\widehat{f}  (\xi ) |^{2}e^{-4\pi\varepsilon  |\xi |}d\xi.
\end{split}
\end{equation*}
Then
\begin{equation*}
\int_{\mathbb{R}^{d}}\int_{0}^{1}  \big|\frac{\partial}{\partial\varepsilon}\mathrm{P}_{\varepsilon}  (f )  (s ) \big|^{2}\varepsilon d\varepsilon ds\\
 = \frac{1}{4}\int_{\mathbb{R}^{d}}  (1-e^{-4\pi  |\xi |}-4\pi |\xi |e^{-4\pi  |\xi |} )  |\widehat{f}  (\xi ) |^{2}d\xi.
\end{equation*}
Therefore
\begin{equation*}
\begin{split}
  \|  s^{c}  (f ) \|  _2^{2}  &= \tau\int_{\mathbb{R}^{d}}\int_{\widetilde{\Gamma}}  \big|\frac{\partial}{\partial\varepsilon}\mathrm{P}_{\varepsilon}  (f )  (s+t ) \big|^{2}\frac{d\varepsilon dt}{\varepsilon^{d-1}}ds \\
&  =   \tau\int_{\mathbb{R}^{d}}\int_{0}^{1}\int_{B  (s,\varepsilon )}  \big|\frac{\partial}{\partial\varepsilon}\mathrm{P}_{\varepsilon}  (f )  (t ) \big|^{2}\frac{d\varepsilon dt}{\varepsilon^{d-1}} ds\\
&   =   c_{d}\, \tau\int_{\mathbb{R}^{d}}\int_{0}^{1}  \big|\frac{\partial}{\partial\varepsilon}\mathrm{P}_{\varepsilon}  (f )  (s ) \big|^{2}\varepsilon d\varepsilon ds\\
&    =   \frac{c_d}{4}\tau \int_{\mathbb{R}^{d}}  (1-e^{-4\pi  |\xi |}-4\pi |\xi |e^{-4\pi  |\xi |} )  |\widehat{f}  (\xi ) |^{2}d\xi,
\end{split}
\end{equation*}
where $c_{d}$ is the volume of the unit ball in $\mathbb{R}^{d}$.
Meanwhile,
$$
  \|  \mathrm{P}*f \|  _2^{2}=\tau\int_{\mathbb{R}^{d}}e^{-4\pi  |\xi |}  |\widehat{f}  (\xi ) |^{2}d\xi.
$$
Then we deduce \eqref{eq: 1} from the equality
$$\frac{4}{c_d }    \|  s^{c}  (f ) \|  _2^{2}+     \|  \mathrm{P}*f \|  _2^{2}=\tau \int_{\mathbb{R}^{d}}  (1- 4\pi |\xi |e^{-4\pi  |\xi |} )  |\widehat{f}  (\xi ) |^{2}d\xi $$
and the fact that $0\leq 4\pi |\xi |e^{-4\pi  |\xi |} \leq \frac 1 e$ for every $\xi\in \mathbb{R}^d$. Passing to adjoint, \eqref{eq: 1} also tells us that $\|  f\|  _{\h_2^r(\R,\M)}\approx \|  f^*\| _2= \|  f\| _2$. Then we have

\begin{equation}\label{rem: h2=L2}
\h_{2}^{c}  (\mathbb{R}^{d},\mathcal{M} )= \h_{2}^{r}  (\mathbb{R}^{d},\mathcal{M} )= L_{2}  (\mathcal{N} )
\end{equation}
with equivalent norms.

Next, if we apply \eqref{eq: Planchel-1} instead of \eqref{eq: Planchel} in the above proof, we get the following polarized version of \eqref{eq: 1},
\begin{equation} \label{eq:polarized}
\begin{split}
 \int_{\mathbb{R}^{d}}f  (s )g^{*}  (s )ds
  &   =  4 \int_{\mathbb{R}^{d}}\int_0^1 \frac{\partial}{\partial\varepsilon}\mathrm{P}_{\varepsilon}  (f )  (s )\frac{\partial}{\partial\varepsilon}\mathrm{P}_{\varepsilon}  (g )^{*}  (s )\varepsilon\, d\varepsilon ds\\
 &\;\;\;\; +  \int_{\mathbb{R}^{d}}\mathrm{P}*f  (s )  (\mathrm{P}*g  (s ) )^{*}ds + 4\pi \int_{\mathbb{R}^{d}}\mathrm{P}*f  (s )  (I(\mathrm{P})*g  (s ) )^{*}ds.\\
&  =  \frac{4}{c_d}\int_{\R}\iint_{\widetilde{\Gamma}}\frac{\partial}{\partial \e}\mathrm{P}_\e(f)(s+t)\frac{\partial}{\partial \e}\mathrm{P}_\e(g)^*(s+t)\frac{dtd\e}{\e^{d-1}}ds \\
 & \;\;\;\;  +  \int_{\mathbb{R}^{d}}\mathrm{P}*f  (s )  (\mathrm{P}*g  (s ) )^{*}ds + 4\pi \int_{\mathbb{R}^{d}}\mathrm{P}*f  (s )  (I(\mathrm{P})*g  (s ) )^{*}ds
 \end{split}
\end{equation}
for nice $f$, $g\in L_{1}  (\mathcal{M};\mathrm{R}_{d}^{c} )+L_{\infty}  (\mathcal{M}; \mathrm R_{d}^{c} )$ (recalling that $I$ is the Riesz potential of order $1$).

\subsection{Operator-valued bmo spaces}

Now we introduce the noncommutative analogue of bmo spaces defined
in \cite{Goldberg1979}. For any cube $Q\subset \mathbb{R}^{d}$, we
denote its center by $c_{Q}$, its side length by $l(Q)$, and its volume by $  |Q |$.
Let $f\in L_{\infty}  (\mathcal{M};\mathrm{R}^c_d)$. The
mean value of $f$ over $Q$ is denoted by $f_{Q}:=\frac{1}{  |Q |}\int_{Q}f  (s )ds$.
We set
\begin{equation}\label{eq: def bmo}
  \|  f \|  _{{\bmo}^{c}}  = \max   \Big\{ \sup_{  |Q |<1}  \big\|  (\frac{1}{  |Q |}\int_{Q}  |f-f_{Q} |^{2}dt )^{\frac{1}{2}} \big\|  _{\mathcal{M}} , \sup_{  |Q |=1}  \big\|  (\int_{Q}  |f |^{2}dt)^{\frac{1}{2}} \big\|  _{\mathcal{M}} \Big\}  .
\end{equation}
Then we define
$$
{\bmo}^{c}  (\mathbb{R}^{d},\mathcal{M} )=  \{ f\in L_{\infty}  (\mathcal{M};\mathrm{R}_{d}^{c} ):\,  \|  f \|  _{{\bmo}^{c}}<\infty \} .
$$
Respectively, define ${\bmo}^{r}  (\mathbb{R}^{d},\mathcal{M} )$
to be the space of all $f\in L^{\infty}  (\mathcal{M};\mathrm{R}_{d}^{r} )$
such that $$  \|  f^{*} \|  _{{\bmo}^{c}}<\infty $$
with the norm $  \|  f \|  _{{\bmo}^{r}}=  \|  f^{*} \|  _{{\bmo}^{c}}.$
And ${\bmo}  (\mathbb{R}^{d},\mathcal{M} )$ is defined
as the intersection of these two spaces
$$
{\bmo}  (\mathbb{R}^{d},\mathcal{M} )={\bmo}^{c}  (\mathbb{R}^{d},\mathcal{M} )\cap{\bmo}^{r}  (\mathbb{R}^{d},\mathcal{M} )
$$
equipped with the norm
$$
  \|  f \|  _{{\bmo}}=\max  \{   \|  f \|  _{{\bmo}^{c}},  \|  f \|  _{{\bmo}^{r}} \} .
$$

\begin{rmk} \label{|Q|=1 eq |Q|geq1}
Let $Q$ be a cube with volume $k^d\leq |Q| < (k+1)^d$ for some positive integer $k$. Then $Q$ can be covered by at most $(k+1)^d$ cubes with volume $1$, say $Q_j$'s. Evidently, 
$$\frac{1}{  |Q |}\int_{Q}  |f  |^{2}dt\leq k^{-d} \int_{Q}  |f  |^{2}dt\leq k^{-d} \sum_{j=1}^{(k+1)^d}\int_{Q_j}  |f  |^{2}dt. $$
Hence,
$$\sup_{  |Q |\geq 1}  \big\|  (\frac{1}{|Q|}\int_{Q}  |f |^{2}dt )^{\frac{1}{2}} \big\|  _{\mathcal{M}}  \leq 2^{\frac{d}{2}}\sup_{  |Q |=1}  \big\|  (\int_{Q}  |f |^{2}dt )^{\frac{1}{2}} \big\|  _{\mathcal{M}}  .$$
Thus, if we replace the second supremum in \eqref{eq: def bmo} over all  cubes of volume one by that over all cubes of volume not less than one, we get an equivalent norm of $\bmo ^c(\R,\M)$.
\end{rmk}

\begin{prop}
\label{Prop: bmo L}
Let $f\in{\bmo}^{c}  (\mathbb{R}^{d},\mathcal{M} ).$
Then
$$
  \|  f \|  _{L_{\infty}  (\mathcal{M};\mathrm{R}_{d}^c )}\lesssim  \|  f \|  _{{\bmo}^{c}}.
$$
 Moreover, ${\bmo}  (\mathbb{R}^{d},\mathcal{M} ),$ ${\bmo}^{c}  (\mathbb{R}^{d},\mathcal{M} )$
and ${\bmo}^{r}  (\mathbb{R}^{d},\mathcal{M} )$ are Banach
spaces.\end{prop}
\begin{proof}
Let $Q_0$ be the cube centered at the origin with side length 1 and  $Q_m= Q_0+ m$ for each $m\in \mathbb{Z}^d$.
For $f\in L_{\infty}  (\mathcal{M};\mathrm{R}_{d}^{c} )$,
\begin{equation*}
\begin{split}
  \|  f \|  _{L_{\infty}  (\mathcal{M};\mathrm{R}_{d}^{c} )}^{2}  &  =  \Big\|  \int_{\mathbb{R}^{d}}\frac{  |f  (t ) |^{2}}{1+  |t |^{d+1}}dt \Big\|  _{\mathcal{M}} \leq \sum_{m\in \mathbb{Z}^d}  \Big\|  \int_{Q_m}\frac{  |f  (t ) |^{2}}{1+  |t |^{d+1}}dt \Big\|  _{\mathcal{M}}\\
&   \lesssim  \sum_{m\in \mathbb{Z}^d}   \Big\|  \frac{1}{1+  |m   |^{d+1}} \int_{Q_m}  |f  (t ) |^{2}dt \Big\|  _{\mathcal{M}}\\
&    \lesssim   \|  f \|  _{{\bmo}^{c}}^{2} \sum_{m\in \mathbb{Z}^d}  \frac{1}{1+  |m   |^{d+1}} \lesssim   \|  f \|  _{{\bmo}^{c}}^{2}.
\end{split}
\end{equation*}
It is then easy to check that $\bmo ^c(\R,\M)$ is a Banach space.
\end{proof}

\begin{prop}\label{BMO-leq-bmo}
We have the inclusion ${\bmo}^{c}  (\mathbb{R}^{d},\mathcal{M} )\subset{\rm BMO}^{c}  (\mathbb{R}^{d},\mathcal{M} )$.
More precisely,
there exists a constant $C$ depending only on the dimension $d$, such that for any $f\in{\bmo}^{c}  (\mathbb{R}^{d},\mathcal{M} ),$
\begin{equation}
  \|  f \|  _{{\rm BMO}^{c}}\leq C  \|  f \|  _{{\bmo}^{c}}.\label{eq:bmo-BMO}
\end{equation}
\end{prop}
\begin{proof}
By virtue of Remark \ref{|Q|=1 eq |Q|geq1}, it suffices to compare the term $  \Big\|  (\frac{1}{  |Q |}\int_{Q}  |f |^{2}dt )^{\frac{1}{2}} \Big\|  _{\mathcal{M}}$
and the term $  \Big\|  (\frac{1}{  |Q |}\int_{Q}  |f-f_{Q} |^{2}dt )^{\frac{1}{2}} \Big\|  _{\mathcal{M}}$ for
$  |Q |\geq1$. By the triangle inequality and \eqref{eq: 2-1}, we have
\begin{equation*}
\begin{split}
  \Big\|    (\frac{1}{  |Q |}\int_{Q}  |f-f_{Q} |^{2}dt )^{\frac{1}{2}} \Big\|  _{\mathcal{M}}  & \leq   \Big\|    (\frac{1}{  |Q |}\int_{Q}  |f |^{2}dt )^{\frac{1}{2}} \Big\|  _{\mathcal{M}}+ \|  f_{Q} \|  _{\mathcal{M}}\\
 & \leq  2  \Big\|    (\frac{1}{  |Q |}\int_{Q}  |f |^{2}dt )^{\frac{1}{2}} \Big\|  _{\mathcal{M}},
\end{split}
\end{equation*}
 which leads immediately to \eqref{eq:bmo-BMO}.
 \end{proof}
Classically, BMO functions are related to Carleson measures (see
\cite{Garnett1981}). A similar relation still holds in the present noncommutative
local setting. We say that an $\mathcal{M}$-valued measure $d\lambda$
on the strip $S=\mathbb{R}^d\times (0,1)$ is a Carleson measure if
$$
N  (\lambda )= \sup_{  |Q |< 1}  \{ \frac{1}{  |Q |}  \big\|  \int_{T  (Q )}d\lambda \big\|  _{\mathcal{M}}:Q\subset\mathbb{R}^{d}\mbox{ cube } \} <\infty,
$$
where $T  (Q )=Q\times  (0, l(Q) ]$.

\begin{lem}
\label{lem: Carleson < bmo}Let $g\in{\bmo}^{c}  (\mathbb{R}^{d},\mathcal{M} )$.
Then $d\lambda_{g}=  |\frac{\partial}{\partial\varepsilon}\mathrm{P}_{\varepsilon}  (g )  (s ) |^{2}\varepsilon \, dsd\varepsilon$
is an $\mathcal{M}$-valued Carleson measure on the strip $S$ and
$$
\max \{N  (\lambda_{g} )^\frac{1}{2},  \,\,\|  \mathrm{P}*g \|  _{L_{\infty}  (\mathcal{N} )}\} \lesssim  \|  g \|  _{{\bmo}^{c}}.
$$
\end{lem}

\begin{proof}
Given a cube $Q$ with $|Q|<1$, denote by $2Q$ the cube with the same center and twice the side length of $Q$. We decompose $g=g_1+g_2+g_3$, where $g_1=(g-g_{2Q})\mathbbm{1}_{2Q}$ and $g_2=(g-g_{2Q})\mathbbm{1}_{\R\backslash 2Q}$. Since $\int \frac{\partial}{\partial\e}\mathrm{P}_\e(s)ds=0$ for any $\e>0$, we have $\frac{\partial}{\partial\e}\mathrm{P}_\e(g)=\frac{\partial}{\partial\e}\mathrm{P}_\e(g_1)+\frac{\partial}{\partial\e}\mathrm{P}_\e(g_2)$. By \eqref{eq: 2-1},
$$
N  (\lambda_{g} )\leq 2(N  (\lambda_{g_1} )+N  (\lambda_{g_2} )).
$$
We first deal with $N  (\lambda_{g_1} )$. By \eqref{eq: Planchel} and \eqref{eq:bmo-BMO}, we have
\begin{equation*}
\begin{split}
\int_{T(Q)}  \big|\frac{\partial}{\partial\varepsilon}\mathrm{P}_{\varepsilon}  (g_1 )  (s ) \big|^{2}\varepsilon dsd\varepsilon & \leq  \int_{\R}\int_0^\infty \big|\frac{\partial}{\partial\varepsilon}\mathrm{P}_{\varepsilon}  (g_1 )  (s ) \big|^{2}\varepsilon dsd\varepsilon\\
&  =  \int_{\R}\int_0^\infty  \Big| \widehat{\frac{\partial}{\partial\varepsilon}\mathrm{P}_{\varepsilon}}(\xi) \Big|^2 |\widehat{g}_1(\xi) |^2\e d\e ds\\
 & \lesssim \int_{\R}  | g_1(s) |^2 ds =\int_{2Q}  | g-g_{2Q} |^2 ds \lesssim |Q| \,\|g\|^2_{\bmo^c}.
\end{split}
\end{equation*}
Thus, $N  (\lambda_{g_1} )\lesssim \|g\|^2_{\bmo^c}$.
  Since $\big|\frac{\partial}{\partial\varepsilon}\mathrm{P}_{\varepsilon}(s)\big|\lesssim \frac{1}{(\e+|s |)^{d+1}}$, applying \eqref{eq: 2-1}, we obtain
\begin{equation*}
  \big| \frac{\partial}{\partial\varepsilon}\mathrm{P}_{\varepsilon}(g_2)(s) \big|^2  \lesssim \frac{1}{\e}\int_{\R\backslash {2Q}}\frac{  | g(t)-g_{2Q} |^2}{(\e+  |s-t |)^{d+1}}dt.
\end{equation*}
The integral on the right hand side of the above inequality can be treated by a standard argument as follows: for any $(s,\e)\in T(Q)$, 
\begin{equation*}
\begin{split}
\int_{\R\backslash {2Q}}\frac{  | g(t)-g_{2Q} |^2}{(\e+  |s-t |)^{d+1}}dt & \lesssim  \int_{\R\backslash {2Q}}\frac{  | g(t)-g_{2Q} |^2}{  |t-c_Q |^{d+1}}dt\\
& \lesssim   \sum_{k\geq 1} \int_{2^{k+1}Q\backslash 2^{k}Q} \frac{  | g(t)-g_{2Q} |^2}{  |t-c_Q |^{d+1}}dt\\
 &   \lesssim \frac{1}{l(Q)}\sum_{k\geq1}2^{-k}\frac{1}{|2^{k+1}Q|}\int_{2^{k+1}Q}{  | g(t)-g_{2Q} |^2}dt\\
  &  \lesssim  \frac{1}{l(Q)}\|g \|^2_{\bmo^c},
\end{split}
\end{equation*}
where $c_Q$ is  the center of $Q$. Then, it follows that  $N  (\lambda_{g_2} )\lesssim \|g\|^2_{\bmo^c}$.

Now we deal with the term $  \|  \mathrm{P}*g  (s ) \|  _{\mathcal{M}}$. Let $Q_{m}=Q_0+m$ be the translate of the cube with
volume one centered at the origin, so $\mathbb{R}^d=\cup_{m\in \mathbb{Z}^d}Q_m$. By (\ref{eq: 2-1}), for any $s\in \mathbb{R}^d$, we have
\begin{equation*}
\begin{split}
  \|  \mathrm{P} *g  (s ) \|  _{\mathcal{M}}   &  =  \big\|  \sum_{m}\int_{Q_{m}}\mathrm{P}   (t )g  (s-t )dt \big \|  _{\mathcal{M}}\\
& \leq   \sum_{m}   \big(\int_{Q_{m}}  | \mathrm{P}   (t ) |^2dt )^{\frac{1}{2}}\cdot \sup_{m\in \mathbb{Z}^d}  \|  (\int_{Q_{m}}  |g  (s-t ) |^2dt \big)^{\frac{1}{2}} \|  _{\mathcal{M}}\\
&   \lesssim   \sup_{  |Q |=1}  \Big\|  (\frac{1}{  |Q |}\int_{Q}  |g(t) |^{2}dt )^\frac{1}{2} \Big\|  _{\mathcal{M}}\\
 &    \lesssim  \|  g \|  _{{\bmo}^{c}}.
\end{split}
\end{equation*}
  Thus, $  \|  \mathrm{P} *g \|  _{L_{\infty}  (\mathcal{N} )}=\sup_{s\in\mathbb{R}^{d}}  \|  \mathrm{P} *g  (s ) \|  _{\mathcal{M}}\lesssim  \|  g \|  _{{\bmo}^{c}}$,
which completes the proof.
\end{proof}

Reexamining the above proof, we find that the facts used to prove $ N  (\lambda_{g} )^\frac{1}{2}\lesssim   \|  g \|  _{{\bmo}^{c}}$ are 
\begin{itemize}
\item $\int_{\R} \e \frac{\partial}{\partial\e}\mathrm{P}_\e(s)ds=0$ for $  \forall\, \e<0$;
\item $\sup_{\xi\in \R}\int_0^\infty  \Big| \widehat{\e\frac{\partial}{\partial\varepsilon}\mathrm{P}_{\varepsilon}}(\xi) \Big|^2 \frac{d\e}{\e} <\infty$;
\item $\big| \e \frac{\partial}{\partial\varepsilon}\mathrm{P}_{\varepsilon}(s)\big|\lesssim \frac{\e}{(\e+|s |)^{d+1}}$.
\end{itemize}
We can easily check that if we replace $\e \frac{\partial}{\partial\e}\mathrm{P}_\e$ above by $\Psi_\e= \frac{1}{\e^{d}} \Psi (\frac{ \cdot}{\e})$, where $\Psi$ is a Schwartz function such that $\widehat{\Psi}(0)=0$, the corresponding  three conditions still hold.
On the other hand, the only fact used for proving the inequality $  \|  \mathrm{P} *g \|  _{L_{\infty}  (\mathcal{N} )} \lesssim  \|  g \|  _{{\bmo}^{c}}$ is that
$$\sum_{m}   (\int_{Q_{m}}  | \mathrm{P}   (t ) |^2dt )^{\frac{1}{2}}<\infty.$$
Recall that $H_2^\sigma(\R)$ denotes the potential Sobolev space, consisting of distributions $f$ such that $J^\sigma(f)\in L_2(\R)$. It is equipped with the norm $\|f\|_{H_2^\sigma(\R)}=\|J^\sigma f\|_{L_2(\R)}$. If $\psi$ is a function on $\mathbb{R}^d$ such that $\widehat{\psi}\in H_2^\sigma(\R)$ for some $\sigma>\frac{d}{2}$, we have
 $$\sum_{m}   \big(\int_{Q_{m}}  | \psi  (s ) |^2dt \big)^{\frac{1}{2}}\lesssim   \big( \sum_m \frac{1}{(1+|m|^2)^\sigma} \big)^\frac{1}{2}   \big(\int_{\R}(1+|s|^2)^\sigma |\psi(s)|^2ds \big)^\frac{1}{2}\lesssim \|  \widehat{\psi} \| _{H_2^\sigma}.
 $$
  Based on the above observation,  we have the following generalization of Lemma \ref{lem: Carleson < bmo}:

\begin{lem}\label{lem: general Carleson}
Let $\psi$ be the (inverse) Fourier transform of a function in $H_2^\sigma(\R)$, and $\Psi$ be a Schwartz function such that $\widehat{\Psi}(0)=0$. If $g\in{\bmo}^{c}  (\mathbb{R}^{d},\mathcal{M} )$, then $d\mu_{g}=   |\Psi_\e*g (s)|^2\frac{d\e ds}{\e}$
is an $\mathcal{M}$-valued Carleson measure on the strip $S$ and
\begin{equation}\label{Carleson < bmo by psi}
\max \{N  (\mu_{g} )^\frac{1}{2}, \,\, \|  \psi*g \|  _{L_{\infty}  (\mathcal{N} )}\} \lesssim  \|  g \|  _{{\bmo}^{c}}.
\end{equation}
In particular,
\begin{equation}\label{Carleson < bmo by potential}
\max \{N  (\mu_{g} )^\frac{1}{2}, \,\, \|  J(\mathrm{P} )*g \|  _{L_{\infty}  (\mathcal{N} )}\} \lesssim  \|  g \|  _{{\bmo}^{c}}.
\end{equation}
\end{lem}

\begin{proof}
\eqref{Carleson < bmo by psi} follows from the above discussion; \eqref{Carleson < bmo by potential} is ensured by \eqref{Carleson < bmo by psi} and the fact that  $(1+|\xi|^2)^{\frac 1 2 } e^{-2\pi |\xi|} \in H_2^\sigma(\R)$, which can be checked by a direct computation.
\end{proof}

\begin{rmk}
We will see in the next section that the converse inequality of \eqref{Carleson < bmo by potential} also holds.
\end{rmk}

%%%%%%%%%%%%%%%%%%%%%%%%%%%%%%%%%
\section{The dual space of $\h_{p}^{c}$ for $1\leq p< 2$}\label{section-dual}
%%%%%%%%%%%%%%%%%%%%%%%%%%%%%%%%%%

In this section, we describe the dual of $\h_p^c(\R,\M)$ for $1\leq p< 2$ as bmo type spaces. We call these spaces ${\bmo} _q^c(\R,\M)$ (with $q$ the conjugate index of $p$). The argument used here is modelled on the one used in \cite{Goldberg1979} when studying the duality between $\mathcal{H}_{p}^c(\R,\M)$ and $\BMO_q^c(\R,\M)$. However, due to the truncation of the square functions,  some highly  non-trivial modifications are needed.

\subsection{Definition of ${\bmo}_q^c$} 
Let $2<q\leq \infty$. We define ${\bmo} _q^c(\R,\M)$ to be the space of all $f\in L_q(\M;\mathrm R^c_d)$ such that
$$
\|f\|_{{\bmo} _q^c}=\Big(\Big\|\underset{\substack{s\in Q\subset \R\\ |Q|<1}}{\sup{} ^+}\frac{1}{|Q|}\int_Q|f(t)-f_Q|^2dt\Big\|_{\frac{q}{2}}^\frac{q}{2}+ \Big\|\underset{\substack{s\in Q\subset \R\\ |Q|=1}}{\sup{} ^+}\frac{1}{|Q|}\int_Q|f(t)|^2dt\Big\|_{\frac q 2 }^\frac{q}{2}\Big)^\frac{1}{q}<\infty.
$$
If $q=\infty$, ${\bmo} _q^c(\R,\M)$ coincides with the space $\bmo^c(\R,\M)$ introduced in the previous section.

Note that the norm $\|\sup _i^+ a_i\|_{\frac{q}{2}}$ is just an intuitive notation since the pointwise supremum does not make any sense in the noncommutative setting. This is the norm of the Banach space $L_{\frac{q}{2}}(\N;\ell_\infty)$; we refer to \cite{Pisier, Junge, Junge-Xu-08} for more information.

If $1\leq p< \infty $ and $(a_i)_{i\in \mathbb{Z}}$ is a sequence of positive elements in $L_p(\N)$, it has been proved by Junge (see \cite{Junge}, Remark 3.7) that
\begin{equation} \label{eq: Lp infty norm}
\|\sup _i{}^+ a_i\|_{p}=\sup \big\{\sum_{i\in \mathbb{Z}}\tau ( a_i b_i):\,b_i\in L_{q}(\N),\,b_i \geq 0, \,\|\sum_{i\in \mathbb{Z}}b_i
\|_{q} \leq 1 \big\}.
\end{equation}
It is also known that a positive sequence $(x_i)_i$ belongs to $L_p(\N; \ell_\infty)$  if and only if
there is an $a\in L_p(\N)$ such that $x_i\leq  a$ for all $i$, and moreover,
$$
\|(x_i)\|_{L_p(\N; \ell_\infty)} =\inf\{ \|a\|_p :\;\; a\in L_p(\N), x_i\leq a,   \forall\, i    \}.
$$
Then we get the following fact (which can be taken as an equivalent  definition): $f\in {\bmo} _q^c(\R,\M)$  if and only if
\begin{equation}\label{eq: a}
\exists\, a\in L_{\frac{q}{2}}(\N) \text{ s.t. } \frac{1}{|Q|}\int_{Q}|f(t)-f_Q|^2dt\leq a(s),\; \forall\,s\in Q \text{ and} \;\forall \,Q \subset \R \text{ with } |Q|<1
\end{equation}
and
\begin{equation}\label{eq: b}
\exists \,b\in L_{\frac{q}{2}}(\N) \text{ s.t. } \frac{1}{|Q|}\int_{Q}|f(t)|^2dt\leq b(s),\; \forall\,s\in Q \text{ and} \;\forall \,Q \subset \R  \text{ with } |Q|=1.
\end{equation}
If this is the case, then
$$
\|f\|_{{\bmo} _q^c}=\inf\big\{\big( \|a\|_{\frac{q}{2}}^{\frac{q}{2}}+\|b\|_{\frac{q}{2}}^{\frac{q}{2}}\big)^{\frac{1}{q}}: a, b  \text{ as in }\eqref{eq: a} \text{ and } \eqref{eq: b} \text{ respectively}\big\}.
$$

In fact, the cubes considered in the definition of ${\bmo} _q^c(\R,\M)$ can be reduced to cubes with dyadic lengths. Let $Q_s^k$ denote the cube centered at $s$ and with side length $2^{-k}$, $k\in \mathbb{Z}$. Set
$$
f_k^\# (s)=\frac{1}{|Q_s^{k}|}\int_{Q_s^{k}}\big| f(t)-f_{Q_s^k}\big|^2dt \quad \text{and}\quad f^\# (s)=\frac{1}{|Q_s^{0}|}\int_{Q_s^{0}}\big| f(t)\big|^2dt.
$$

\begin{lem}\label{bmoq-equi}
If $q> 2$, then
$$
\Big(\|\sup _{k\geq 1}{}^+f_k^\# \|_{\frac{q}{2}}^\frac{q}{2}+ \|f^\# \|_{\frac{q}{2}}^\frac{q}{2}\Big)^\frac{1}{q}
$$
gives an equivalent norm in ${\bmo} _q^c(\R,\M)$.
\end{lem}
\begin{proof}
It is obvious from the definition  that
$$
\|\sup _{k\geq 1}{}^+f_k^\#  \|_{\frac{q}{2}}^\frac{1}{2}\leq \|f \|_{{\bmo} _q^c}\quad \text{and}\quad \|f^\#  \|_{\frac{q}{2}}^\frac{1}{2}\leq \|f \|_{{\bmo} _q^c}.
$$
We notice that for any cube $Q$ with $|Q|<1$ and $s\in Q$, there exists $k\geq -1$ such that $Q\subset Q_s^k$ and $|Q_s^k|\leq 4^d |Q|$. Thus
$$
\frac{1}{4^d}\Big\|\underset{\substack{s\in Q\subset \R\\ |Q|<1}}{\sup{} ^+}\frac{1}{|Q|}\int_Q|f(t)-f_Q|^2dt\Big\|_{\frac{q}{2}}^\frac{1}{2}\lesssim \|\sup _{k\geq -1}{}^+f_k^\#  \|_{\frac{q}{2}}^\frac{1}{2}\lesssim 2^d \|\sup _{k\geq 1}{}^+f_k^\# \|_{\frac{q}{2}}^\frac{1}{2}.
$$
Similarly,
$$
\frac{1}{4^d}\Big\|\underset{\substack{s\in Q\subset \R\\ |Q|=1}}{\sup{} ^+}\frac{1}{|Q|}\int_Q|f(t)|^2dt\Big\|_{\frac{q}{2}}^\frac{1}{2}\leq 2^d \|f^\# \|_{\frac{q}{2}}^\frac{1}{2}.
$$
Thus the lemma is proved.
\end{proof}

From the proofs of  Proposition \ref{Prop: bmo L} and Lemma \ref{lem: Carleson < bmo}, we can easily see that their $q$-analogues still hold in the present setting. We leave the proofs to the reader.

\begin{prop}\label{prop: bmoq Lq}
Let $q> 2$ and $f\in{\bmo}_q^{c}  (\mathbb{R}^{d},\mathcal{M} )$.
Then
$$
  \|  f \|  _{L_q  (\mathcal{M};\mathrm{R}_{d}^c )}\lesssim  \|  f \|  _{{\bmo}_q^{c}}.
$$
\end{prop}

\begin{lem}\label{lem: Caleson<bmo q}
Let $f\in {\bmo} _q^c(\R,\M)$ and assume that the operators $a$ and $b$ satisfy \eqref{eq: a} and \eqref{eq: b} respectively. Then $d\lambda_f$ is a $q$-Carleson measure in the following sense:
$$
\frac{1}{|Q|}\int_{T(Q)}|\frac{\partial}{\partial\varepsilon}\mathrm{P} _{\varepsilon}(f)(t)|^2\e dtd\e\lesssim  a(s),\; \forall\,s\in Q \text{ and } \;\forall \,Q \subset \R \text{ with } |Q|<1.
$$
Moreover, $|\psi *f(s)|^2\lesssim b(s)$ for any $s\in \R$, if $\psi$ is the (inverse) Fourier transform of a function in $H_2^\sigma(\R)$.
\end{lem}

\subsection{A bounded map}

In the sequel, we equip the truncated cone $\widetilde{\Gamma}=\{(s,\e)\in \mathbb{R}_+^{d+1} :   |s|<\e< 1\}$ with the measure $\frac{dtd\e}{\e^{d+1}}$.   For any $1\leq p<\infty$, we will embed $\h_{p}^{c}  (\mathbb{R}^{d},\mathcal{M} )$
into a larger space $L_{p}  \big(\mathcal{N};L_{2}^{c}(\widetilde{\Gamma}) \big) \oplus_p L_{p}  (\mathcal{N} )$. Here $L_{p}  \big(\mathcal{N};L_{2}^{c}(\widetilde{\Gamma}) \big)\oplus_p L_{p}  (\mathcal{N} )$
is the $\ell_p$-direct sum of the Banach spaces $L_{p}  \big(\mathcal{N};L_{2}^{c}(\widetilde{\Gamma}) \big)$
and $L_{p}  (\mathcal{N} )$, equipped with the norm
$$
  \|    (f,g ) \|  =  \Big(  \|  f \|  _{L_{p}  \big(\mathcal{N};L_{2}^{c}(\widetilde{\Gamma})\big)}^p+  \|  g \|  _{L_{p}  (\mathcal{N} )}^p \Big)^{\frac{1}{p}}
$$
 for $f\in L_{p}  \big(\mathcal{N};L_{2}^{c}(\widetilde{\Gamma}) \big)$
and $g\in L_{p}  (\mathcal{N} )$, with the usual modification for $p=\infty$.

\begin{defn}\label{def:Phi&Psi}
\label{def:map}
We define a map $\fE$ from $\h_{p}^{c}  (\mathbb{R}^{d},\mathcal{M} )$ 
to $L_{p}  \big(\mathcal{N};L_{2}^{c}(\widetilde{\Gamma}) \big)\oplus_p L_{p}  (\mathcal{N} )$
by
$$
\fE  (f )  (s,t,\varepsilon )=  \big(\e\,\frac{\partial}{\partial\varepsilon}\mathrm{P} _{\varepsilon}  (f )  (s+t ),\mathrm{P} *f  (s ) \big),
$$
and a map $\fF$ for sufficiently nice $h=  (h',h'' )\in L_{p}  \big(\mathcal{N};L_{2}^{c}(\widetilde{\Gamma}) \big)\oplus_p L_{p}  (\mathcal{N} )$  by
$$
\fF  (h )  (u )=\int_{\mathbb{R}^{d}}  \Big[\frac{4}{c_d}\iint_{\widetilde{\Gamma}}h' (s,t,\varepsilon )\frac{\partial}{\partial\varepsilon}\mathrm{P} _{\varepsilon}  (s+t-u )\frac{dtd\varepsilon}{\e^d} +h''  (s )(\mathrm{P} +4\pi I(\mathrm{P} ))  (s-u ) \Big]ds\,.
$$
\end{defn}

By definition, the map $\fE$ embeds $\h_{p}^{c}  (\mathbb{R}^{d},\mathcal{M} )$ isometrically into 
$L_{p}  \big(\mathcal{N};L_{2}^{c}(\widetilde{\Gamma}) \big)\oplus_p L_{p}  (\mathcal{N} )$. The following results, Theorems \ref{lem: L bmo} and \ref{thm: bdd retraction}  show that by identifying $\h_{p}^{c}  (\mathbb{R}^{d},\mathcal{M} )$ as a subspace of $L_{p}  \big(\mathcal{N};L_{2}^{c}(\widetilde{\Gamma}) \big)\oplus_p L_{p}  (\mathcal{N} )$  via $\fE$, $\h_{p}^{c}  (\mathbb{R}^{d},\mathcal{M} )$ is complemented in $L_{p}  \big(\mathcal{N};L_{2}^{c}(\widetilde{\Gamma}) \big)\oplus_p L_{p}  (\mathcal{N} )$ for every $1<p<\infty$ by virtue of the map $\fF$.

\begin{prop}
Let $1\leq p<\infty$. Then for any nice $f\in L_1(\M; \mathrm{R}_d^c)+L_{\infty}(\M;\mathrm{R}_d^c)$, we have
$$
\fF (\fE (f))=f.
$$
\end{prop}
\begin{proof}
Applying \eqref{eq:polarized}, we get, for any nice function $g$,
\begin{equation*}
\begin{split}
\int_{\R}\fF (\fE (f))(u)g(u)du & = \int_{\R} \Big[\frac{4}{c_d}\iint_{\widetilde{\Gamma}}\frac{\partial}{\partial \e}\mathrm{P} _\e(f)(s+t)\frac{\partial}{\partial \e}\mathrm{P} _\e(s+t-u)\frac{dtd\e}{\e^{d-1}}g(u)du \\
& \;\;\;\;+\mathrm{P} *f(s)\int (\mathrm{P} (s-u)+4\pi I(\mathrm{P})(s-u))g(u)du\Big ] ds\\
& =   \int_{\R}\Big[ \frac{4}{c_d}\iint_{\widetilde{\Gamma}}\frac{\partial}{\partial \e}\mathrm{P} _\e(f)(s+t)\frac{\partial}{\partial \e}\mathrm{P} _\e(g)(s+t)\frac{dtd\e}{\e^{d-1}} \\
 & \;\;\;\;  +\mathrm{P} *f(s) (\mathrm{P} *g+4\pi I(\mathrm{P})*g)(s)\Big ] ds\\
& =  \int_{\R} f(u)g(u)du\,,
\end{split}
\end{equation*}
which completes the proof.
\end{proof}

The following dyadic covering lemma is known. Tao Mei \cite{Mei2007} proved this lemma for the $d$-torus and also for the real line. For the case $\R$ with $d> 1$, we refer the interested readers to \cite{Conde-13,Hytonen-10} for more details.  In the following,  we give a sketch of the way how we choose the dyadic covering.

\begin{lem}\label{lem: dyadic cover}
There exist a constant $C > 0$, depending only on $d$, and $d+1$ dyadic increasing filtrations $\D^i=\{\D^i_j\}_{j\in \mathbb{Z}}$ of $\sigma$-algebras on $\R$ for $0\leq i\leq d$, such that for any cube $Q\subset \R$, there is a cube  $D^i_{m, j} \in \D^i_j$ satisfying $Q\subset D^i_{m, j}$ and $ |D^i_{m, j} | \leq C |Q|$.
\end{lem}
\begin{proof}
 Let $\{\alpha^i\}_{i=0}^d$ be a sequence in the interval $(0,1)$ such that
$$
\min_{i\neq i'}|\alpha^i -\alpha^{i'} |>0.
$$
Then we define
 \begin{equation}
 \alpha_j^i=
\begin{cases}
\alpha^i, \quad j\geq 0,\\
\alpha^i +\frac{1}{3}(2^{-j}-1), \quad j<0 \,\text{ and } -j \text{ even},\\
\alpha^i -\frac{1}{3}(2^{-j}+1), \quad j<0 \,\text{ and } -j \text{ odd}.
\end{cases}
  \end{equation}
The $\sigma$-algebra $\D^i_j$ is generated by the cubes
 $$
D^i_{m,j}=  (\alpha_j^i+  m_{1} 2^{-j}, \alpha_j^i+(m_{1}+1) 2^{-j} ]\times\cdots\times  ( \alpha_j^i+  m_{d} 2^{-j}, \alpha_j^i+(m_{d}+1) 2^{-j} ],
$$
for all $m=  (m_{1},\cdots,m_{d} )\in\mathbb{Z}^{d}$.

For any cube $Q\subset \R$, there exist a constant $C$, depending only on $\{\alpha^i\}_{i=0}^d$ and $d$, and a dyadic cube  $D^i_{m, j}$ such that $Q\subset D^i_{m, j}$ and $ |D^i_{m, j} | \leq C |Q|$.
\end{proof}

To show the boundedness of the map $\fF$, we need the following assertion by Mei, see \cite[Proposition 3.2]{Mei2007}; we include a proof for this lemma, since the one in \cite{Mei2007} is the one dimensional case. Let $1\leq p<\infty$, and $f\in L_p(\N)$ be a positive function. Let $Q$ be a cube centered at the origin, and denote $Q_t = t +Q$. Then we define
$$
f^Q(t)=\frac{1}{|Q|}\int_{Q_t}f(s)ds.
$$

\begin{lem}\label{lem: Mean function}
Let $1\leq p<\infty$ and  let $(f_k)_{k\in\mathbb{Z}}$ be a positive sequence in $L_p(\N)$ and $(Q^k)_{k\in\mathbb{Z}}$ be a sequence of cubes centered at the origin. Then
$$
\|\sum_{k\in\mathbb{Z}}(f_{k})^{Q^k}\|_p \lesssim \|\sum_{k\in\mathbb{Z}}f_{k}\|_p.
$$
\end{lem}

\begin{proof}
Similarly to the proof of \cite[Proposition 3.2]{Mei2007}, we are going to apply \cite[Theorem~0.1]{Junge} for noncommutative martingales. By Lemma \ref{lem: dyadic cover}, we can cover every $Q^k$ by some $D^i_{m', j_k}$, and thus by some $D^i_{m, j_k -1}$, which has twice the side length of $D^i_{m', j_k}$. Moreover, $|D^i_{m, j_k -1}|\leq C |Q^k|$. Obviously, $t+Q^k$ is still covered by $t+D^i_{m, j_k-1}$, but the later is not necessary a dyadic cube in $\mathcal D^i_{j_k-1}$. Let us adjust the translation vector $t=(t_1,... , t_d)$ as follows. Write $Q^k = (-a, a] \times ...\times (-a, a]$ and $D^i_{m, j_k-1}= (b_1, b_2 ] \times...\times (b_1, b_2]$, then either $b_2 -a \geq 2^{-j_k}$ or $-a-b_1\geq  2^{-j_k} $. Without loss of generality, we can assume $b_2 -a \geq 2^{-j_k}$. Now set $\widetilde t  =  (\widetilde t_1,... , \widetilde t_d)$ with $\widetilde t_j$ the largest real number in the set $2^{-j_k} \mathbb Z$ less than $t_j$. Then we can check that $t+Q^k$ is covered by $\widetilde t+D^i_{m, j_k-1}$ and that the later is a dyadic cube. Thus, 
$$(f_{k})^{Q^k}  \leq   C \sum_{0\leq i\leq d} \mathbb{E}(f_k | \D^{i}_{j_k} ),$$
where $\mathbb{E}(\cdot | \D^{i}_{j} )$ denotes the conditional expectation with respect
to $\D^{i}_{j}$. Then the lemma follows from \cite[Theorem~0.1]{Junge}.
\end{proof}

\begin{thm}\label{lem: L bmo}
For $2<p\leq \infty$, the map $\fF$ extends to a bounded map from $L_{p}  \big(\mathcal{N};L_{2}^{c}(\widetilde{\Gamma}) \big)\oplus_p L_{p}  (\mathcal{N} )$
to ${\bmo}_p^{c}  (\mathbb{R}^{d},\mathcal{M} )$.
\end{thm}
\begin{proof}
We have to show that for any $h=  (h',h'' )\in L_p  \big(\mathcal{N};L_{2}^{c}(\widetilde{\Gamma})\big) \oplus_p  L_p  (\mathcal{N} )$,
$$
  \| \fF (h ) \|  _{{\bmo}_p^{c}}\lesssim  \|  h \|  _{L_p  \big(\mathcal{N};L_{2}^{c}(\widetilde{\Gamma})\big)\oplus_p L_{p}  (\mathcal{N} )}.
$$
Fix $h= (h', h'')\in L_p  \big(\mathcal{N};L_{2}^{c}(\widetilde{\Gamma})\big) \oplus_p  L_p  (\mathcal{N} )$
and set $\varphi=\fF  (h )$. We will apply Lemma \ref{bmoq-equi} to estimate the $\bmo_p^c$-norm of $\fF(h)$.
For $v\in \R$ and $k\in \mathbb{N}$, denote by $Q_v^k$  the cube centered at $v$ and with side length $2^{-k}$, then we have $Q_v^k=v+Q_0^k$. We set
$$
h_1'(s,t,\e)=h'(s,t,\e)\mathbbm{1}_{Q_v^{k-1}}(s), \quad h_2'(s,t,\e)=h'(s,t,\e)\mathbbm{1}_{(Q_v^{k-1})^c}(s)
$$
and
$$
\varphi_k^\# (v)=\frac{1}{|Q_v^{k}|}\int_{Q_v^{k}}\big| \varphi(u)-\varphi_{Q_v^k}\big|^2du.
$$
Let $$B^{Q_0^{k}}(v)=\int_{\R}\iint_{\widetilde{\Gamma}} (\frac{\partial}{\partial \e}\mathrm{P} _\e)^{Q_0^{k}}(s,t,v)h_2'(s,t,\e)\frac{dtd\e}{\e^d} ds$$
with $(\frac{\partial}{\partial \e}\mathrm{P} _\e)^{Q_0^{k}}(s,t,v)=\frac{1}{|{Q_v^{k}}|}\int_{Q_v^{k}} \frac{\partial}{\partial \e}\mathrm{P} _\e(s+t-u)du$. Then, we have
\begin{equation*}
\begin{split}
\varphi_k^{\#}(v)  & \lesssim \frac{1}{|{Q_v^{k}}|}\int_{Q_v^{k}} |\varphi(u)-B^{Q_0^{k}}(v) |^2du\\
 &\lesssim \frac{1}{|{Q_v^{k}}|}\int_{Q_v^{k}}\Big| \int_{(Q_v^{k-1})^c}\iint_{\widetilde{\Gamma}} h _2'(s,t,\e)\big[\frac{\partial}{\partial \e}\mathrm{P} _\e(s+t-u)-(\frac{\partial}{\partial \e}\mathrm{P} _\e)^{Q_0^{k}}(s,t,v) \big]\frac{dtd\e}{\e^d} ds \Big|^2 du\\
&\;\;\;\;+  \frac{1}{|{Q_v^{k}}|}\int_{Q_v^{k}}\Big| \int_{Q_v^{k-1}}\iint_{\widetilde{\Gamma}} h _1'(s,t,\e)\frac{\partial}{\partial \e}\mathrm{P} _\e(s+t-u)\frac{dtd\e}{\e^d} ds \Big|^2du\\
&\;\;\;\; + \frac{1}{|{Q_v^{k}}|}\int_{Q_v^{k}}\Big| \int_{\R}h''(s)[\mathrm{P} (s-u)+4\pi I(\mathrm{P})(s-u)]ds\Big|^2 du.
\end{split}
\end{equation*}
When $s\in (Q_v^{k-1})^c$, $u\in Q_v^k$ and $(t,\e) \in \widetilde{\Gamma}$, we have
$|s+t-u|+\e \approx |s-v| +\e $ with uniform constants. Then,
\begin{equation*}
\begin{split}
  &  \iint_{\widetilde{\Gamma}}\Big| \frac{\partial}{\partial \e}\mathrm{P} _\e(s+t-u)-(\frac{\partial}{\partial \e}\mathrm{P} _\e)^{Q_0^{k}}(s,t,v) \Big|^2\frac{dt d\e}{\e^{d-1}}\\
&   \lesssim  \iint_{\widetilde{\Gamma}}\Big(\frac{2^{-k}}{(|s+t-u|+\e)^{d+2}} \Big)^2 \frac{dt d\e}{\e^{d-1}}\lesssim \int_0^1 \int_{B(0,\e)}\frac{2^{-2k}}{(|s-v|^2+\e^2)^{d+2}}  dt\,\frac{ d\e}{\e^{d-1}} \\
&  = c_d \int_0^1 \frac{2^{-2k} \e}{(|s-v|^2+\e^2)^{d+2}} d\e \lesssim \frac{2^{-2k}}{|s-v|^{2d+2}}.
\end{split}
\end{equation*}
Let $(a_k)_{k\in \mathbb{N}}$ be a  positive sequence  such that  $\|\sum_{k\geq 1}a_k\|_{(\frac{p}{2})'}\leq 1$, where $r'$ denotes the conjugate index of $r$.  Let 
\begin{equation*}
\begin{split}
\rm A &= \sum_{k\geq 1}\tau \int_{\R}\int_{(Q_v^{k-1})^c}\frac{ 2^{-2k}}{ |s-v|^{d+1}}ds \cdot \int_{(Q_v^{k-1})^c}\frac{1}{|s-v|^{d+1}}\iint_{\widetilde{\Gamma}}|h _2'(s,t,\e)|^2\frac{dtd\e }{\e^{d+1}}ds \cdot a_k(v)dv\\
\rm B & = \sum_{k\geq 1}\tau \int_{\R}\frac{1}{|{Q_v^{k}}|}\int_{Q_v^{k}}\big| \int_{Q_v^{k-1}}\iint_{\widetilde{\Gamma}} h _1'(s,t,\e)\frac{\partial}{\partial \e}\mathrm{P} _\e(s+t-u)\frac{dtd\e }{\e^{d+1}}ds \big|^2du\cdot a_k(v)dv\\
\rm C & =\sum_{k\geq 1}\tau \int_{\R} \frac{1}{|{Q_v^{k}}|}\int_{Q_v^{k}}\big| \int_{\R} h''(s)[\mathrm{P} (s-u)+4\pi I(\mathrm{P})(s-u)]ds\big|^2 du\cdot a_k(v)dv.
 \end{split}
\end{equation*}
Then, $$\sum_{k\geq 1}\tau \int \varphi_k^{\#}(v)a_k(v)dv\lesssim \rm A+ \rm B+\rm C.$$
First, we estimate the term $\rm A$.  Applying the Fubini theorem and the H\"older inequality, we arrive at
\begin{equation*}
\begin{split}
\rm A &\lesssim   \sum_{k\geq 1}\tau \int_{\R}2^{-k}\int_{(Q_s^{k-1})^c}|v-s|^{-d-1}\iint_{\widetilde{\Gamma}}|h _2'(s,t,\e)|^2 \frac{dtd\e }{\e^{d+1}}ds\, a_k(v)dv \\
& \leq  \Big\| \iint _{\widetilde{\Gamma}}|h _2'(\cdot,t,\e)|^2 \frac{dtd\e}{\e^{d+1}} \Big\|_{\frac{p}{2}}\cdot \Big\|\sum_{k\geq 1}2^{-k} \int_{(Q_s^{k-1})^c}| v-s |^{-d-1}a_k(v)dv\Big\|_{(\frac{p}{2})'}\\
 & \lesssim  \| h' \|_{L_p(\N;L_2^c(\widetilde{\Gamma}))}^2\cdot \Big\|\sum_{k\geq 1}2^{-k} \sum_{j\leq k} \int_{Q_s^{j-2}\backslash Q_s^{j-1}}{2^{(j-1)(d+1)}}a_k(v)dv\Big\|_{(\frac{p}{2})'}.
\end{split}
\end{equation*}
Here and in the context below, $\| \cdot\|_{(\frac{p}{2})'}$ is the norm of $L_{(\frac{p}{2})'}(\N)$ with respect to the variable $s\in \R$.
Now we apply Lemma \ref{lem: Mean function} to estimate the second factor of the last term:
\begin{equation*}
\begin{split}
   \Big\|\sum_{k\geq 1} \sum_{j\leq k}{2^{(j-1)d}} \int_{Q_s^{j-2}\backslash Q_s^{j-1}} 2^{j-k-1}a_k(v)dv\Big\|_{(\frac{p}{2})'} & \lesssim  \Big\|\sum_{j\in \mathbb{Z}} \sum_{\substack{k\geq j\\ k\geq 1}}2^{j-k-1}a_k \Big\|_{(\frac{p}{2})'} \\
&   \lesssim \big\|\sum_{k\geq 1} a_k \big\|_{(\frac{p}{2})'}\leq 1.
\end{split}
\end{equation*}
Then we move to the estimate of $\rm B$:
\begin{equation*}
\begin{split}
\rm B  & \leq  \sum_{k\geq 1}\int_{\R}{2^{kd}}\tau \int_{\R}\Big| \int_{Q_v^{k-1}}\iint_{\widetilde{\Gamma}} h _1'(s,t,\e)a_k^\frac{1}{2}(v)\frac{\partial}{\partial \e}\mathrm{P} _\e(s+t-u)\frac{dtd\e }{\e^{d}}ds\Big|^2dudv \\
& \leq   \sum_{k\geq 1}\int_{\R}{2^{kd}}\sup _{\|f\|_2=1} \Big|\tau \int_{Q_v^{k-1}}\iint_{\widetilde{\Gamma}} h _1'(s,t,\e)a_k^\frac{1}{2}(v)\frac{\partial}{\partial \e}\mathrm{P} _\e(f)^*(s+t)\frac{dtd\e }{\e^{d}}ds\Big|^2dv.
\end{split}
\end{equation*}
Since $\h_2^c(\R,\M)=L_2(\N)$ with equivalent norms,  by the Cauchy-Schwarz inequality and Lemma \ref{lem: Mean function}, we get
\begin{equation*}
\begin{split}
\rm B  & \leq   \sum_{k\geq 1}\int_{\R}{2^{kd}}\tau \int_{Q_v^{k-1}}\iint_{\widetilde{\Gamma}} | h _1'(s,t,\e)|^2 \frac{dtd\e }{\e^{d+1}}ds \, a_k(v)dv\cdot \|f\|_{\h_2^c}\\
&  \lesssim   \sum_{k\geq 1}\tau \int_{\R}\iint_{\widetilde{\Gamma}} | h _1'(s,t,\e)|^2 \frac{dtd\e }{\e^{d+1}}
{2^{kd}}\int_{Q_s^{k-1}} a_k(v)dv ds  \\
& \leq  \| h'\|^2_{L_p\big(\N;L_2^c(\widetilde{\Gamma})\big)} \Big\|\sum_{k\geq 1} {2^{kd}}\int_{Q_s^{k-1}}a_k(v)dv\Big\|_{(\frac{p}{2})'}\\
 &  \leq  2^d \| h'\|^2_{L_p\big(\N;L_2^c (\widetilde{\Gamma})\big)} \big\|\sum_{k\geq 1}a_k\big\|_{(\frac{p}{2})'} \leq  2^d\| h'\|^2_{L_p\big(\N;L_2^c (\widetilde{\Gamma})\big)}.
\end{split}
\end{equation*}
The techniques used to estimate the term $\rm C$ are similar to that of $\rm B$:
\begin{equation*}
\begin{split}
\rm C &= \sum_{k\geq 1} \tau \int_{\R}{2^{kd}}\int_{Q_v^{k-1}}\big| \int_{\R} h''(s)[\mathrm{P} (s-u)+4\pi I(\mathrm{P}) (s-u)]ds\big|^2 a_k(v)dvdu\\
 &  \leq  \Big\|\sum_{k\geq 1}{2^{kd}}\int_{Q_s^{k-1}}a_k(v)dv \Big\|_{(\frac{p}{2})'} \Big\|\big| \int_{\R} h''(s)[\mathrm{P} (s-\cdot)+4\pi I(\mathrm{P}) (s-\cdot)]ds\big|^2\Big\|_{\frac{p}{2}}\\
  & \lesssim  \Big\|\big| \int_{\R} h''(s)[\mathrm{P} (s-\cdot)+4\pi I(\mathrm{P}) (s-\cdot)]ds\big|^2\Big\|_{\frac{p}{2}},
\end{split}
\end{equation*}
Take $f\in L_{p'}(\N)$ with norm one such that 
\begin{equation*}
\Big\|\big| \int_{\R} h''(s)[\mathrm{P} (s-u)+4\pi I(\mathrm{P}) (s-u)]ds\big|^2\Big\|_{\frac{p}{2}} =  \Big| \tau\int_{\R}h''(s)[\mathrm{P} *f(s)+4\pi I(\mathrm{P})*f(s)] ds\Big|^2.
\end{equation*}
 Then 
\begin{equation*}
\begin{split}
  \Big| \tau\int_{\R}h''(s)[\mathrm{P} *f(s)+4\pi I(\mathrm{P})*f(s)] ds\Big|^2 & \leq \| h''\|_p^2 \| \mathrm{P} *f+4\pi I(\mathrm{P})*f\|_{p'}^2\\
&\lesssim   \| h''\|_p^2  \| f\|_{p'}^2 = \| h''\|_p^2 .
\end{split}
\end{equation*}
Combining the estimates of $\rm A, \rm B$ and $\rm C$ with \eqref{eq: Lp infty norm}, we obtain
$$
\|\sup _{k\geq 1}{}^+ \varphi_k^{\#}\|_{\frac{p}{2}}\lesssim   \|  h \|  _{L_p  \big(\mathcal{N};L_{2}^{c}  (\widetilde{\Gamma} ) \big)\oplus_p L_{p}  (\mathcal{N} )}^2.
$$

It remains to establish the $L_{\frac p 2}$-norm of $\varphi^\# (s)=\frac{1}{|Q_s^{0}|}\int_{Q_s^{0}}\big| \varphi(t)\big|^2dt$, which is relatively easy. For any positive operator $a$ such that  $\|a\|_{L_{(\frac{p}{2})'}(\N)}\leq 1$, we have
\begin{equation*}
\begin{split}
  \tau \int \varphi^{\#}(v)a(v)dv & \lesssim     \tau \int_{\R} \int_{Q_v^{0}}\big| \int_{\R}\iint_{\widetilde{\Gamma}} h '(s,t,\e)\frac{\partial}{\partial \e}\mathrm{P} _\e(s+t-u)\frac{dtd\e }{\e^{d+1}}ds \big|^2du\cdot a(v)dv\\
 & \;\;\;\;+ \tau \int_{\R}  \int_{Q_v^{0}}\big| \int_{\R} h''(s)[\mathrm{P} (s-u)+4\pi I(\mathrm{P})(s-u)]ds\big|^2 du\cdot a(v)dv\\
 &   \stackrel{\mathrm{def}}{=}  \rm B'+\rm C'.
\end{split}
\end{equation*}
The terms $\rm B'$ and $\rm C'$ are treated in the same way as $\rm B$ and $\rm C$ respectively. The results are
\begin{equation*}
\begin{split}
{\rm B'} &\leq  \tau \int_{\R}\iint_{\widetilde{\Gamma}} | h '(s,t,\e)|^2 \frac{dtd\e }{\e^{d+1}}
 \int_{Q_s^0} a(v)dv ds \leq  \| h'\|^2_{L_p\big(\N;L_2^c (\widetilde{\Gamma})\big)} \|a\|_{(\frac{p}{2})'}\,,\\
{\rm C'} & \leq   \Big\| \int_{Q_s^{0}}a(v)dv \Big\|_{(\frac{p}{2})'} \Big\|\big| \int_{\R} h''(s)[\mathrm{P} (s-\cdot)+4\pi I(\mathrm{P}) (s-\cdot)]ds\big|^2\Big\|_{\frac{p}{2}}\lesssim \| h''\|_p^2.
\end{split}
\end{equation*}
So we obtain
$$
\| \varphi^{\#}\|_{\frac{p}{2}}\lesssim   \|  h \|  _{L_p  \big(\mathcal{N};L_{2}^{c}  (\widetilde{\Gamma}) \big)\oplus_p L_{p}  (\mathcal{N} )}^2.
$$
Thus, Lemma \ref{bmoq-equi} ensures that
$$
  \|  \fF (h ) \|  _{{\bmo}_p^{c}}\lesssim \|  h \|  _{L_p  \big(\mathcal{N};L_{2}^{c}  (\widetilde{\Gamma} ) \big)\oplus_p L_{p}  (\mathcal{N} )},
$$
which proves the theorem.
\end{proof}

\begin{cor}
\label{cor: L_1 h1}
Let $1\leq p< 2$. For any $f\in L_{p}  \big(\mathcal{M};L_{2}^c(\mathbb{R}^{d},(1+  |t |^{d+1})dt) \big),$
we have
$$
  \|  f \|  _{\h_{p}^{c}}\lesssim  \|  f \|  _{L_{p}  \big(\mathcal{M};L_{2}^c (\mathbb{R}^{d},(1+  |t |^{d+1})dt ) \big)}.
$$
\end{cor}

\begin{proof}
To simplify notation, we denote $L_{2}\big(\mathbb{R}^{d},(1+  |t |^{d+1})dt\big)$
by $\mathrm{W}_{d}$. Let $q$ be the conjugate index of $p$. By duality, we can choose $h=(h',h'')\in L_q  (\mathcal{N};L_{2}^{c} )\oplus_q  L_q  (\mathcal{N} )$ with norm one such that 
\begin{equation*}
\begin{split}
 &\|  s^{c}  (f ) \|  _p+  \|  \mathrm{P} *f \|  _p\\
 &= \Big|\tau\int_{\mathbb{R}^{d}}\iint_{\widetilde{\Gamma}}\frac{\partial}{\partial\varepsilon}\mathrm{P} _{\varepsilon}  (f )  (s+t )h'^{*}  (s,t,\varepsilon )\frac{dtd\e}{\e^d} ds+\tau\int_{\mathbb{R}^{d}}\mathrm{P} *f  (s )h''^{*}  (s )ds \Big|\\
&  =   \big|\tau\int f  (u)\widetilde{\fF} (h )^* (u )du \big|, 
 \end{split}
\end{equation*}
where 
\begin{equation}\label{eq: widetilde Gamma}
\widetilde{\fF}(h ) (u )=\int_{\mathbb{R}^{d}}  \Big[\iint_{\widetilde{\Gamma}}h' (s,t,\varepsilon )\frac{\partial}{\partial\varepsilon}\mathrm{P} _{\varepsilon}  (s+t-u )\frac{dtd\varepsilon}{\e^d} +h''  (s )\mathrm{P} (s-u ) \Big]ds\,.
\end{equation}
Following the proof of Theorem \ref{lem: L bmo}, we can easily check that $\widetilde{\fF}$ is also bounded from $L_{q}  \big(\mathcal{N};L_{2}^{c}(\widetilde{\Gamma}) \big)\oplus_q L_{q}  (\mathcal{N} )$
to ${\bmo}_q^{c}  (\mathbb{R}^{d},\mathcal{M} )$. 
They applying  Proposition \ref{prop: bmoq Lq} and Theorem \ref{lem: L bmo}, we have
\begin{equation*}
\begin{split}
 &   \big|\tau\int f  (s )\widetilde{\fF}(h)^*(s )ds \big|\\
&  \lesssim  \sup_{  \|  \varphi  \|  _{{\bmo} _q^{c}  (\mathbb{R}^{d},\mathcal{M} )}\leq1}  \big|\tau\int f  (s )\varphi^{*}  (s )ds \big|\\
&  \lesssim  \sup_{  \|  \varphi \|  _{L_q  (\mathcal{M};\mathrm{R}_{d}^{c} )}\leq1}  \Big|\tau\int  (1+  |s |^{d+1}) f  (s )\varphi^{*}  (s )\frac{ds}{ 1+  |s |^{d+1}} \Big|\\
 & =  \|    (1+  |s |^{d+1} )f \|  _{L_{p}  (\mathcal{M};{\mathrm{R}}_{d}^{c} )}=\|  f \|  _{L_{p}  (\mathcal{M};\mathrm{W}_{d}^{c} )}.
\end{split}
\end{equation*}
Thus we obtain the desired assertion.
\end{proof}

\subsection{Duality}
Now we are going to present the $\h_p^c$-$\bmo_q^c$ duality for $1\leq p<2$. We begin this subsection by a lemma which will be very useful in the sequel.
\begin{lem}\label{lem: duality}
Let $1\leq p\leq 2$ and $q$ be its conjugate index. For $f\in \h_{p}^{c}  (\mathbb{R}^{d},\mathcal{M} )\cap L_{2}  (\mathcal{N} )$
and $g\in{\bmo}_q^{c}  (\mathbb{R}^{d},\mathcal{M} )$,
\[
  \big|\tau\int_{\mathbb{R}^{d}}f  (s )g^{*}  (s )ds \big|\lesssim  \| f\|_{\h_{p}^{c} } \|  g \|  _{{\bmo}_q^{c}}.
\]
\end{lem}

\begin{proof}
It suffices to prove the lemma for  compactly supported (relative to the variable of $\mathbb{R}^{d}$) $f\in \h_p^c(\R,\M)$. We assume that $f$ is sufficiently nice that all calculations below are legitimate. We need two auxiliary square functions. For $s\in\R$ and $\e\in [0,1]$, we define
\begin{equation}\label{eq:two variables square function P}
s^{c}  (f )  (s,\varepsilon ) =  \Big(\int_{\varepsilon}^{1}\int_{B  (s,r-\frac{\e}{2} )}  \big|\frac{\partial}{\partial r}\mathrm{P}_{r}  (f )  (t ) \big|^{2}\frac{dtdr}{r^{d-1}} \Big)^{\frac{1}{2}},
\end{equation}
\begin{equation}
\overline{s}^{c}  (f )  (s,\varepsilon ) =  \Big(\int_{\varepsilon}^{1}\int_{B  (s,\frac{r}{2} )}  \big|\frac{\partial}{\partial r}\mathrm{P}_{r}  (f )  (t ) \big|^{2}\frac{dtdr}{r^{d-1}} \Big)^{\frac{1}{2}}.
\end{equation}
Both $\overline{s}^{c}  (f )(s,\e)$ and $s^{c}  (f )(s,\e)$ are decreasing in $\e$ and $s^{c}  (f )(s,0)=s^{c}  (f )(s)$. In addition, it is clear that $\overline{s}^{c}  (f )(s,\e)\leq s^{c}  (f )(s,\e)$. Let $(e_i)_{i\in I}$ be an increasing family of $\tau$-finite projections of $\M$ such that $e_i$ converges to $1_\M$ in the strong operator topology. Then we can approximate $s^{c}  (f )  (s,\varepsilon )$ by $s^{c}(e_i f e_i)(s,\varepsilon )$. Thus we can assume that $\tau$ is finite; under this finiteness assumption, for any small $\delta>0$ (which will tend to zero in the end of the proof), consider $s^{c}  (f )  (s,\varepsilon ) +\delta 1_{\M}$ instead of $s^{c}  (f )  (s,\varepsilon )$, we can assume that $s^{c}(f )(s,\varepsilon)$
is invertible in $\mathcal{M}$ for every $(s,\varepsilon )\in S$. 
By \eqref{eq:polarized} and the Fubini theorem, we have
\begin{equation*}
\begin{split}
  \big|  \tau\int f  (s )g^{*}  (s )ds   \big| &   \lesssim     \Big|    \tau \int_{\mathbb{R}^{d}}\int_{0}^{1}\frac{\partial}{\partial\varepsilon}\mathrm{P}_{\varepsilon}  (f )  (s )\frac{\partial}{\partial\varepsilon}\mathrm{P}_{\varepsilon}  (g )^{*}  (s )\varepsilon\, d\varepsilon ds    \Big| \\
 &\quad \, +    \Big|   \tau\int_{\mathbb{R}^{d}}\mathrm{P}*f  (s )  (\mathrm{P}*g  (s ) )^{*}ds +    \tau\int_{\mathbb{R}^{d}}\mathrm{P}*f  (s )  (I(\mathrm{P})*g  (s ) )^{*}ds    \Big| \\
 &  = \Big|  \frac{2^d}{c_d}\tau\int_{\mathbb{R}^{d}}\int_{0}^{1}\int_{B(s,\frac{\e}{2})}\frac{\partial}{\partial\varepsilon}\mathrm{P}_{\varepsilon}  (f )  (t )\frac{\partial}{\partial\varepsilon}\mathrm{P}_{\varepsilon}  (g )^{*}(t)\frac{d\varepsilon dt}{\e^{d-1}}ds\Big|  \\
  & \;\;\;\; +  \Big| \tau\int_{\mathbb{R}^{d}}\mathrm{P}*f  (s )  (\mathrm{P}*g  (s ) )^{*}ds +   \tau\int_{\mathbb{R}^{d}}\mathrm{P}*f  (s )  (I(\mathrm{P})*g  (s ) )^{*}ds\Big| .
 \end{split}
\end{equation*}
Then, 
 \begin{equation*}
\begin{split}
&  \big|  \tau\int f  (s )g^{*}  (s )ds   \big| \\
&\lesssim 
    \Big|  \frac{2^d}{c_d}\tau\int_{\mathbb{R}^{d}}\int_{0}^{1}\int_{B(s,\frac{\e}{2})}\frac{\partial}{\partial\varepsilon}\mathrm{P}_{\varepsilon}  (f )  (t )s^{c}  (f )  (s,\varepsilon )^{\frac{p-2}{2}}s^{c}  (f )  (s,\varepsilon )^{\frac{2-p}{2}}\frac{\partial}{\partial\varepsilon}\mathrm{P}_{\varepsilon}  (g )^{*}(t)\frac{d\varepsilon dt}{\e^{d-1}}ds \Big|  \\
 &\;\;\;\; +  \Big(\,  \Big| \tau\int_{\mathbb{R}^{d}}\mathrm{P}*f  (s )  (\mathrm{P}*g  (s ) )^{*}ds \Big| +  \Big|  \tau\int_{\mathbb{R}^{d}}\mathrm{P}*f  (s )  (I(\mathrm{P})*g  (s ) )^{*}ds \Big| \, \Big)\\
 &  \stackrel{\mathrm{def}}{=} \, \mathrm{I}+\mathrm{II}.
 \end{split}
\end{equation*}
The term $\mathrm{II}$ is easy to deal with. By the H\"older inequality and \eqref{Carleson < bmo by potential}, we get
\begin{equation*}
\mathrm{II}  \leq   \|  \mathrm{P}*f \| _p  \|  \mathrm{P}*g \|  _q+  \|  \mathrm{P}*f \|  _p  \|  I(\mathrm{P})*g \|  _q.
\end{equation*}
Then by \cite[Proposition V.3 and Lemma V.3.2]{Stein1970} we have
$$  \|  \mathrm{P}*g \|_q \lesssim   \|  J(\mathrm{P})*g \|  _q,\;\;\mbox{and}\;\;   \|  I(\mathrm{P})*g \|  _q \lesssim   \|  J(\mathrm{P})*g \|  _q.$$
Hence, by Lemma \ref{lem: Caleson<bmo q},
$$\mathrm{II} \lesssim   \|  g \|  _{{\bmo}_q^{c}}  \|  f \|  _{\h_{p}^{c}}.$$

Now we estimate the term $\mathrm{I}$. By the Cauchy-Schwarz inequality
\begin{equation*}
\begin{split}
\frac{c_d^2}{4^d}\, \mathrm{I^2} & \leq  \tau\int_{\mathbb{R}^{d}}\int_{0}^{1}\Big(\int_{B(s,\frac{\e}{2})}  |\frac{\partial}{\partial\varepsilon}\mathrm{P}_{\varepsilon}  (f )  (t ) |^{2} \frac{dt}{\e^{d-1}} \Big)s^{c}  (f )  (s,\varepsilon )^{p-2} d\varepsilon ds\\
&\;\;\;\;\cdot \, \tau\int_{\mathbb{R}^{d}}\int_{0}^{1}\Big(\int_{B(s,\frac{\e}{2})}  |\frac{\partial}{\partial\varepsilon}\mathrm{P}_{\varepsilon}  (g)(t ) |^{2} \frac{dt}{\e^{d-1}} \Big) s^{c}  (f )  (s,\varepsilon )^{2-p}d\varepsilon ds\\
&  \stackrel{\mathrm{def}}{=}  A\cdot B.
\end{split}
\end{equation*}
Note here that $s^{c}  (f )(s,\e)$ is the function of two variables
defined by \eqref{eq:two variables square function P}, which is differentiable in the w$^*$-sense.
We first deal with $A$. Using $\overline{s}^{c}  (f )(s,\e)\leq s^{c}  (f )(s,\e)$, we have
\begin{equation*}
\begin{split}
A & \leq  \tau\int_{\mathbb{R}^{d}}\int_{0}^{1}\int_{B(s,\frac{\e}{2})}  |\frac{\partial}{\partial\varepsilon}\mathrm{P}_{\varepsilon}  (f )(t ) |^{2} \overline{s}^{c}(f )(s,\varepsilon )^{p-2}\frac{d\varepsilon dt}{\e^{d-1}}ds\\
& = -\tau\int_{\mathbb{R}^{d}}\int_{0}^{1}\big(\frac{\partial }{\partial\varepsilon}\overline{s}^{c}  (f )(s,\e)^{2}\big)\overline{s}^{c} (f )(s,\e)^{p-2}d\e ds\\
&  =  -2 \tau\int_{\mathbb{R}^{d}}\int_{0}^{1}\overline{s}^{c}  (f )  (s,\varepsilon )^{p-1}\frac{\partial }{\partial\varepsilon}\overline{s}^{c}  (f )(s,\e)d\varepsilon ds.
 \end{split}
\end{equation*}
Since $1\leq p< 2$ and $\overline{s}^{c}  (f )(s,\e)$ is decreasing in $\e$, $\overline{s}^{c}  (f )(s,\e)^{p-1}\leq \overline{s}^{c}  (f )(s,0)^{p-1}$.  At the same time, $\frac{\partial }{\partial\varepsilon}\overline{s}^{c}(f)(s,\e)\leq 0$. Therefore,
  \begin{equation*}
  \begin{split}
A & \lesssim - \tau\int_{\mathbb{R}^{d}}\overline{s}^{c}  (f )  (s,0 )^{p-1}\int_0^1\frac{\partial }{\partial\varepsilon}\overline{s}(f)^{c}(s,\e)d\e ds\\
 & \lesssim \tau\int_{\mathbb{R}^{d}}s^{c}  (f )  (s,0 )^{p}ds=  \|  f \|  _{\h_p^c}^p.\\
\end{split}
\end{equation*}

The estimate of $B$ is harder. For any  $ j\in\mathbb{N}$, we need to create a square net partition
in $\R$ as follows:
$$
Q_{m,j}=  (\frac{1}{\sqrt{d}}  (m_{1}-1 )2^{-j},\frac{1}{\sqrt{d} }m_{1}2^{-j} ]\times\cdots\times  (\frac{1}{\sqrt{d}}  (m_{d}-1 )2^{-j},\frac{1}{\sqrt{d}}m_{d}2^{-j} ]
$$
 with $m=  (m_{1},\cdots,m_{d} )\in\mathbb{Z}^{d}$.
 Let $c_{m,j}$ denote the center of $Q_{m,j}$. Define
\begin{equation}\label{eq: square function}
{\mathbb{S}^{c}}  (f )  (s,j )  =  \Big(\int_{2^{-j}}^1\int_{B  (c_{m,j},r )}  |\frac{\partial}{\partial r}\mathrm{P}_{r}  (f )  (t ) |^{2}\frac{dtdr}{r^{d-1}} \Big)^{\frac{1}{2}}\quad \text{if } s\in Q_{m,j}.
\end{equation}
For any $s\in\mathbb{R}^{d}$ and $ k\in\mathbb{N}_0$  ($ \mathbb{N}_0$ being the set of nonnegative integers),  we define
\begin{equation*}
d  (s,k )  =  \mathbb{S}^{c}  (f )  (s,k )^{2-p}-\mathbb{S}^{c}  (f )  (s,k-1 )^{2-p}.
\end{equation*}
Since $B(s,r-\frac{\e}{2})\subset B(c_{m,j},r)$ whenever $s\in Q_{m,j}$ and $\e\geq 2^{-j}$, we have
$$s^{c}  (f )  (s,\varepsilon ) \leq \mathbb{S}^{c}  (f )  (s,j ),\,  \forall\, s\in Q_{m,j},\varepsilon\geq 2^{-j}.
$$
It is clear that $\mathbb{S}^{c}  (f )(s,j)$ is increasing in $j$, so $d(s,k)\geq 0$.  At the same time, $d(s,k)$ is constant on $Q_{m,k}$ and $\sum_{k\leq j}d(s,k)=\mathbb{S}^{c}  (f )(s,j)^{2-p}$.
Therefore,
\begin{equation*}
\begin{split}
B  &\lesssim  \tau \sum_{m\in\Z}\sum_{j\geq 1}\int_{Q_{m,j}}\int_{2^{-j}}^{2^{-j+1}}\Big(\int_{B(s,\frac{\e}{2})}  |\frac{\partial}{\partial\varepsilon}\mathrm{P}_{\e}(g)(t) |^{2}\frac{dt}{\e^{d-1}}\Big)  \mathbb{S}^{c}  (f )  (s,j )^{2-p} d\varepsilon ds\\
&  =   \tau\int_{\mathbb{R}^{d}}\sum_{j\geq 1}\mathbb{S}^{c}  (f )  (s,j )^{2-p}\int_{2^{-j}}^{2^{-j+1}}\Big(\int_{B(s,\frac{\e}{2})}  |\frac{\partial}{\partial\varepsilon}\mathrm{P}_{\e}(g)(t) |^{2}\frac{dt}{\e^{d-1}}\Big) d\varepsilon ds\\
&   =  \tau\int_{\mathbb{R}^{d}}\sum_{j\geq 1}\sum_{1\leq k\leq j}d(s,k)\int_{2^{-j}}^{2^{-j+1}}\Big(\int_{B(s,\frac{\e}{2})}  |\frac{\partial}{\partial\varepsilon}\mathrm{P}_{\e}(g)(t) |^{2}\frac{dt}{\e^{d-1}}\Big) d\varepsilon ds\\
&  = \tau\int_{\mathbb{R}^{d}}\sum_{k\geq 1}d(s,k)\sum_{j\geq k}\int_{2^{-j}}^{2^{-j+1}}\Big(\int_{B(s,\frac{\e}{2})}  |\frac{\partial}{\partial\varepsilon}\mathrm{P}_{\e}(g)(t) |^{2}\frac{dt}{\e^{d-1}}\Big) d\varepsilon ds\\
&  =  \tau\sum_{m}\sum_{k\geq 1}d(s,k)\int_{Q_{m,k}}\int_0^{2^{-k+1}}\Big(\int_{B(s,\frac{\e}{2})}  |\frac{\partial}{\partial\varepsilon}\mathrm{P}_{\e}(g)(t) |^{2}\frac{dt}{\e^{d-1}}\Big)  d\varepsilon ds\,.
\end{split}
\end{equation*}
Since $g\in {\bmo} _q^c$, Lemma \ref{lem: Caleson<bmo q} ensures the existence of a positive operator $a\in L_{\frac{q}{2}}(\N)$ such that $\|a\|_{\frac{q}{2}}\lesssim \|g\|_{{\bmo} _q^c}^2$ and
$$
\frac{1}{|Q|}\int_{T(Q)}  |\frac{\partial}{\partial\varepsilon}\mathrm{P}_{\e}(g)(t) |^{2}\e dtd\e \leq a(s)\, \text{ and for }s\in Q  \text{ and for all cubes }Q \text{ with }|Q|<1.
$$
Let $\widetilde{Q}_{m,k}$ be the cube concentric with $Q_{m,k}$ and having side length $2^{-k+1}$.
By the Fubini theorem and Lemma \ref{lem: Carleson < bmo}, we have
\begin{equation*}
\begin{split}
\int_{Q_{m,k}}\int_0^{2^{-k+1}}\Big(\int_{B(s,\frac{\e}{2})}  |\frac{\partial}{\partial\varepsilon}\mathrm{P}_{\e}(g)(t) |^{2}\frac{dt}{\e^{d-1}}\Big)  d\varepsilon ds
&\leq  2^d \int_{\widetilde{Q}_{m,k}}\int_0^{2^{-k+1}}  |\frac{\partial}{\partial\varepsilon}\mathrm{P}_{\varepsilon}  (g )  (s ) |^{2}\varepsilon d\varepsilon ds\\
&=  2^d \int_{T(\widetilde{Q}_{m,k})}  |\frac{\partial}{\partial\varepsilon}\mathrm{P}_{\varepsilon}  (g )  (s ) |^{2}\varepsilon d\varepsilon ds\\
& \lesssim \int_{Q_{m,k}}a(s)ds.
\end{split}
\end{equation*}
Then we deduce
\begin{equation*}
\begin{split}
B & \lesssim   \tau \sum_m \sum_{k\geq 1}\int_{Q_{m,k}}d  (s,k )a(s)ds \\
& =\tau \int_{\R}\sum_{k\geq 1}d(s,k)a(s)ds\\
&  =   \tau\int_{\mathbb{R}^{d}}\mathbb{S}^{c}  (f )  (s,+\infty  )^{2-p}a(s)ds\\
& =   \tau \int_{\R} S^c(f)(s)^{2-p}a(s)ds\leq   \|  S^c(f)\|_{p}^{2-p}\| a \|_{\frac{q}{2}}\\
&  \leq   \|  f \|  _{\h_{p}^{c}}^{2-p}\| a \|_{\frac{q}{2}}\lesssim   \|  f \|  _{\h_{p}^{c}}^{2-p}\| g \|_{{\bmo} _q^c}^2.
\end{split}
\end{equation*}
Combining the estimates of $A$, $B$  and $\rm II$, we complete the proof.

\end{proof}

The following is the main theorem of  this section.
\begin{thm}\label{thm: duality P}
Let $1\leq p <2$ and $q$ be its conjugate index.
We have $  \h_{p}^{c}  (\mathbb{R}^{d},\mathcal{M} )^*={\bmo}_q^{c}  (\mathbb{R}^{d},\mathcal{M} )$
with equivalent norms. More precisely, every $g\in\bmo_q^{c}  (\mathbb{R}^{d},\mathcal{M} )$
defines a continuous linear functional on $\h_{p}^{c}  (\mathbb{R}^{d},\mathcal{M} )$
by
$$
\ell_g  (f )=\tau\int f  (s )g^{*}  (s )ds,\,  \forall\, f\in L_{p}  (\mathcal{M};\mathrm{W}_{d}^{c} ).
$$
 Conversely, every $\ell \in   \h_{p}^{c}  (\R,\M )^*$
can be written as above and is associated to some $g\in{\bmo}_q^{c}  (\mathbb{R}^{d},\mathcal{M} )$ with
$$
 \| \ell \|  _{  (\h_{p}^{c})^*}\approx\|  g \|  _{{\bmo}_q^{c}}.
$$

\end{thm}
\begin{proof}
First, by Lemma \ref{lem: duality}, we get
\begin{equation}\label{dual-ineq}
  | \ell_{g}  (f ) |\lesssim  \|  g \|  _{{\bmo}_q^{c}}  \|  f \|  _{\h_{p}^{c}}
\end{equation}

Now we prove the converse. Suppose that $\ell \in \h_{p}^{c}  (\R,\M )^*$.
By the Hahn-Banach theorem, $\ell$ extends to a continuous functional on
$L_{p}  \big(\mathcal{N};L_{2}^{c}(\widetilde{\Gamma}) \big)\oplus_p L_{p}  (\mathcal{N} )$
with the same norm. Thus, there exists $h=(h',h'')\in L_q  \big(\mathcal{N};L_{2}^{c}  (\widetilde{\Gamma}) \big)\oplus_q  L_q  (\mathcal{N} )$
such that
 $$  \| h \|  _{L_q  \big(\mathcal{N};L_{2}^{c}  (\widetilde{\Gamma} ) \big)\oplus_q L_q  (\mathcal{N} )}=  \|  \ell \|  _{ (\h_{p}^{c})^*}$$
 and 
\begin{equation*}
\begin{split}
\ell  (f ) & = 
\tau\int_{\mathbb{R}^{d}}\iint_{\widetilde{\Gamma}}\frac{\partial}{\partial\varepsilon}\mathrm{P} _{\varepsilon}  (f )  (s+t )h'^{*}  (s,t,\varepsilon )\frac{dtd\e}{\e^d} ds+\tau\int_{\mathbb{R}^{d}}\mathrm{P} *f  (s )h''^{*}  (s )ds,\\
&  =  \tau\int_{\mathbb{R}^{d}}f (u)\widetilde{\fF}(h)^*(u)du,
\end{split}
\end{equation*}
 where $\widetilde{\fF}$ is the map defined in \eqref{eq: widetilde Gamma}.

 Let $g=\widetilde{\fF}(h)$. Following the proof of Theorem \ref{lem: L bmo}, we have
$$\|g\|_{{\bmo}_q^{c}}\lesssim \|\ell\|_{  (\h_{p}^{c})^*}$$
 and
\begin{equation*}
\ell (f )  = \tau\int_{\mathbb{R}^{d}}f  (s )g^{*}  (s )ds,\,  \forall\, f\in L_p\big(\M;\mathrm{W}^c_{d}\big).
\end{equation*}
Thus, we have accomplished the proof of the theorem. 
\end{proof}

The following corollary gives an equivalent norm of the space $\bmo_q^c$. Note that it is a strengthening of the one-sided estimates in Lemmas \ref{lem: Carleson < bmo} and \ref{lem: general Carleson}.

\begin{cor}
\label{cor: Carleson=Dbmo}
Let $2<q\leq \infty$. Then $g\in{\bmo}_q^{c}  (\mathbb{R}^{d},\mathcal{M} )$
if and only if $d\lambda_{g}=  |\frac{\partial}{\partial\varepsilon}\mathrm{P}_{\varepsilon}  (g )  (s ) |^{2}\varepsilon dsd\varepsilon$
is an $\mathcal{M}$-valued Carleson $q$-measure on $S$ and $   \|  J( \mathrm{P})*g \|  _q<\infty $. Furthermore,
$$
  \|  g \|  _{{\bmo}_q^{c}}\approx \Big\|\underset{\substack{s\in Q\subset \R\\ |Q|<1}}{\sup{} ^+}\frac{1}{|Q|}\int_{T(Q)}|\frac{\partial}{\partial\varepsilon}\mathrm{P}_{\varepsilon}(g)(t)|^2\e dtd\e \Big\|_{\frac{q}{2}}^{\frac{1}{2}}+   \|  J( \mathrm{P})*g \|  _q .
$$
\end{cor}
\begin{proof}
From the proof of Lemma \ref{lem: duality}, we can see that if $d\lambda_{g}=  |\frac{\partial}{\partial\varepsilon}\mathrm{P}_{\varepsilon}  (g )  (s ) |^{2}\varepsilon dsd\varepsilon$
is an $\mathcal{M}$-valued Carleson $q$-measure on $S$ and $J(\mathrm{P})*g\in L_q(\N)$, then $g$ defines
a continuous functional on $\h_p^c(\R,\M)$:
$$
\ell (f )=\tau\int_{\mathbb{R}^{d}}f  (s )g^{*}  (s )ds,
$$
 and
 $$  \|  \ell \|  _{  (\h_{p}^{c} )^{*}}\lesssim \Big\|\underset{\substack{s\in Q\subset \R\\ |Q|<1}}{\sup{} ^+}\frac{1}{|Q|}\int_{T(Q)}|\frac{\partial}{\partial\varepsilon}\mathrm{P}_{\varepsilon}(g)(t)|^2\e dtd\e \Big\|_{\frac{q}{2}}^{\frac{1}{2}}+  \|  J( \mathrm{P})*g \|  _q .$$
According to Theorem \ref{thm: duality P}, there exists a function $g^{\prime}\in{\bmo}_q^{c}  (\mathbb{R}^{d},\mathcal{M} )$
such that 
$$  \|  g^{\prime} \|  _{{\bmo}_q^{c}}\lesssim \Big\|\underset{\substack{s\in Q\subset \R\\ |Q|<1}}{\sup{} ^+}\frac{1}{|Q|}\int_{T(Q)}|\frac{\partial}{\partial\varepsilon}\mathrm{P}_{\varepsilon}(g)(t)|^2\e dtd\e \Big\|_{\frac{q}{2}}^{\frac{1}{2}}+  \|  J( \mathrm{P})*g \|  _q $$
and that
$$
\tau\int_{\mathbb{R}^{d}}f  (s )g^{*}  (s )ds=\tau\int_{\mathbb{R}^{d}}f  (s )g^{\prime*}  (s )ds,
$$
for any $f\in \h_p^c(\R,\M)$. Thus, $g=g^{\prime}$
with
$$  \|  g \|  _{{\bmo}_q^{c}}\lesssim \Big\|\underset{\substack{s\in Q\subset \R\\ |Q|<1}}{\sup{} ^+}\frac{1}{|Q|}\int_{T(Q)}|\frac{\partial}{\partial\varepsilon}\mathrm{P}_{\varepsilon}(g)(t)|^2\e dtd\e \Big\|_{\frac{q}{2}}^{\frac{1}{2}}+  \|  J( \mathrm{P})*g \|  _q .$$
The inverse inequality is already contained in Lemmas \ref{lem: Carleson < bmo} and \ref{lem: general Carleson}. We obtain the desired assertion.
\end{proof}

\subsection{The equivalence $\h_q=\bmo_q$}
We now show that $\h_q^c(\R,\M)=\bmo_q^c(\R,\M)$ for $2<q<\infty$. Thus according to the duality obtained in the last subsection, the dual of $\h_p^c(\R,\M)$ agrees with $\h_q^c(\R,\M)$ when $1<p<2$. 
Let us begin with two lemmas concerning the comparison of $s^c(f)$ and $g^c(f)$. We require an auxiliary truncated square function.
For $s\in \R$ and $\e\in [0,\frac{2}{3}]$, we define:
\begin{equation}\label{eq: two variables square function}
\widetilde g^{c}  (f )  (s,\varepsilon ) =  \Big(\int_{\varepsilon}^{\frac{2}{3}}  |\frac{\partial}{\partial r}\mathrm{P}_{r}  (f )  (s ) |^{2}rdr \Big)^{\frac{1}{2}}.
\end{equation}

\begin{lem}
\label{lem:poison g Lusin compare}
We have
$$
\widetilde g^{c}  (f )  (s,\varepsilon )\lesssim s^{c}  (f )  (s,\frac{\varepsilon}{2} ),
$$
where the relevant constant depends only on the dimension $d$.
\end{lem}

\begin{proof}
By translation, it suffices to prove this inequality for $s=0$. Given $\e\in [0,{\frac{2}{3}}]$, for any $r$ such that $\e\leq r\leq {\frac{2}{3}}$, let us denote the ball centered at $(0,r)$ and tangent to the boundary of the cone $\{(t,u)\in \mathbb{R}^{d+1}_+: |t|<\frac{r-\frac{\e}{2}}{r}u\}$ by $\widetilde{B}_r $. We notice that the radius of $\widetilde{B}_r $ is greater than or equal to $\frac{r}{\sqrt{5}}$. By the harmonicity
of $\frac{\partial}{\partial r}\mathrm{P}_{r}  (f )$, we have
$$
\frac{\partial}{\partial r}\mathrm{P}_{r}  (f )  (0 )=\frac{1}{|\widetilde{B}_r |}\int_{\widetilde{B}_r }\frac{\partial}{\partial u}\mathrm{P}_{u}  (f )  (t )dt.
$$
Then by \eqref{eq: 2-1}, we arrive at
$$
  |\frac{\partial}{\partial r}\mathrm{P}_{r}  (f )  (0 ) |^{2}\leq\frac{\sqrt{5}^{d+1}}{c_{d+1}r^{d+1}}\int_{\widetilde{B}_r }  |\frac{\partial}{\partial u}\mathrm{P}_{u}  (f )  (t ) |^{2}dt,
$$
where $c_{d+1}$ is the volume of the unit ball of $\mathbb{R}^{d+1}$. Integrating the above inequality, we get
\begin{equation}\label{eq: g and s}
\int_{\varepsilon}^{\frac{2}{3}} |\frac{\partial}{\partial r}\mathrm{P}_{r}  (f )  (0 ) |^{2}rdr\leq\int_{\varepsilon}^{\frac{2}{3}}\frac{\sqrt{5}^{d+1}}{c_{d+1} r^{d}}\int_{\widetilde{B}_r }  |\frac{\partial}{\partial u}\mathrm{P}_{u}  (f )  (t ) |^{2}dtdudr.
\end{equation}
Since $(t,u)\in \widetilde{B}_r $ implies  $\frac{\sqrt{5}}{\sqrt{5}+1}u\leq r\leq \frac{\sqrt{5}}{\sqrt{5}-1}u$ and $\frac{\e}{2}\leq u\leq 1$,  the right hand side of \eqref{eq: g and s} can be majorized by
$$
\frac{\sqrt{5}^{d+1}}{c_{d+1} }\int_\frac{\e}{2}^1\int_{\widetilde{B}_r }  |\frac{\partial}{\partial u}\mathrm{P}_{u}  (f )  (t ) |^{2}\int_\frac{u}{2}^{2u}\frac{1}{r^d}drdtdu\leq C  |s^{c}  (f )(0,\frac{\e}{2}) |^2,
$$
where $C$ is a constant depending only on $d$. Therefore, $\widetilde g^{c}  (f )  (0,\varepsilon )\lesssim s^{c}  (f )  (0,\frac{\varepsilon}{2} )$.
\end{proof}

\begin{lem}\label{lem: s equi g}
Let $1\leq p<\infty$. Then for any $f\in \h_{p}^{c}  (\mathbb{R}^{d},\mathcal{M} )$, we have
$$
\|  s^{c}  (f )\| _p+\| \mathrm{P}*f\|_p \lesssim \| g^{c}  (f )\| _p+\| \mathrm{P}*f\|_p .
$$
\end{lem}

\begin{proof}

 We first deal with the case when $1\leq p<2$. Let $g$ be a function in ${\bmo} _q^c(\R,\M)$ ($q$ is the conjugate index of $p$). Following a similar calculation as \eqref{eq:polarized}, we can easily check that 
\begin{equation*}
\begin{split}
  \tau \int_{\mathbb{R}^{d}}f  (s )g^{*}  (s )ds
& =   4 \tau \int_{\mathbb{R}^{d}}\int_{0}^{\frac{2}{3}}\frac{\partial}{\partial\varepsilon}\mathrm{P}_{\varepsilon}  (f )  (s )\frac{\partial}{\partial\varepsilon}\mathrm{P}_{\varepsilon}  (g )^{*}  (s )\varepsilon d\varepsilon ds\\
& \;\;\;\; +\Big(\tau \int_{\mathbb{R}^{d}}\mathrm{P}*f  (s )  (\mathrm{P}_{\frac{1}{3}}*g  (s ) )^{*}ds+\frac{8\pi}{3}\tau \int_{\mathbb{R}^{d}}\mathrm{P}*f  (s )  (I(\mathrm{P}_{\frac{1}{3}})*g  (s ) )^{*}ds\Big)\\
&   \stackrel{\mathrm{def}}{=}   \rm I+\rm II.
\end{split}
\end{equation*}
The term $\rm II$ can be treated in the same way as in the proof of Lemma \ref{lem: duality}:
$$
{\rm II} \lesssim \| \mathrm{P}*f\|_p\cdot \| J ({\mathrm{P}}_{\frac{1}{3}})*f \|_p.
$$
Applying Lemma \ref{lem: Caleson<bmo q},  we have 
$$
{\rm II} \lesssim \| \mathrm{P}*f\|_p\cdot \| g\|_{\bmo _q^c}.
$$
 Concerning the term $\rm I$, we have
\begin{equation*}
\begin{split}
|\rm I |^{2}  & \lesssim  \tau\int_{\mathbb{R}^{d}}\int_{0}^{\frac{2}{3}}  |\frac{\partial}{\partial\varepsilon}\mathrm{P}_{\varepsilon}  (f )  (s ) |^{2} \widetilde g^{c}  (f )  (s,\varepsilon )^{p-2}\e d\varepsilon ds \cdot \,\tau\int_{\mathbb{R}^{d}}\int_{0}^{\frac{2}{3}}  |\frac{\partial}{\partial\varepsilon}\mathrm{P}_{\varepsilon}  (g )  (s ) |^{2} \widetilde g^{c}  (f )  (s,\varepsilon )^{2-p} \e d\varepsilon ds\\
 &  \stackrel{\mathrm{def}}{=} A'\cdot B'.
 \end{split}
\end{equation*}
Following the argument for the estimate of  $A$ in the proof of Lemma \ref{lem: duality}, we deduce  similarly that $A'\lesssim \| \widetilde g^c(f)\| _p^p$.
Now we deal with term $B'$. By Lemma \ref{lem:poison g Lusin compare}, we have
$$
B'\leq \tau\int_{\mathbb{R}^{d}}\int_{0}^{\frac{2}{3}}  |\frac{\partial}{\partial\varepsilon}\mathrm{P}_{\varepsilon}  (g )  (s ) |^{2} s^{c}  (f )  (s,\frac{\varepsilon}{2} )\e d\varepsilon ds.
$$
Then we can apply almost the same argument as in the estimate of $B$. There is only one minor difference: when $\e\geq 2^{-j}$ and $s\in Q_{m,j}$, we have $s^c(f)(s,\frac{\e}{2})\leq \mathbb{S}^c(f)(s,j+1)$. We conclude that 
$$
B'\lesssim    \|  g \|  _{{\bmo}_q^{c}}^{2}  \|  s^c(f) \|  _{p}^{2-p}.
$$
Combining the estimates of $\rm I$, $ A'$ and $B'$ with Theorem \ref{thm: duality P}, we get
$$
\|  s^c( f) \| _p+\| \mathrm{P}*f\|_p \lesssim \| \widetilde g^{c}  (f )\| _p +\| \mathrm{P}*f\|_p \lesssim \| g^{c}  (f )\| _p+\| \mathrm{P}*f\|_p .
$$

The case $p=2$ is obvious. For $p>2$, choose a positive $g\in L_{(\frac{p}{2})'}(\N) $ with norm one such that,
\begin{equation*}
\begin{split}
\|s^c(f)\|_p^2 & = \Big\|  \iint_{\widetilde{\Gamma}}|\frac{\partial}{\partial \e}\mathrm{P}_\e(f)(\cdot+t)|^2\frac{dtd\e}{\e^{d-1}}\Big\|_{\frac{p}{2}}\\
& =   \tau \int_{\R}\iint_{\widetilde{\Gamma}}|\frac{\partial}{\partial \e}\mathrm{P}_\e(f)(s+t)|^2\frac{dtd\e}{\e^{d-1}} g(s)ds\\
&=\tau \int_{\R}\int_0^1 |\frac{\partial}{\partial\e}\mathrm{P}_\e(f)(t)|^2\frac{dtd\e}{\e^{d-1}} \int_{B(t,\e)}g(s)ds.
\end{split}
\end{equation*}
By the noncommutative Hardy-Littlewood  maximal inequality (the one dimension $\mathbb{R}$ case is given by \cite[Theorem 3.3]{Mei2007}, the case $\R$ is a simple corollary of \eqref{eq: Lp infty norm} and Lemma \ref{lem: Mean function}), there exists a positive $a\in L_{(\frac{p}{2})'}(\N)$ such that $\| a\|_{(\frac{p}{2})'}\leq 1$ and 
$$
\frac{1}{|B(t,2^{-k})|}\int_{B(t,2^{-k})}g(s)ds\leq a(t), \quad   \forall\, t\in \R, \,   \forall\,\e>0.
$$
Therefore, 
\begin{equation*}
\begin{split}
\|s^c(f)\|_p^2 & =\tau \int_{\R}\int_0^1 |\frac{\partial}{\partial\e}\mathrm{P}_\e(f)(t)|^2\frac{dtd\e}{\e^{d-1}} \int_{B(t,\e)}g(s)ds \\
& \leq c_d \tau \int_{\R}\int_0^1 |\frac{\partial}{\partial\e}\mathrm{P}_\e(f)(t)|^2\e a(t)dtd\e\\
& \leq c_d \big\| \int_0^1|\frac{\partial}{\partial\e}\mathrm{P}_\e(f)(t)|^2\e {dtd\e}\big\|_{\frac{p}{2}}  \| a\|_{(\frac{p}{2})'}\\
& \leq c_d \| g^c(f)\|_{p}.
\end{split}
\end{equation*}
Then the assertion for the case $p>2$ is also proved.
\end{proof}

To proceed further, we introduce the definition of tent spaces.   In the noncommutative setting, these spaces were first defined and studied by Mei \cite{Mei-tent}.

\begin{defn}\label{def-tent}
For any function defined on $\mathbb{R}^{d}\times (0,1)=S$
with values in $L_{1}  (\mathcal{M} )+\M$, whenever it exists, we
define
$$
A^{c}  (f )(s)=  \Big(\int_{\widetilde{\Gamma}}  |f(t+s,\varepsilon) |^{2}\frac{dtd\varepsilon}{\varepsilon^{d+1}} \Big)^{\frac{1}{2}},s\in\mathbb{R}^{d}.
$$
For $1\leq p<\infty$, we define
$$
T_{p}^{c}(\R,\M)=  \{f:A^{c}  (f )\in L_{p}  (\mathcal{N} ) \}
$$
equipped with the norm $  \|  f \|  _{T_{p}^{c}(\R,\M)}=  \|  A^{c}  (f ) \|  _p$.
For $p=\infty$, define the operator-valued column $T_{\infty}^{c}$ norm of $f$ as
$$
  \|  f \|  _{T_{\infty}^{c}}=\sup_{|Q| \leq 1}  \Big\| \Big(\frac{1}{  |Q |}\int_{T(Q)}  |f  (s,\varepsilon ) |^{2}\frac{ds d\varepsilon}{\varepsilon} \Big)^{\frac{1}{2}} \Big\| _{\mathcal{M}},
$$
and the corresponding space is
$$
T_{\infty}^{c}(\R,\M)=\{f: \|f\|_{T_{\infty}^{c}}<\infty\}.
$$

\end{defn}

\begin{rmk}\label{rem: duality tent}
By the same arguments used in the proof of Theorem \ref{thm: duality P}, we can prove the duality that $T_{p}^{c}(\R,\M)^*=T_{q}^{c}(\R,\M)$ for $1\leq p<\infty$ and $\frac{1}{p}+\frac{1}{q}=1$. For the case $p=1$, it suffices to replace $\frac{\partial}{\partial\varepsilon}\mathrm{P}_{\varepsilon}  (f ) (s)$ and $\frac{\partial}{\partial\varepsilon}\mathrm{P}_{\varepsilon}  (g) (s)$ in the proof of Lemma \ref{lem: duality} by $f(s,\e)$ and $g(s,\e)$ respectively. A similar argument will give us the inclusion that $T_\infty^c(\R,\M)\subset T_{1}^{c}(\R,\M)^*$. On the other hand, since $L_\infty(\N;L_2^c(\widetilde{\Gamma}))\subset T_\infty^c(\R,\M)$, we get the reverse inclusion. 
 For $1<p<\infty$,  the tent space $T_p^c(\R,\M)$ we define above is a complemented subspace of the column tent space defined in \cite{{Mei2007}}. So by Remark 4.6 in \cite{XXX17}, 
 we obtain the duality that $T_{p}^{c}(\R,\M)^*=T_{q}^{c}(\R,\M)$.
\end{rmk}

\begin{thm}\label{thm: hq=bmoq}
For $2<q<\infty$, $\h_q^c(\R,\M)={\bmo} _q^c(\R,\M)$ with equivalent norms.
\end{thm}
\begin{proof}
First, we show the inclusion $\h_q^c(\R,\M)\subset {\bmo} _q^c(\R,\M)$. By Theorem \ref{thm: duality P}, it suffices to show that $\h_q^c(\R,\M)\subset \h_p^c(\R,\M)^*$. Applying \eqref{eq:polarized}, for any $f\in \h_q^c(\R,\M)$ and $g\in \h_p^c(\R,\M)$, we have
\begin{equation*}
\begin{split}
 \tau \int_{\R}g(s)f^*(s)ds  &  = \frac{4}{c_d}\int_{\R}\iint_{\widetilde{\Gamma}}\frac{\partial}{\partial \e}\mathrm{P}_\e(g)(s+t)\frac{\partial}{\partial \e}\mathrm{P}_\e(f)^*(s+t)\frac{dtd\e}{\e^{d-1}}ds\\
  & \;\;\;\;  +  \int_{\mathbb{R}^{d}}\mathrm{P}*g  (s )  (\mathrm{P}*f  (s ) )^{*}ds + 4\pi \int_{\mathbb{R}^{d}}I(\mathrm{P})*g  (s )  (\mathrm{P}*f  (s ) )^{*}ds.
  \end{split}
\end{equation*} 
 Then, by the H\"older inequality, 
 \begin{equation*}
\begin{split} 
 \big| \tau \int_{\R}g(s)f^*(s)ds \big|  &\leq  \big\|\e\cdot \frac{\partial}{\partial \e}\mathrm{P}_\e(g) \big\|_{L_p\big(\N;L_2(\widetilde{\Gamma})\big)}  \|\e\cdot \frac{\partial}{\partial \e}\mathrm{P}_\e(f) \|_{L_q\big(\N;L_2(\widetilde{\Gamma})\big)}\\
 & \;\;\;\;  +  \| (\mathrm{P}+I(\mathrm{P}) )*g\|_p\cdot \| \mathrm{P}*f \|_q\\
&   \leq \Big(  \big\|s^c(g) \|_p+\| (\mathrm{P}+I(\mathrm{P}) )*g\big\|_p \Big)\| f \|_{\h_q^c}.
 \end{split}
\end{equation*}
Now, we show that for any $1\leq p<2$ and $g\in {\h_p^c(\R,\M)} $, we have $\| (\mathrm{P}+I(\mathrm{P}) )*g\|_p \lesssim\| g\|_{\h_p^c} $. Since $2<q<\infty$, we have $1<\frac{q}{2}<\infty$. Applying the noncommutative Hardy-Littlewood maximal inequality,  we get
$$
\| f\|_{\bmo _q^c}\lesssim \Big\| {\sup_{s\in Q\subset \R}}^+\frac{1}{|Q|}\int_Q|f(t)|^2dt\Big\|_{\frac{q}{2}}^\frac{1}{2}\lesssim \||f|^2\|_{\frac{q}{2}}^\frac{1}{2}=\| f\|_{q}.
$$
This implies that $L_q(\N)\subset \bmo _q^c(\R,\M)$ for any $2<q\leq\infty$. Then by Theorem \ref{thm: duality P}, we get $\h_p^c(\R,\M)\subset L_p(\N)$. Therefore we deduce that 
\begin{equation}\label{eq: IP less hp}
\| (\mathrm{P}+I(\mathrm{P}) )*g\|_p  \lesssim \| g\|_p \lesssim\| g\|_{\h_p^c}.
\end{equation}
Thus, 
$$
\big| \tau \int_{\R}g(s)f^*(s)ds\big| \lesssim \| f \|_{\h_q^c}\| g\|_{\h_p^c}.
$$
We have proved $\h_q^c(\R,\M)\subset {\bmo} _q^c(\R,\M)$.

Let us turn to the reverse inclusion ${\bmo} _q^c(\R,\M)\subset \h_q^c(\R,\M)$. We need to make use of the tent spaces in Definition \ref{def-tent}.
We claim that for $q>2$, every $f\in {\bmo} _q^c(\R,\M)$ induces a linear functional on $T_p^c(\R,\M)\oplus _p L_p(\N)$. Indeed, for any $h=(h',h'')\in T_p^c(\R,\M)\oplus _p L_p(\N)$, we define
\begin{equation}\label{eq: l_f}
\begin{split}
 \ell_f(h)& =\tau \int_{\R}\int_0^1 h'(s,\e)\frac{\partial}{\partial \e}\mathrm{P}_\e(f)^*(s)d\e ds \\
 &\;\;\;\;+\tau \int_{\R}h''(s)[(\mathrm{P}*f)^*(s)+4\pi (I(\mathrm{P})*f)^*(s)]ds. 
 \end{split}
\end{equation}
Set
\begin{equation*}
\begin{split}
A^c(h')(s,\e) &=  \int_\e^1 \int_{B(s,r-\frac{\e}{2})}|h'(s,\e)|^2\frac{dt dr}{r^{d+1}},\\
\overline{A}^c(h')(s,\e) &=  \int_\e^1 \int_{B(s,\frac{r}{2})}|h'(s,\e)|^2\frac{dt dr}{r^{d+1}}.
\end{split}
\end{equation*}
Then by the Cauchy-Schwarz inequality, we arrive at
\begin{equation*}
\begin{split}
  | \ell_f(h) |   &\lesssim  \Big( \tau\int_{\mathbb{R}^{d}}\int_{0}^{1}\Big(\int_{B(s,\frac{\e}{2})}  |\frac{\partial}{\partial\varepsilon}\mathrm{P}_{\varepsilon}  (f )  (t ) |^{2} \frac{dt}{\e^{d-1}} \Big)\overline{A}^c(h')(s,\e) ^{p-2} d\varepsilon ds\Big)^\frac{1}{2}\\
 &\;\;\;\; \cdot\Big(\tau\int_{\mathbb{R}^{d}}\int_{0}^{1}\Big(\int_{B(s,\frac{\e}{2})}  |\frac{\partial}{\partial\varepsilon}\mathrm{P}_{\varepsilon}  (f )  (t ) |^{2} \frac{dt}{\e^{d-1}} \Big) \overline{A}^c(h')(s,\e) ^{2-p}d\varepsilon ds\Big)^\frac{1}{2}\\
& \;\;\;\; +   \big| \tau\int_{\mathbb{R}^{d}}h''(s)  (\mathrm{P}*f  (s ) )^{*}ds \big|+   \big| \tau\int_{\mathbb{R}^{d}}h''(s)  (I(\mathrm{P})*f  (s ) )^{*}ds \big| .
\end{split}
\end{equation*}
Following a similar argument as in the proof of Lemma \ref{lem: duality}, we obtain that
$$
| \ell_f(h) |\lesssim (\| h'\|_{T_p^c}+\|h''\|_{L_p})\| f\|_{{\bmo} _q^c}\lesssim \| h\|_{T_p^c\oplus_p L_p}\cdot \| f\|_{{\bmo} _q^c},
$$
which implies that $\| \ell_f \| \leq c_q \| f\|_{{\bmo} _q^c}$. So the claim is proved.

Next we show that $\|f\|_{\h_q^c}\leq C_q \| \ell_f \|$. By definition, we can regard $T_p^c$ as a closed subspace of $L_p\big(\N;L_2^c(\widetilde\Gamma)\big)$ in the natural way.
Then, $\ell_f$ extends to a linear functional on $L_p\big(\N;L_2^c(\widetilde\Gamma)\big)\oplus _p L_p(\N)$. Thus, there exists $g=(g',g'')\in L_q\big(\N;L_2^c (\widetilde\Gamma)\big)\oplus _q L_q(\N)$ such that
$$\| g\|_{L_q\big(\N;L_2^c(\widetilde\Gamma)\big)\oplus _q L_q(\N)}\leq \| \ell_f \|
$$
and for any $h=(h',h'')\in L_p\big(\N;L_2^c(\widetilde\Gamma)\big)\oplus _p L_p(\N)$,
\begin{equation*}
\begin{split}
 \ell_f(h)  & = \tau \int_{\R}\iint_{\widetilde\Gamma}h'(s,\e)g'^*(s,t,\e)\frac{dtd\e}{\e^{d+1}}ds+\tau\int_{\R}h''(s)g''^*(s)ds\\
&  =    \tau \int_{\R}\int_0^1h'(s,\e)\int_{B(s,\e)}g'^*(s,t,\e)dt \frac{dsd\e}{\e^{d+1}}+\tau\int_{\R}h''(s)g''^*(s)ds.
\end{split}
\end{equation*}
Comparing the equalities above with \eqref{eq: l_f}, we get
\begin{equation*}
\frac{\partial}{\partial \e}\mathrm{P}_\e(f)(s) =\frac{1}{\e^{d+1}}\int_{B(0,\e)}g'(s,t,\e)dt
\end{equation*}
and
\begin{equation*}
\mathrm{P}*f+4\pi I(\mathrm{P})*f=g''.
\end{equation*}
By Lemma \ref{lem: s equi g}, we have
\begin{equation*}
\begin{split}
\| f \|_{\h_q^c} & \lesssim  \Big\| \big(\int_{0}^1 |\frac{\partial}{\partial \e}\mathrm{P}_\e(f)|^2 \e d\e\big)^\frac{1}{2}\Big\|_q + \| \mathrm{P}*f \|_q\\
& \leq c_d \,\Big\| \big(\int_{0}^1\frac{1}{\e^{d+1}} \int_{B(0,\e)} | g'(s,t,\e) |^2 dt d\e\big)^\frac{1}{2}\Big\|_q + \| \mathrm{P}*f \|_q\\
& \lesssim  \| g' \|_{L_q\big(\N;L_2^c(\widetilde\Gamma )\big)}+ \| \mathrm{P}*f \|_q.
\end{split}
\end{equation*}
Now let us majorize the second term $\| \mathrm{P}*f \|_q$ by $\| g'' \|_q$. Indeed, consider the function $$G(s)=2\pi \int_0^\infty e^{-2\pi \e}\mathrm{P}_\e(s)d\e. $$
We can easily check that $G \in L_1(\R)$, $\|G\|_1\leq 1$ and $\widehat{G}(\xi)=(1+|\xi |)^{-1}$. This means that the operator $(1+I)^{-1}$ is a contractive Fourier multiplier on $L_q(\N)$. Therefore,
$$
\| \mathrm{P}*f \|_q \leq \| (\mathrm{P}+I(\mathrm{P}))*f \|_q \leq 4\pi \| g'' \|_q.
$$
Finally, we conclude that
$\|f\|_{\h_q^c}\lesssim \| \ell_f \|\lesssim \|f\|_{{\bmo} _q^c}$ and then $\h_q^c(\R,\M)={\bmo} _q^c(\R,\M)$ with equivalent norms.
\end{proof}

Armed with the theorem above, we are able to  extend the content of Theorem \ref{lem: L bmo}.

\begin{thm} \label{thm: bdd retraction}
$\, $
\begin{enumerate}[\rm(1)]
\item  The map $\fF$ extends to a bounded map from $L_\infty\big(\N;L_2^c(\widetilde\Gamma)\big)\oplus _\infty L_\infty(\N)$ into $\bmo^c(\R,\M)$ and
$$
\|\fF(h)\|_{\bmo^c}\lesssim \| h\|_{L_\infty\big(\N;L_2^c(\widetilde\Gamma)\big)\oplus_\infty L_\infty(\N)}.
$$
\item For $1<p<\infty$, $\fF$ extends to a bounded map from $L_p\big(\N;L_2^c(\widetilde\Gamma)\big)\oplus _p L_p(\N)$ into $\h_p^c(\R,\M)$ and
$$
\|\fF(h)\|_{\h_p^c}\lesssim \| h\|_{L_p\big(\N;L_2^c(\widetilde{\Gamma})\big)\oplus_p L_p(\N)}.
$$
\end{enumerate}

\end{thm}

\begin{proof}
(1) is already contained in Theorem \ref{lem: L bmo}.
When $p>2$, (2) follows from Theorem  \ref{lem: L bmo} and Theorem \ref{thm: hq=bmoq}. The case $p=2$ is trivial. For the case $1<p<2$, according to Theorem \ref{thm: duality P}, we have 
$$
 \| \fF(h)\|_{\h_p^c}  \lesssim  \sup_{\| f \|_{{\bmo} _q^c}\leq 1}\big| \tau\int_{\R}\fF(h)(s)f^*(s)ds \big|.
$$
Then, by Theorem \ref{thm: hq=bmoq} and \eqref{eq: IP less hp}, for $h=(h',h'')\in L_p\big(\N;L_2^c(\widetilde{\Gamma})\big)\oplus_p L_p(\N)$, we have
\begin{equation*}
\begin{split}
&  \sup_{\| f \|_{{\bmo} _q^c}\leq 1}\big| \tau\int_{\R} \fF(h)(s)f^*(s)ds \big|\\
&\lesssim \sup_{\| f \|_{\h_q^c}\leq 1}\Big| \tau\int_{\R}\big[\iint_{\widetilde{\Gamma}}h'(s,t,\e)\frac{\partial}{\partial \e}\mathrm{P}_\e(f)^*(s+t)dtd\e  + h''(s)([\mathrm{P}+4\pi I(\mathrm{P})]*f^*(s))\big]ds \Big|\\
 & \lesssim \| h\|_{L_p\big(\N;L_2^c(\widetilde{\Gamma})\big)\oplus_p L_p(\N)}.
\end{split}
\end{equation*}
The desired inequality is proved.
\end{proof}

The above theorem shows that,  for any $1<p<\infty$, $\h _p^c(\R,\M)$ is a complemented subspace of $L_{p}  \big(\mathcal{N};L_{2}^{c}(\widetilde{\Gamma}) \big)\oplus_p L_{p}  (\mathcal{N} )$. Thus, we deduce the following duality theorem:

\begin{thm}\label{cor: duality}
We have $ \h_p^c(\R,\M)^*=\h_q^c(\R,\M)$ with equivalent norms for any $1<p<\infty$.
\end{thm}

%%%%%%%%%%%%%%%
\section{Interpolation}\label{section-interp}
%%%%%%%%%%%%%%%

In this section we study the interpolation of local Hardy and bmo spaces by transferring the problem to that of the operator-valued Hardy and $\BMO$ spaces defined in \cite{Mei2007}. We begin with an easy observation on the difference between $\bmo_q^c$ and $\BMO_q^c$ norms.

\begin{lem}\label{cor: equi norm of bmo q}
For $2<q\leq \infty$, we have
$$  \|  g \|  _{{\bmo}_q^{c}}\approx \big( \|  g \|  _{{\BMO}_q^{c}} ^q +  \|  J( \mathrm{P})*g \|  _q^q\big)^\frac{1}{q}.
$$

\end{lem}

\begin{proof}
Repeating the proof of Proposition \ref{BMO-leq-bmo} with $\|\cdot\|_\M $ replaced by $\|\cdot\|_{L_{\frac{q}{2}}(\N;\ell_\infty)}  $, we have $\|g\|_{\BMO_q^c} \lesssim \|g\|_{\bmo_q^c}$. By Lemma \ref{lem: Caleson<bmo q}, it is also evident that $\|  J( \mathrm{P})*g \|  _q \lesssim \|g\|_{\bmo_q^c}$. Then we obtain 
$$ \big( \|  g \|  _{{\BMO}_q^{c}} ^q +  \|  J( \mathrm{P})*g \|  _q^q\big)^\frac{1}{q}\lesssim \|  g \|  _{{\bmo}_q^{c}}.
$$
On the other hand, by Corollary \ref{cor: Carleson=Dbmo}, we have
$$
  \|  g \|  _{{\bmo}_q^{c}}\lesssim \Big\|\underset{\substack{s\in Q\subset \R\\ |Q|<1}}{\sup{} ^+}\frac{1}{|Q|}\int_{T(Q)}|\frac{\partial}{\partial\varepsilon}\mathrm{P}_{\varepsilon}(g)(t)|^2\e dtd\e \Big\|_{\frac{q}{2}}^{\frac{1}{2}}  + \|  J( \mathrm{P})*g \|  _q .
$$
Clearly, the first term on the right side can be estimated from above by $\|g\|_{\BMO_c^q}$ (see \cite[Theorem~3.4]{XXX17}). Therefore,
$$
  \|  g \|  _{{\bmo}_q^{c}}\lesssim \|g\|_{\BMO_c^q} + \|  J( \mathrm{P})*g \|  _q\approx   \big( \|  g \|  _{{\BMO}_q^{c}} ^q +  \|  J( \mathrm{P})*g \|  _q^q\big)^\frac{1}{q}.
$$
Thus, the lemma is proved.
\end{proof}

 Define $F_q(\N)$ to be the space of all $ f\in L_q(\M;\mathrm R^c_d)$ such that $\|  J( \mathrm{P})*f \|_q<\infty$. From the above lemma, we see that $\bmo _q^c(\R,\M)$ and $\BMO_ q^c(\R,\M)\oplus_q F_q(\N)$ have equivalent norms. By the interpolation between $\BMO_ q^c(\R,\M)$ and $\BMO^c(\R,\M)$ (see \cite{Mei2007} for more details), we deduce the following lemma:

\begin{lem}\label{lem: interpolation of dyadic}
Let $2<q< \infty$ and $0<\theta<1$. Then
$$
\big({\bmo} _{q}^c(\R,\M),   {\bmo}^c(\R,\M)  \big)_{\theta}\subset  {\bmo} _{\varrho}^c(\R,\M) \quad \text{ with }\quad \varrho=\frac{q}{1-\theta}.
$$

\end{lem}

\begin{proof}
By Lemma \ref{cor: equi norm of bmo q}, we see that
$$
{\bmo} _{q}^c(\R,\M)=\BMO_ q^c(\R,\M)\oplus_q F_q(\N).
$$
with equivalent norms.
Define a map
\begin{equation*}
\begin{split}
\varUpsilon_q: F_q(\N) \longrightarrow &\, L_q(\N)\\
f   \longmapsto & \,J({\rm{P}})*f.
\end{split}
\end{equation*}
Thus,  $\varUpsilon_q$ defines an isometric embedding of $F_q(\N)$ into $L_q(\N)$. Then by the interpolation between $\BMO_ q^c(\R,\M)$ and $\BMO^c(\R,\M)$, we get
\begin{equation*}
\begin{split}
\big({\bmo} _{q}^c(\R,\M),   {\bmo}^c(\R,\M)  \big)_{\theta} &=  \big(\BMO_ q^c(\R,\M)\oplus_q F_q(\N), \BMO^c(\R,\M)\oplus_\infty F_\infty(\N)\big)_{\theta}\\
&=  \big( \BMO_ q^c(\R,\M),  \BMO^c(\R,\M) \big)_{\theta} \oplus_{\varrho}  \big(F_q(\N),  F_\infty(\N)\big)_{\theta}\\
& \subset   \BMO_ {\varrho}^c(\R,\M) \oplus_{\varrho} F_{\varrho}(\N)={\bmo} _{\varrho}^c(\R,\M),
\end{split}
\end{equation*}
which completes the proof.
\end{proof}

\begin{thm}\label{thm: interpolation}
Let $1<p<\infty$. We have
$$
\big({\bmo}^c(\R,\M), \h_1^c(\R,\M)\big)_{\frac{1}{p}}=\h_p^c(\R,\M).
$$
\end{thm}

\begin{proof}
Let $1<p<2$ and $\frac{1}{p'} =\frac{1-\theta}{p} + \theta$. Since the map $\fE$ in Definition \ref{def:map} is an isometry from $\h _p^c(\R,\M)$ to $L_{p}  \big(\mathcal{N};L_{2}^{c}(\widetilde{\Gamma}) \big)\oplus_p L_{p}  (\mathcal{N} )$, we have
\begin{equation}\label{eq: h1 hp interpolation subset}
\big(\h_p^c(\R,\M),\h _1^c(\R,\M)\big)_\theta \subset \h _{p'}^c(\R,\M).
\end{equation}
By Theorem \ref{cor: duality}, $\h_p^c$ is a reflexive Banach space. Then applying \cite[Corollary~4.5.2]{BL1976}, we know that the dual of $\big(\h_p^c(\R,\M),\h _1^c(\R,\M)\big)_\theta$ is $\big({\bmo} _{q}^c(\R,\M),   {\bmo}^c(\R,\M)  \big)_{\theta}$. Therefore, if the inclusion \eqref{eq: h1 hp interpolation subset} is proper, we will get the proper inclusion 
$${\bmo} _{\varrho}^c(\R,\M)\subsetneq \big({\bmo} _{q}^c(\R,\M),   {\bmo}^c(\R,\M)  \big)_{\theta},$$
 which is in contradiction with Lemma \ref{lem: interpolation of dyadic}. Thus, we have
\begin{equation}\label{eq: h1 hp interpolation 2}
\big(\h_p^c(\R,\M),\h _1^c(\R,\M)\big)_\theta = \h _{p'}^c(\R,\M).
\end{equation}
 By duality and \cite[Corollary~4.5.2]{BL1976} again, the above equality implies that for $q' = \frac{q}{1-\theta}$,
\begin{equation}\label{eq: h1 hp interpolation 2 bis}
\big(\h_q^c(\R,\M),\bmo^c(\R,\M)\big)_\theta = \h _{q'}^c(\R,\M).
\end{equation}

For the case where $1< p_1 ,p_2 <\infty$, the interpolation of $\h_{p_1}^c(\R,\M)$ and $ \h_{p_2}^c(\R,\M)$ is much easier to handle. Indeed, by Theorem \ref{thm: bdd retraction}, we have,  for any $1<p<\infty$, $\h _p^c(\R,\M)$ is a complemented subspace of $L_{p}  \big(\mathcal{N};L_{2}^{c}(\widetilde{\Gamma}) \big)\oplus_p L_{p}  (\mathcal{N} )$ via the maps $\fE$ and $\fF$ in Definition \ref{def:Phi&Psi}. This implies that, for any $1<p_1,p_2<\infty$,
$$
\big( \h_{p_1}^c(\R,\M), \h_{p_2}^c(\R,\M)\big)_\theta =\h_p^c(\R,\M),
$$
with $\frac 1 p =\frac {1-\theta}{p_1} +\frac{\theta}{p_2}$.
Combining this equivalence with  \eqref{eq: h1 hp interpolation 2}, \eqref{eq: h1 hp interpolation 2 bis}, and applying Wolff's interpolation theorem (see \cite{Wolff}), we get the desired assertion.
\end{proof}

 The following theorem is the mixed version of Theorem \ref{thm: interpolation}, which states that $\h_1(\R,\M)$ and $\bmo(\R,\M)$ are also good endpoints of $L_p(\N)$.

\begin{thm}\label{interp-mix}
Let $1<p<\infty$. We have $\big(X, Y \big)_{\frac 1 p } = L _p(\N)$ with equivalent norms, where $X = \bmo(\R,\M)$ or $L_\infty(\N)$, and $Y =\h_1(\R,\M)$ or $L_1(\N)$.
\end{thm}

\begin{proof}
By the same argument as in the proof of Theorem \ref{thm: interpolation}, we have the inclusion
$$\big(\bmo_q(\R,\M),  \bmo (\R,\M)\big)_\theta \subset \bmo_{q'}(\R,\M)\;\quad q' =\frac q \theta ,$$
 which ensures by duality that
 $$\big(\h_p(\R,\M),  \h_1 (\R,\M)\big)_\theta \supset \h_{p'} (\R,\M) = L_{p'} (\N) $$
 for $\frac{1}{p'} =\frac{1-\theta}{p} + \theta$. Then by Proposition \ref{equi-Hp-hp-Lp},
$$L_{p'} (\N )    \subset  \big(\h_p(\R,\M),  \h_1 (\R,\M)\big)_\theta = \big(L_p(\N),  \h_1 (\R,\M)\big)_\theta.
$$
Since $\h_1 (\R,\M)\subset L_1 (\N)$, then  $$
\big(\h_p(\R,\M),  \h_1 (\R,\M)\big)_\theta \subset \big(L_p(\N),  L_1 (\N)\big)_\theta = L_{p'} (\N). $$
Combining the estimates above, we have 
$$
\big(\h_p(\R,\M),  \h_1 (\R,\M)\big)_\theta=L_{p'} (\N).
$$
Again, using duality and Wolff's interpolation theorem, we conclude the proof by the same trick as in the proof of Theorem \ref{thm: interpolation}.
\end{proof}

We end this section by some  real interpolation results.

\begin{cor}\label{real-inter-hardy}
Let $1<p<\infty$. Then we have
\begin{enumerate}[\rm(1)]
\item $\big(\bmo^c(\R,\M), \h_1^c(\R,\M)\big)_{\frac 1 p , p} = \h _p^c(\R,\M)$ with equivalent norms.
\item $\big(X, Y \big)_{\frac 1 p , p} = L _p(\N)$ with equivalent norms, where $X = \bmo(\R,\M)$ or $L_\infty(\N)$, and $Y =\h_1(\R,\M)$ or $L_1(\N)$.
\end{enumerate}
\end{cor}

\begin{proof}
Both (1) and (2) follow from \cite[Theorem~4.7.2]{BL1976}; we only prove (1). Let $1<p_1 < p< p_2 <\infty$ with $\frac 1 p = \frac{1-\eta}{ p_1 }  +\frac{\eta}{p_2}$. By \cite[Theorem~4.7.2]{BL1976}, we write
\begin{equation}
\begin{split}
& \big(\bmo^c(\R,\M), \h_1^c(\R,\M)\big)_{\frac 1 p , p}\\
 &= \Big(\big(\bmo^c(\R,\M), \h_1^c(\R,\M)\big)_{\frac {1}{p_1}}, \big(\bmo^c(\R,\M), \h_1^c(\R,\M)\big)_{\frac {1}{p_2}}\Big)_{\eta, p}.
\end{split}
\end{equation}
Then the assertion (1) follows from Theorem \ref{interp-mix} and the facts that $\big(L_{p_1}(\N), L_{p_2}(\N)\big)_{\eta, p } =L_p(\N) $ and that $\h_p^c(\R,\M)$ is a complemented subspace of $L_{p}  \big(\mathcal{N};L_{2}^{c}(\widetilde{\Gamma}) \big)\oplus_p L_{p}  (\mathcal{N} )$.
\end{proof}

%%%%%%%%%%%%%%%%%%%%%%%
\section{Calder\'on-Zygmund theory}\label{section-CZ}
%%%%%%%%%%%%%%%%%%%%%%%%%

We introduce the Calder\'on-Zygmund theory on operator-valued local Hardy spaces in this section. It is closely related to the similar results of \cite{HLMP2014}, \cite{JMP2014}, \cite{Parcet} and \cite{XXY17}. The results in the following will be used in the next section to investigate various square functions that characterize local Hardy spaces.

Let $K$ be an $L_{1}  (\mathcal{M} )+\M)$-valued tempered distribution which coincides on $\mathbb{R}^{d}\setminus  \{ 0 \} $
with a locally integrable $L_{1}  (\mathcal{M} )+\mathcal{M}$-valued
function. We define the left singular integral operator $K^{c}$ associated to $K$ by
\[
K^{c}  (f )  (s )=\int_{\mathbb{R}^{d}}K  (s-t )f  (t )dt,
\]
 and the right singular integral operator $K^{r}$ associated to $K$ by
\[
K^{r}  (f )  (s )=\int_{\mathbb{R}^{d}}f  (t )K  (s-t )dt.
\]
Both   $K^{c}  (f )$ and $K^{r}  (f )$ are well-defined
for sufficiently nice functions $f$ with values in $L_{1}  (\mathcal{M} )\cap \M$,
for instance, for $f\in\mathcal{S}\otimes  (L_{1}  (\mathcal{M} )\cap \M )$.

Let ${\bmo}_{0}^{c}  (\mathbb{R}^{d},\mathcal{M} )$ denote the subspace of 
${\bmo}^{c}  (\mathbb{R}^{d},\mathcal{M} )$ consisting of compactly supported 
functions. The following lemma is an analogue of Lemma 2.1 in \cite{XXX17} for inhomogeneous spaces. Notice that the usual Calder\'on-Zygmund operators (the operators satisfying the condition (1) and (3) in the following lemma) are not necessarily bounded on $\h_1^c(\R,\M)$. Thus, we need to impose an extra decay at infinity on the kernel $K$.

\begin{lem}
\label{C-Z lem}Assume that
\begin{enumerate}[\rm(1)]
\item the Fourier transform of $K$ is
bounded: $\sup_{\xi\in\mathbb{R}^{d}}  \|  \widehat{K}  (\xi ) \|  _{\mathcal{M}}<\infty$;
\item $K$ satisfies the size estimate at infinity: there exist $C_1$ and $\rho >0$ such that 
\[
  \|  K  (s ) \|  _{\mathcal{M}}\leq \frac{C_1}{  |s |^{d+\rho}},\thinspace  \forall\, |s|\geq 1;
\]
\item $K$ has the Lipschitz regularity:
there exist $C_2$ and $\gamma >0$ such that
\[
  \|  K  (s-t )-K  (s ) \|  _{\mathcal{M}}\leq C_2\frac{  |t |^\gamma}{  |s-t |^{d+\gamma}},\thinspace  \forall\,  |s |>2  |t |.
\]
\end{enumerate}
Then $K^{c}$ is bounded on $\h_{p}^{c}  (\mathbb{R}^{d} ,\mathcal{M} )$ for $1\leq p<\infty$
and from ${\bmo}_{0}^{c}  (\mathbb{R}^{d},\mathcal{M} )$
to ${\bmo}^{c}  (\mathbb{R}^{d},\mathcal{M} )$.

A similar statement also holds for $K^{r}$ and the corresponding
row spaces.
 \end{lem}
\begin{proof}
First suppose that $K^c$ maps constant functions to zero. This amounts to requiring that $K^c(\mathbbm{1}_{\R})=0$. Let $Q\subset \R$ be a cube with $|Q|<1$. 
Since the assumption of Lemma 2.1 in \cite{XXX17} are included in the ones of this lemma, we get
\[
  \Big\|  \big(\frac{1}{  |Q |}\int_{Q}  |K^{c}  (f )-K^{c}  (f )_{Q} |^{2}dt \big)^{\frac{1}{2}} \Big\|  _{\mathcal{M}}  \lesssim \|f\|_{\BMO^c} \lesssim \|f\|_{\bmo^c} .
\]
Now let us focus on the cubes with side length $1$. Let $Q$ be a cube with $  |Q |=1$ and $\widetilde{Q}=2Q$ be the cube concentric with $Q$ and with
side length $2$. Decompose $f$ as $f=f_{1}+f_{2}$, where
$f_{1}=\mathbbm{1}_{\widetilde{Q}}f$ and $f_{2}=\mathbbm{1}_{\mathbb{R}^{d}\setminus\widetilde{Q}}f$.
Then $K^{c}  (f )=K^{c}  (f_{1} )+K^{c}  (f_{2} )$. We have
\begin{equation*}
  \Big\|  \frac{1}{  |Q  |}\int_{Q }  |K^{c}  (f ) |^{2}ds \Big\|  _{\mathcal{M}}  \lesssim  \Big\|  \frac{1}{  |Q  |}\int_{Q }  |K^{c}  (f_{1} ) |^{2}ds \Big\|  _{\mathcal{M}}+  \Big\|  \frac{1}{  |Q |}\int_{Q}  |K^{c}  (f_{2} ) |^{2}ds \Big\|  _{\mathcal{M}}.
\end{equation*}
The first term is easy to estimate. By assumption (1) and \eqref{eq: 2-1}, 
\begin{equation*}
\begin{split}
  \Big\|  \frac{1}{  |Q |}\int_{Q}  |K^{c}  (f_{1} ) |^{2}ds \Big\|  _{\mathcal{M}} & \leq \Big\|  \frac{1}{  | Q  |}\int_{\mathbb{R}^{d}}  |\widehat{K}  (\xi )\widehat{f_{1}}  (\xi ) |^{2}d\xi\Big \|  _{\mathcal{M}}\\
 &   \lesssim  \Big \|  \frac{1}{  | Q  |}\int_{\mathbb{R}^{d}}  |\widehat{f_{1}}  (\xi ) |^{2}d\xi \Big\|  _{\mathcal{M}}\\
  &= \Big\|  \frac{1}{  | Q  |}\int_{\widetilde{Q}}  |f  (s) |^{2}ds \Big\|  _{\mathcal{M}}\\
 &  \lesssim \sup_{  |Q |= 1}  \Big\|  \frac{1}{  |Q |}\int_{Q}  |f  (s) |^{2}ds \Big\|  _{\mathcal{M}}.
\end{split}
\end{equation*}
To estimate the second term, using assumption (2) and \eqref{eq: 2-1} again, we have
\begin{equation*}
\begin{split}
  |K^{c}  (f_{2} )  (s ) |^{2}  & = \Big|\int_{\mathbb{R}^{d}}K  (s-t )f_{2}  (t )dt \Big|^{2}= \Big|\int_{\mathbb{R}^{d}\setminus\widetilde{Q}}K  (s-t )f  (t )dt \Big|^{2}\\
 &  \leq \int_{\mathbb{R}^{d}\setminus\widetilde{Q}}  \|  K  (s-t ) \|  _{\mathcal{M}}dt\cdot\int_{\mathbb{R}^{d}\setminus\widetilde{Q}}  \|  K  (s-t ) \|  _{\mathcal{M}}^{-1}  |K  (s-t )f  (t ) |^{2}dt\\
 &  \lesssim  \int_{\mathbb{R}^{d}\setminus\widetilde{Q}}  \|  K  (s-t ) \|  _{\mathcal{M}}  |f  (t ) |^{2}dt\\
&  \lesssim \int_{\mathbb{R}^{d}\setminus\widetilde{Q}}\frac{1}{  |s-t |^{d+\rho}}  |f  (t ) |^{2}dt.
 \end{split}
\end{equation*}
Set $\widetilde{Q}_m = \widetilde{Q}+2m$ for every $m\in \mathbb{Z}^d$. Then $\mathbb{R}^{d}\setminus\widetilde{Q}=\cup_{m\neq 0}\widetilde{Q}_m $. Continuing the estimate of $  |K^{c}  (f_{2} )  (s ) |^{2}$,  for
any $s\in  Q $, we have
 \begin{equation*}
 \begin{split}
  |K^{c}  (f_{2} )  (s ) |^{2} &  \leq  \sum_{m\neq 0}\int_{\widetilde{Q}_m}\frac{1}{  |s-t |^{d+\rho}}  |f  (t ) |^{2}dt\\
 &   \approx  \sum_{m\neq 0}\frac{1}{  |m |^{d+\rho}}\int_{\widetilde{Q}_m}  |f  (t ) |^{2}dt\lesssim \|f\|_{\bmo^c}.
\end{split}
\end{equation*}
Combining the previous estimates, we deduce that $K^{c}$ is bounded from ${\bmo}_{0}^{c}  (\mathbb{R}^{d},\mathcal{M} )$
to ${\bmo}^{c}  (\mathbb{R}^{d},\mathcal{M} )$.

Now we illustrate that the additional requirement $K^c(\mathbbm{1}_{\R})=0$ is not needed. First, a similar argument as above ensures that for every compactly supported $f\in L_\infty(\N)$, $\|K^c(f)\|_{\bmo^c} \lesssim \|f\|_\infty$. Then we follow the argument of \cite[Proposition~II.5.15]{C-F} to extend $K^c$ on the whole $ L_\infty (\N)$, as
$$K^c(f)(s) = \lim_j \big[K^c(f\mathbbm{1}_{B_j})(s) -\int_{1<|t|\leq j} K(-t) f(t) dt \big],\quad   \forall\, s\in \R,$$
where $B_j$ is the ball centered at the origin with radius $j$.
Let us show that the sequence on the right hand side converges pointwise in the norm $\|\cdot\|_{\M}$ and uniformly on any compact set $\Omega \subset \R$. To this end, we denote
by $g_j$ the $j$-th term of this sequence. Let $l$ be the first natural number such that $l \geq 2\sup _{s\in \Omega} |s|$. Then for $s\in \Omega$ and $j> l$, we have
$$g_j(s)=g_l(s) +\int_{ l<|t| \leq j} \big(K(s-t) -K(-t)\big) f(t) dt . $$
By assumption (3), the integral on the right hand side is bounded by a bounded multiple of $\|f\|_\infty$, uniformly on $s\in \Omega$. This ensures the convergence of $g_j$, so $K^c(f)$ is a well-defined function. Now we have to estimate the $\bmo^c$-norm of $K^c(f)$. Taking any cube $Q\subset \R$, by the uniform convergence of $g_j$ on $Q$ in $\M$, we have
\begin{equation*}
\big\| \big(\int_Q | K^c(f)(s) - (K^c(f))_Q   |^2 ds\big)^{\frac 1 2 } \big\|_{\M}  = \lim_j\big\| \big(\int_Q   | g_j(s) - (g_j)_Q   |^2 ds \big)^{\frac 1 2 } \big\|_{\M}.
\end{equation*}
Similarly,
$$
\big\| \big(\int_Q | K^c(f)(s)    |^2 ds\big)^{\frac 1 2 } \big\|_{\M} =  \lim_j\big\| \big(\int_Q   | g_j(s)    |^2 ds \big)^{\frac 1 2 } \big\|_{\M} .$$
Hence,  by the fact that $g _j$ and $K^c(f\mathbbm{1}_{B_j})$ differ by a constant, we obtain
$$ \|K^c(f)\|_{\bmo^c} =\lim_j \|g_j\|_{\bmo^c}  \lesssim\limsup_j \|K^c(f\mathbbm{1}_{B_j})\|_{\bmo^c} +\|f\|_\infty \lesssim \|f\|_\infty.$$
Therefore, $K^c$ defined above extends to a bounded operator from $L_\infty(\N) $ to $\bmo^c(\R,\M)$. In particular, $K^c(\mathbbm{1}_{\R})$ determines a function in $\bmo^c(\R,\M)$. Then for $f$ and $Q$ as  above, we have $K^c(f)=K^c(f_1)+K^c(f_2)+K^c(\mathbbm{1}_{\R})f_{\widetilde Q}$, so
 \begin{equation*}\begin{split}
 \|K^c(f)\|_{\bmo^c}
 &\leq \|K^c(f_1)\|_{\bmo^c}+\|K^c(f_2)\|_{\bmo^c}+\|K^c(\mathbbm{1}_{\R})\|_{\bmo^c}\,\|f_{\widetilde Q}\|_{\M}\\
 &\lesssim  \|f\|_{\bmo^c}+\|f_{\widetilde Q}\|_{\M}\lesssim \|f\|_{\bmo^c}\,.
   \end{split}
   \end{equation*}
 Thus we have proved the $\bmo^c$-boundedness of $K^c$ in the general case.

By duality, the boundedness of $K^c$ on $\h_1^c(\R,\M)$ is equivalent to that of its adjoint map $(K^c)^*$ on $\bmo^c_0(\R,\M)$. It is easy to see that $(K^{c})^*$ is also a singular integral operator:
\[
  (K^{c} )^{*}  (g )=\int_{\mathbb{R}^{d}}\widetilde{K}  (s-t )g  (t )dt,
\]
 where $\widetilde{K}  (s )=K^{*}  (-s )$. Obviously,
$\widetilde{K}$ also satisfies the same assumption as $K$, so $  (K^{c} )^{*}$
is bounded on ${\bmo}_0^{c}  (\mathbb{R}^{d},\mathcal{M} )$.
Thus we get the boundedness of $K^{c}$ on $\h_{1}^{c}  (\mathbb{R}^{d},\mathcal{M} )$. Then, by the interpolation between $\h_{1}^{c}  (\mathbb{R}^{d},\mathcal{M} )$ and ${\bmo}^{c}  (\mathbb{R}^{d},\mathcal{M} )$ in Theorem \ref{thm: interpolation}, we get the boundedness of $K^c$ on $\h_p^c(\R,\M)$ for $1<p<\infty$. The assertion is proved.
\end{proof}

\begin{rmk}
Under the assumption of the above lemma,  $K^c(\mathbbm{1}_{\R})$ is a constant, so it is the zero element  in $\BMO^c(\R,\M)$.

\end{rmk}

A special case of Lemma \ref{C-Z lem} concerns the Hilbert-valued
kernel $K$. Let $H$ be a Hilbert space and $\mathsf{k}:\mathbb{R}^{d}\rightarrow H$
be a $H$-valued kernel. We view the Hilbert space as the column matrices
in $B  (H )$ with respect to a fixed orthonormal basis. Put $K  (s )=\fk  (s )\otimes1_{\mathcal{M}}\in B  (H )\overline{\otimes}\mathcal{M}$.
For nice functions $f:\mathbb{R}^{d}\rightarrow L_{1}  (\mathcal{M} )+\mathcal{M} $,
$K^{c}  (f )$ takes values in the column subspace of $L_{1}  (B  (H )\overline{\otimes}\mathcal{M} )+L_{\infty}  (B  (H )\overline{\otimes}\mathcal{M} )$.
Consequently,
\[
  \|  K^{c}  (f ) \|  _{L_{p}  (B  (H )\overline{\otimes}\mathcal{N} )}=  \|  K^{c}  (f ) \|  _{L_{p}  (\mathcal{N};H^{c} )}.
\]
Since $\fk  (s )\otimes1_{\mathcal{M}}$ commutes with $\mathcal{M}$,
$K^{c}  (f )=K^{r}  (f )$ for $f\in L_{2}  (\mathcal{N} )$.
Let us denote this common operator by $\fk^{c}$. Here the superscript
$c$ refers to the previous convention that $H$ is identified with
the column matrices in $B  (H )$. Thus, Lemma \ref{C-Z lem}
implies the following
\begin{cor}\label{Cor C-Z}
Assume that
\begin{enumerate}[\rm(1)]
\item $\sup_{\xi\in\mathbb{R}^{d}}  \|  \widehat{\fk}  (\xi ) \|  _{H}<\infty$;
\item $  \|  \fk  (s ) \|  _{H}\lesssim\frac{1}{  |s |^{d+\rho}}$,
$\thinspace  \forall\, |s|\geq 1$,
for some $\rho>0$;
\item $  \|  \fk  (s-t )- \fk  (s ) \|  _{H}\lesssim\frac{  |t |^\gamma}{  |s-t |^{d+\gamma}}$,
$\thinspace  \forall\,  |s |>2  |t |$, for some $\gamma>0$.
\end{enumerate}
Then the operator $\fk^{c}$ is bounded
\begin{enumerate}[\rm(1)]
\item from ${\bmo}_{0}^{\alpha}  (\mathbb{R}^{d},\mathcal{M} )$
to ${\bmo}^{\alpha}  (\mathbb{R}^{d},B  (H )\overline{\otimes}\mathcal{M} )$,
where $\alpha=c$, $\alpha=r$ or we leave out $\alpha$;
\item and from $\h_{p}^{c}  (\mathbb{R}^{d},\mathcal{M} )$
to $\h_{p}^{c}  (\mathbb{R}^{d},B  (H )\overline{\otimes}\mathcal{M} )$ for $1\leq p<\infty$.
\end{enumerate}
\end{cor}

\begin{proof}
Since $K^c(f)=K^r(f)$ on the subspace $L_p(\N)\subset L_p(B(H)\overline{\otimes}\N)$, (1) follows immediately from Lemma \ref{C-Z lem}. Denote the column subspace of $\bmo^c(\R,B(H)\overline{\otimes}\M)$ (resp. $\h_{p}^{c}  (\mathbb{R}^{d},B(H)\overline{\otimes}\mathcal{M} )$) by $\bmo^c(\R, H^c \overline{\otimes}\M)$ (resp. $\h_{p}^{c}  (\mathbb{R}^{d},H^c\overline{\otimes}\mathcal{M} )$). Consider the adjoint operator of $\fk^c$ which is denoted by $(\fk^c)^*$. It admits the convolution kernel $\widetilde{K}(s)=\widetilde{\fk}(s) \otimes 1_\M$, where $\widetilde{\fk}(s)=\fk(-s)^*$ (so it is a row matrix). Applying Lemma \ref{C-Z lem} to $(\fk^c)^*$, we get that $(\fk^c)^*$ is bounded from $\bmo^c(\R, H^c \overline{\otimes}\M)$ to $\bmo^c(\R, \M)$. Then $\fk^c$ is bounded from $\h_1^c(\R, \M)$ to $\h_1^c(\R, H^c \overline{\otimes}\M)$, and thus bounded from $\h_1^c(\R, \M)$ into $\h_1^c(\R, B(H) \overline{\otimes}\M)$. Interpolating this with the boundedness of $\fk^c$ from ${\bmo}_{0}^{c}  (\mathbb{R}^{d},\mathcal{M} )$
to ${\bmo}^{c}  (\mathbb{R}^{d},B  (H )\overline{\otimes}\mathcal{M} )$, we deduce the desired assertion in (2).

\end{proof}

\begin{rmk}\label{rem: L_p H h_p}
Let $1\leq p\leq 2$. Since $L_\infty (\N)  \subseteq \bmo^c(\R,\M) $, we get $\h_1^c(\R,\M)\subseteq L_1(\N)$. By Theorem \ref{thm: interpolation} and the fact that $\h_2^c(\R,\M) = L_2(\N) $, we have $\h_{p}^{c}  (\mathbb{R}^{d},\mathcal{M} )\subseteq L_{p}  (\mathcal{N} )$. Then Corollary \ref{Cor C-Z} ensures that
\[
  \|  \fk^{c}  (f ) \|  _{L_{p}  (\mathcal{N};H^{c} )}\lesssim  \|  \fk^{c}  (f ) \|  _{\h_{p}^{c}  (\mathbb{R}^{d},B  (H )\overline{\otimes}\mathcal{M} )}\lesssim  \|  f \|  _{\h_{p}^{c}  (\mathbb{R}^{d},\mathcal{M} )}
\]
for any $f\in \h_{p}^{c}  (\mathbb{R}^{d},\mathcal{M} )$.
\end{rmk}

%%%%%%%%%%%%%%%%%%%%%
\section{General characterizations}\label{section-general charact}
%%%%%%%%%%%%%%%%%%%%%

Applying the operator-valued Calder\'on-Zygmund theory developed in the last section, we will show  that the Poisson kernel in the square functions which are used to define $\h_p^c(\R,\M)$ can be replaced by any reasonable test function. As an application, we are able to compare the operator-valued  local Hardy spaces $\h_p^c(\R,\M)$ defined in this paper with the operator-valued  Hardy spaces $\mathcal{H}_p^c(\R,\M)$ in \cite{Mei2007}. We will use multi-index notation. For $m=  (m_{1},\cdots,m_{d} )\in\mathbb{N}_{0}^{d}$ and $s=  (s_{1},\cdots,s_{d} )\in\mathbb{R}^{d}$,
we set $s^{m}=s_1^{m_{1}}\cdots s_d^{m_{d}}$. Let $  |m |_{1}=m_{1}+\cdots+m_{d}$
and $D^{m}=\frac{\partial^{m_{1}}}{\partial s_{1}^{m_{1}}}\cdots\frac{\partial^{m_{d}}}{\partial s_{d}^{m_{d}}}$.

\subsection{General characterizations} 
Let $\Phi$ be a complex-valued infinitely differentiable function defined on $\R\backslash\lbrace 0\rbrace$. Recall that $\widetilde{\Gamma}= \{ (t,\varepsilon )\in\mathbb{R}_{+}^{d+1}:  |t |<\varepsilon<1 \} $ and $\Phi_\e(s) =\e^{-d}  \Phi(\frac{s}{\e})$.
 For any $f\in L_{1}  (\mathcal{M};\mathrm R_{d}^{c} )+L_{\infty}  (\mathcal{M};\mathrm R_{d}^{c} )$,
we define the local versions of the conic and radial square functions
of $f$ associated to $\Phi$ by
\begin{equation*}
\begin{split}
s_{\Phi}^{c}  (f )  (s )  &=  \Big(\iint_{\widetilde{\Gamma}}  |\Phi_{\varepsilon}*f  (s+t ) |^{2}\frac{dtd\varepsilon}{\varepsilon^{d+1}} \Big)^{\frac{1}{2}},\thinspace s\in\mathbb{R}^{d},\\
g_{\Phi}^{c}  (f )  (s )& =   \Big(\int_{0}^{1}  |\Phi_{\varepsilon}*f  (s ) |^{2}\frac{d\varepsilon}{\varepsilon} \Big)^{\frac{1}{2}},\thinspace s\in\mathbb{R}^{d}.
\end{split}
\end{equation*}
The function $\Phi$ that we use to characterize the operator-valued local Hardy spaces satisfies the following conditions:
\begin{enumerate}[\rm(1)]
\item Every $D^m \Phi$  with $0\leq |m|_1 \leq d$ makes $f\mapsto s_{D^m \Phi}^cf$ and $f\mapsto g_{D^m \Phi}^cf$  Calder\'on-Zygmund singular integral operators in Corollary \ref{Cor C-Z};
\item There exist functions $\Psi, \psi$ and $\phi$ such that 
\begin{equation}
\widehat{\phi}  (\xi )\overline{\widehat{\psi}(\xi)}+\int_{0}^{1}\widehat{\Phi}  (\varepsilon\xi )\overline{\widehat{\Psi}  (\varepsilon\xi )}\frac{d\varepsilon}{\varepsilon}=1, \quad   \forall\, \xi\in \R;  \label{eq: basic2}
\end{equation}
\item The above $\Psi$ and $\psi$ make $d\mu_g=  |\Psi_\e*g(s)|^2\frac{d\e ds}{\e}$ and $\phi*g$ satisfy: 
\[
 \max\Big\{\big\|\underset{\substack{s\in Q\subset \R\\ |Q|<1}}{\sup{} ^+} \frac{1}{|Q|}\int_{T(Q)}
 d\mu _g \big\|_{\frac{q}{2}}^{\frac{1}{2}},\,  \|  \psi*g \|  _{q}\Big\} \lesssim  \|  g \|  _{{\bmo}_q^{c}} \quad \text{for} \,\, q> 2;
\]
\item The above $\phi$ makes $f\mapsto \phi*f$  a Calder\'on-Zygmund singular integral operator in Corollary \ref{Cor C-Z}.
\end{enumerate}

Fix the four  functions $\Phi, \Psi, \phi, \psi$ as above. The following is one of our main results in this section, which states that the functions $\Phi, \phi$ satisfying the above four conditions give a general characterization for $\h_p^c(\R,\M)$.

\begin{thm}\label{thm main1}
Let $1\leq p<\infty$ and $\phi$, $\Phi$ be as above. For any $f\in L_{1}  (\mathcal{M};\mathrm R_{d}^{c} )+L_{\infty}  (\mathcal{M};\mathrm R_{d}^{c} )$, $f\in \h_{p}^{c}  (\mathbb{R}^{d},\mathcal{M} )$  if and only if
$s_{\Phi}^{c}  (f )\in L_{p}  (\mathcal{N} )$ and
$\phi*f\in L_{p}  (\mathcal{N} )$  if and only if $g_{\Phi}^{c}  (f )\in L_{p}  (\mathcal{N} )$
and $\phi*f\in L_{p}  (\mathcal{N} )$. If this is the
case, then
\begin{equation}
 \|  f \|  _{\h_{p}^{c}}\approx  \|  s_{\Phi}^{c}  (f )  \|  _{p}+  \|  \phi*f \|  _{p}\approx  \|  g_{\Phi}^{c}  (f )  \|  _{p}+  \|  \phi*f \|  _{p} \label{eq: main}
\end{equation}
 with the relevant constants depending only on $d, p$, and the pairs  $(\Phi,\Psi)$ and $  (\phi, \psi)$.
\end{thm}
One implication of the above theorem is an easy consequence of conditions (1) and (4) that 
\begin{equation}\label{eq: s<h1} 
 \|  s_{\Phi}^{c}  (f )  \|  _{p}+  \|  \phi*f \|  _{p}\lesssim  \|  f \|  _{\h_{p}^{c}}
\end{equation}
\begin{equation}\label{eq: g<h1} 
 \|  g_{\Phi}^{c}  (f )  \|  _{p}+  \|  \phi*f \|  _{p} \lesssim  \|  f \|  _{\h_{p}^{c}}.
\end{equation}
In order to prove the converse inequalities, we need the following lemma, which can be seen as a generalization of Lemma \ref{lem: duality}.

\begin{lem}\label{lem: main1}
Let $1\leq p< 2$, $q$ be its conjugate index and $\Phi$, $\phi$ be the functions satisfying the above assumption. For $f\in \h_{p}^{c}  (\mathbb{R}^{d},\mathcal{M} )\cap L_{2}  (\mathcal{N} )$
and $g\in{\bmo}_q^{c}  (\mathbb{R}^{d},\mathcal{M} )$,
\[
  \big|\tau\int_{\mathbb{R}^{d}}f  (s )g^{*}  (s )ds \big|\lesssim  (  \|  s_{\Phi}^{c}  (f ) \|  _{p}+  \|  \phi*f \|  _{p} )  \|  g \|  _{{\bmo}_q^{c}}.
\]
\end{lem}
\begin{proof}
The proof of this Lemma is very similar to that of Lemma \ref{lem: duality}, we will just point out the necessary modifications to avoid duplication. We need two auxiliary square functions associated with ${\Phi}$. For $s\in \R$, $\e\in [0,1]$, we define
\begin{equation}
s_{{\Phi}}^{c}  (f )  (s,\varepsilon )=  \Big(\int_{\varepsilon}^{1}\int_{B  (s,r-\frac{\e}{2} )}  |{\Phi}_{r}*f  (t) |^{2}\frac{dtdr}{r^{d+1}} \Big)^{\frac{1}{2}},\label{eq: auxiliary s}
\end{equation}
\begin{equation}
\overline{s}_{\Phi}^{c}  (f )  (s,\varepsilon )=  \Big(\int_{\varepsilon}^{1}\int_{B  (s,\frac{r}{2} )}  |\Phi_r*f  (t ) |^{2}\frac{dtdr}{r^{d+1}} \Big)^{\frac{1}{2}}.
\end{equation}
By assumption (2) of $\Phi$,  we have
\begin{equation*}
\begin{split}
&\tau\int_{\mathbb{R}^{d}}f  (s )g^{*}  (s )ds \\
&    =   \tau\int_{\mathbb{R}^{d}}\int_{0}^{1}\Phi_{\varepsilon}*f  (s )  ({\Psi}_{\varepsilon}*g  (s ) )^{*}\frac{dsd\varepsilon}{\varepsilon}+\tau\int_{\mathbb{R}^{d}}\phi*f  (s )  (\psi*g  (s ) )^{*}ds\\
& = \frac{2^{d}}{c_{d}}\tau\int_{\mathbb{R}^{d}}\int_{0}^{1}\int_{B  (s,\frac{\varepsilon}{2} )}\Phi_{\varepsilon}*f  (t )s_{\Phi}^{c}  (f )  (s,\varepsilon )^{\frac{p-2}{2}} s_{\Phi}^{c}(f )(s,\varepsilon )^{\frac{2-p}{2}} (\Psi_{\varepsilon}*g  (t ) )^{*}\frac{dtd\varepsilon}{\varepsilon^{d+1}}ds\\
 &  \;\;\;\;  +\tau\int_{\mathbb{R}^{d}}\phi*f  (s )  (\psi*g  (s ) )^{*}ds\\
&   \stackrel{\mathrm{def}}{=}  \mathrm{ I}+\mathrm{II}.
 \end{split}
\end{equation*}
 Then by the Cauchy-Schwarz inequality,
\begin{equation*}
\begin{split}
| \mathrm{I}| ^{2}  & \lesssim  \tau\int_{\mathbb{R}^{d}}\int_{0}^{1}  \big(\int_{B  (s,\frac{\varepsilon}{2} )}  |\Phi_{\varepsilon}*f  (t ) |^{2}\frac{dt}{\varepsilon^{d+1}} \big)\overline{s}_{\Phi}^{c}  (f )  (s,\varepsilon )^{p-2}d\varepsilon ds\\
   &\;\;\;\; \cdot\tau\int_{\mathbb{R}^{d}}\int_{0}^{1}  \big(\int_{B  (s,\frac{\varepsilon}{2} )}  |\Psi_{\varepsilon}*g  (t ) |^{2}\frac{dt}{\varepsilon^{d+1}} \big)s_{{\Phi}}^{c}  (f )  (s,\varepsilon )^{2-p}d\varepsilon ds\\
 &   \stackrel{\mathrm{def}}{=} A\cdot B.
\end{split}
\end{equation*}
 Replacing $\e \frac{\partial}{\partial \e}\mathrm{P}_\e(f)$ and $\e \frac{\partial}{\partial \e}\mathrm{P}_\e(g)$ in the proof of Lemma \ref{lem: duality} by $\Phi_\e*f$ and $\Psi_\e*g$ respectively and applying Lemma \ref {lem: Carleson Phi} and assumption (3) of $\Psi$ and $\psi$, we get  the estimates for the terms $A$ and $B$ that 
\begin{equation*}
A  \lesssim   \|  s_{\Phi}^{c}  (f ) \|  _{p}^p \quad \text{and}\quad B \lesssim   \|  g \|  _{{\bmo}_q^{c}}^{2}  \|  s_{\Phi}^{c}  (f ) \|  _{p}^{2-p}.
\end{equation*}
The term $\mathrm{II}$ is easy to deal with. By the H\"older inequality,  Lemma  \ref {lem: Carleson Phi} and assumption (3) again, we get
 \begin{equation*}
\Big| \tau\int_{\mathbb{R}^{d}}\phi*f  (s )  (\psi*g  (s ) )^{*}ds\Big|
 \leq   \| \phi* f \|  _{p}  \| \psi*g \|  _{q}\lesssim   \| \phi* f \|  _{p}  \|  g \| _{\bmo_q^c}.
 \end{equation*}
Combining the estimates for $A$, $B$ and $\mathrm{II}$, we finally get the
desired inequality.
\end{proof}

We also need the radial version of Lemma \ref{lem: main1}. To
this end, we need to majorize the radial square function by the conic
one. When we consider the Poisson kernel,
this result follows from the harmonicity of the Poisson integral (see Lemma \ref{lem:poison g Lusin compare}). However, in the general
case, the harmonicity is no longer available. To overcome this difficulty, a more sophisticated inequality has been developped in \cite{XXX17} to compare non-local radial and conic functions.
Observe that the result given in \cite[Lemma 4.3]{XXX17} is a pointwise
one, which also works  for the local version of square functions if we consider integration over the interval $0<\varepsilon<1$.
The following lemma is an obvious consequence of \cite[Lemma 4.3] {XXX17}.

\begin{lem}\label{lemma 4.3 of XXX}
Let $f\in L_{1}  (\mathcal{M}; \mathrm R_{d}^{c} )+L_{\infty}  (\mathcal{M}; \mathrm R_{d}^{c} )$.
Then
\[
g_{\Phi}^{c}  (f )  (s )^{2}\lesssim\sum_{  |m |_{1}\leq d}s_{D^{m}\Phi}^{c}  (f )  (s )^{2},\,  \forall\, s\in\mathbb{R}^{d}.
\]

\begin{lem}
\label{lem:main2}Let $1\leq p< 2$. For $f\in \h_{p}^{c}  (\mathbb{R}^{d},\mathcal{M} )\cap L_{2}  (\mathcal{N} )$
and $g\in{\bmo}_q^{c}  (\mathbb{R}^{d},\mathcal{M} )$, 
\[
 \Big |\tau\int_{\mathbb{R}^{d}}f  (s )g^{*}  (s )ds   \Big |  \lesssim  (  \|  g_{\Phi}^{c}  (f ) \|  _{p}+  \|  \phi*f \|  _{p} )^{\frac{p}{2}}  \|  f \|  _{\h_{p}^{c}}^{1-\frac{p}{2}}  \|  g \|  _{{\bmo}_q^{c}}.
\]

\end{lem}
\end{lem}
\begin{proof}
This proof is similar to that of Lemma \ref{lem: main1} and we keep the notation there. Let $f\in \h_{p}^{c}  (\mathbb{R}^{d},\mathcal{M} )$ with
compact support (relative to the variable of $\mathbb{R}^{d}$). We
assume that $f$ is sufficiently nice so that all calculations below
are legitimate. Now we need the radial version of $s_{\Phi}^c(f)(s,\e)$,
\[
g_{\Phi}^{c}  (f )  (s,\varepsilon )=  \big(\int_{\varepsilon}^{1}  |\Phi_{r}*f  (s ) |^{2}\frac{dr}{r} \big)^{\frac{1}{2}}
\]
 for $s\in\mathbb{R}^{d}$ and $0\leq \varepsilon\leq 1$. By approximation,
we can assume that $g_{\Phi}^{c}  (f )  (s,\varepsilon )$
is invertible for every $  (s,\varepsilon )\in S$. By \eqref{eq: basic2},
\eqref{eq: Planchel-1} and the Fubini theorem, we have
\begin{equation*}
\begin{split}
 &  \Big |\tau\int_{\mathbb{R}^{d}}f  (s )g^{*}  (s )ds  \Big |^{2}\\
 & \lesssim  \tau\int_{\mathbb{R}^{d}}\int_{0}^{1}  |\Phi_{\varepsilon}*f  (s ) |^{2}g_{\Phi}^{c}  (f )  (s,\varepsilon )^{p-2}\frac{d\varepsilon ds}{\varepsilon}\cdot\tau\int_{\mathbb{R}^{d}}\int_{0}^{1}  |\Psi_{\varepsilon}*g  (s ) |^{2}g_{\Phi}^{c}  (f )  (s,\varepsilon )^{2-p}\frac{d\varepsilon ds}{\varepsilon}\\
 &  \;\;\;\; +   \Big |\tau\int_{\mathbb{R}^{d}}\phi*f  (s )  (\psi*g  (s ) )^{*}ds  \Big |^{2}\\
 & \stackrel{\mathrm{def}}{=}  A^{\prime}B^{\prime}+\mathrm{II}^{\prime}.
\end{split}
\end{equation*}
$\mathrm{II}^{\prime}$
is treated exactly in the same way as before,
$$\mathrm{II}^{\prime}\lesssim  \|  \phi*f \|  _{p}^{2}  \|  \psi*g \|  _{q}^{2} \lesssim   \|  \phi*f \|  _{p}^{p} \|  f\|  _{\h_p^c}^{2-p} \|  g\| _{\bmo_q^c}^2.$$
$A^{\prime}$ is also estimated similarly as in Lemma \ref{lem: main1}, we have
$ A^{\prime}\lesssim \|  g_{\Phi}^{c}  (f ) \|  _{p}^p$.

To estimate $B^{\prime}$, we notice that the proof of \cite[Lemma 1.3]{XXX17}
also gives
\[
g_{\Phi}^{c}  (f )  (s,\varepsilon )^{2}\lesssim\sum_{  |m |_{1}\leq d}s_{D^{m}\Phi}^{c}  (f )  (s,\varepsilon )^{2},
\]
where $s_{D^{m}{\Phi}}^{c}  (f )  (s,\varepsilon )$ is defined by \eqref{eq: auxiliary s} with $D^{m}\Phi$ instead of $\Phi$. Then by the above inequality, Lemma \ref{lem: Carleson Phi} and inequality \eqref{eq: s<h1} with $D^{m}\Phi$ instead of $\Phi$, we obtain
\begin{equation*}
\begin{split}
B^{\prime}  & \lesssim  \sum_{  |m |_{1}\leq d}\tau\int_{\mathbb{R}^{d}}\int_{0}^{1}  |\Psi_{\varepsilon}*g  (s ) |^{2}s_{D^{m}\Phi}^{c}  (f )  (s,\varepsilon )^{2-p}\frac{d\varepsilon ds}{\varepsilon}\\
&   \lesssim   \sum_{  |m |_{1}\leq d}  \|  g \|  _{{\bmo}_q^{c}}^{2}  \|  s_{D^{m}\Phi}^{c}  (f ) \|  _{p}^{2-p}\\
&   \lesssim   \|  g \|  _{{\bmo}_q^{c}}^{2}  \|  f \|  _{\h_p^c}^{2-p}.
\end{split}
\end{equation*}
Therefore,
$$  |\tau\int_{\mathbb{R}^{d}}f  (s )g^{*}  (s )ds |^2 \lesssim   (  \|  g_{\Phi}^{c}  (f ) \|  _{p}  +\|  \phi*f\| _{p}  )^p  \|  f \| ^{2-p} _{\h_p^c}\|  g\| _{\bmo_q^c}^2 ,$$
which completes the proof.
\end{proof}

\begin{proof}[Proof of Theorem \ref{thm main1}]
From Lemmas \ref{lem: main1}, \ref{lem:main2} and Theorem  \ref{thm: duality P}, we conclude that if $1\leq p\leq 2$, we have
\begin{equation*}
\begin{split}
  \|  f \|  _{\h_{p}^{c}} &  \lesssim   \|  s_{\Phi}^{c}  (f ) \|  _{p}+  \|  \phi*f \|  _{p},\\
  \|  f \|  _{\h_{p}^{c}} &  \lesssim  \|  g_{\Phi}^{c}  (f ) \|  _{p}+  \|  \phi*f \|  _{p}.
  \end{split}
\end{equation*}
For the case $2<p<\infty$, by Theorem \ref{cor: duality}, we can choose $g\in \h_q^c(\R,\M)$ (with $q$ the conjugate index of $p$) with norm one such that 
\[
\|f\|_{\h_p^c}\approx \tau \int_{\R}f(s)g^*(s)ds =\tau \int_{\R}\int_0^1\Phi_\e*f(s)\cdot (\Psi_\e *g(s))^*\frac{ds d\e}{\e}+\tau \int_{\R}\phi*f(s)(\psi*g(s))^*ds.
\]
Then by the H\"older inequality and \eqref{eq: g<h1} (applied to $g,\Psi$ and $q$),

\begin{equation*}
\begin{split}
\| f\|_{\h_p^c} & \lesssim   \| g_\Phi^c(f)\|_p \|g_\Psi^c(g)\|_q+\|\phi*f\|_p \|\psi*g\|_q \\
&  \lesssim   ( \| g_\Phi^c(f)\|_p+\|\phi*f\|_p) \| g \|_{\h_q^c} = \| g_\Phi^c(f)\|_p+\|\phi*f\|_p.
\end{split}
\end{equation*}
Similarly, we have
\[
\| f\|_{\h_p^c} \lesssim \| s_\Phi^c(f)\|_p+\|\phi*f\|_p.
\]
Therefore, combined with \eqref{eq: g<h1} and  \eqref{eq: s<h1}, we have  proved the  assertion.
\end{proof}

The rest part of this subsection is devoted to explaining how Theorem \ref{thm main1} generalizes the characterization of $\h_p^c(\R,\M)$. 

Firstly and most naturally, we show how Theorem \ref{thm main1} covers the original definition of $\h_p^c(\R,\M)$. Let us take $\Phi=-2\pi I(\mathrm{P})$ and $\phi=\mathrm{P}$ for example. A simple calculation shows that we can choose $\Psi=-8\pi I(\mathrm{P})$ and $\psi= 4\pi I(\mathrm{P})+\mathrm{P}$ to fulfil \eqref{eq: basic2}.  By the inverse Fourier transform formula, we have
\begin{equation*}
\begin{split}
-2\pi f*I(\mathrm{P})_{\e}(t)& = -2\pi\int e^{2\pi {\rm i}t\cdot\xi}\widehat{f}(\xi)|\e\xi |e^{-2\pi \e |\xi |}d\xi\\
& =  \e\frac{\partial}{\partial\e}\int e^{2\pi {\rm i}t\cdot\xi}\widehat{f}(\xi)e^{-2\pi \e |\xi |}d\xi= \e\frac{\partial}{\partial\e}\mathrm{P}_\e(f)(t).
\end{split}
\end{equation*}
So we return back to the original definition of $\h_p^c(\R,\M)$. Theorem \ref{thm main1} implies that
\[
\|  f\| _{\h_p^c}\approx   \|  s_{\Phi}^{c}  (f ) \|  _{p}+  \|  \phi*f \|  _{p}
\approx   \|  g_{\Phi}^{c}  (f ) \|  _{p}+  \|  \phi*f \|  _{p}.
\]
In particular, we have the following equivalent norm of $\h_p^c(\R,\M)$:
\begin{cor}\label{cor: s equi g}
Let $1\leq p<\infty$. Then for any $f\in \h_{p}^{c}  (\mathbb{R}^{d},\mathcal{M} )$, we have
$$ \|f\|_{\h_p^c} \approx  \|  g^{c}  (f )\| _{p} + \|\mathrm{P}*f\|_p.$$
\end{cor}

Secondly, consider  $\Phi$ to be a Schwartz function on $\R$ satisfying:
\begin{equation}\label{eq: Phi condition}
\begin{cases}
\Phi \text{ is of vanishing mean};\\
\Phi \text{ is nondegenerate in the sense of } \eqref{eq: nondegenerate}.
\end{cases}
\end{equation}
Set $\Phi_{\varepsilon}(s)=\varepsilon^{-d}\Phi(\frac{s}{\varepsilon})$
for $\varepsilon>0$.  In the sequel, we will show that every Schwartz function satisfying \eqref{eq: Phi condition} fulfils the four conditions in the beginning of this subsection. So they all can be used to characterize $\h_p(\R,\M)$.

It is a well-known elementary fact (ef. e.g. \cite[p. 186]{Stein1993}) that there exists a Schwartz function $\Psi$ of vanishing mean such that
\begin{equation}\label{eq: basic}
\int_{0}^{\infty}\widehat{\Phi}(\varepsilon\xi)\overline{\widehat{\Psi}(\varepsilon\xi)}\frac{d\varepsilon}{\varepsilon}=1,\,  \forall\,\xi\in\mathbb{R}^{d}\setminus\left\{ 0\right\} .
\end{equation}

\begin{lem}\label{lem: Schwartz}
$\int_{0}^{1}\widehat{\Phi}  (\varepsilon\cdot )\overline{\widehat{\Psi}  (\varepsilon\cdot )}\frac{d\varepsilon}{\varepsilon}$  is an infinitely differentiable function on $\R$ if we define its value at the origin as $0$.
\end{lem}
\begin{proof}
To prove the assertion, it suffices to show that $\int_{0}^{1}\widehat{\Phi}  (\varepsilon\cdot )\overline{\widehat{\Psi}  (\varepsilon\cdot )}\frac{d\varepsilon}{\varepsilon}$ is infinitely differentiable at the origin. Given $\e \in (0,1]$, we expand $\widehat{\Phi}  (\varepsilon\cdot )$ in the Taylor series at the origin
\[
\widehat{\Phi}  (\varepsilon\xi )=\sum_{|\gamma |_1\leq N}D^{\gamma}\widehat{\Phi}(0)\frac{\e^{|\gamma|_1}\xi^{\gamma}}{\gamma !}+ \sum_{|\gamma |_1= N+1}R_{\gamma}(\e\xi) \,\xi^{\gamma},
\]
with the remainder of integral form equal to 
$$R_{\gamma}(\varepsilon\xi)  = \frac{ (N+1) \varepsilon^{N+1}}{\gamma !} \int_0^1   (1-\theta) ^{N} D^\gamma\widehat{\Phi}(\theta  \varepsilon  \xi) d\theta \,.$$
Since $\widehat{\Phi}(0)=0$, the above Taylor series implies that
\[
\widehat{\Phi}  (\varepsilon\xi )=\sum_{1\leq |\gamma|_1\leq N}D^{\gamma}\widehat{\Phi}(0)\frac{\e^{|\gamma|_1}\xi^{\gamma}}{\gamma !}+ \sum_{|\gamma |_1= N+1}R_{\gamma}(\varepsilon \xi) \,\xi^{\gamma}.
\]
Similarly, we have
\[
\widehat{\Psi}  (\varepsilon\xi )=\sum_{1\leq |\beta |_1 \leq N}D^{\beta}\widehat{\Psi}(0)\frac{\e^{|\beta|_1}\xi^{\beta}}{\beta !}+ \sum_{|\beta |_1= N+1}R'_{\beta}(\varepsilon\xi) \,\xi^{\beta},
\]
where $R'_\beta$ is the integral form remainder of $\widehat{\Psi}$.
Thus, both $\widehat{\Phi}  (\varepsilon\xi )$ and $\widehat{\Psi}  (\varepsilon\xi )$ contain only powers of $\varepsilon$ with order at least $1$.  Therefore, the integral $\int_{0}^{1}\widehat{\Phi}  (\varepsilon\xi )\overline{\widehat{\Psi}  (\varepsilon\xi )}\frac{d\varepsilon}{\varepsilon}$ (and the integrals of arbitrary order derivatives of $\widehat{\Phi}  (\varepsilon\xi )$ and $\widehat{\Psi}  (\varepsilon\xi )$) converge uniformly for $\xi\in \R$ close to the origin. We then obtain that $\int_{0}^{1}\widehat{\Phi}  (\varepsilon\xi )\overline{\widehat{\Psi}  (\varepsilon\xi )}\frac{d\varepsilon}{\varepsilon}$ is infinitely differentiable at the origin $\xi=0$.
\end{proof}

It follows immediately from Lemma \ref{lem: Schwartz} that $\int_{1}^{\infty}\widehat{\Phi}  (\varepsilon\cdot )\overline{\widehat{\Psi}  (\varepsilon\cdot )}\frac{d\varepsilon}{\varepsilon}$ is a Schwartz function if we define its value at the origin by $1$. Then we can find two other functions $\phi$, $\psi$ such that  $\widehat{\phi}, \widehat{\psi}\in H_2^\sigma(\R)$, $\widehat{\phi}(0)>0, \widehat{\psi}(0)> 0$ and 
\begin{equation}
\widehat{\phi}  (\xi )\overline{\widehat{\psi}(\xi)}+\int_{0}^{1}\widehat{\Phi}  (\varepsilon\xi )\overline{\widehat{\Psi}  (\varepsilon\xi )}\frac{d\varepsilon}{\varepsilon}=1, \quad   \forall\, \xi\in \R;  \label{eq: basic2'}
\end{equation}
Indeed, for $\beta >0$ large enough, the function $(1+|\cdot |^2)^{-\beta}$ belongs to $H_2^\sigma(\R)$. On the other hand, if $F\in \mathcal{S}(\R)$, the function $(1+|\cdot |^2)^{\beta}F$ is still in $H_2^\sigma(\R)$. Thus we obtain \eqref{eq: basic2}.

Now let show that conditions (1) and (4) hold for $\Phi$, $\phi$ satisfying \eqref{eq: Phi condition}. First, we deal with the case $1\leq p\leq 2$.  Let $H=L_{2}  (  (0,1 ),\frac{d\varepsilon}{\varepsilon} )$.
Define the kernel $\fk:\mathbb{R}^{d}\rightarrow H$ by $\fk  (s )=\Phi_{\cdot}  (s )$
with $\Phi_{\cdot}  (s ):\varepsilon\mapsto\Phi_{\varepsilon}  (s )$.
Then we can check that
\[
\sup_{\xi\in\mathbb{R}^{d}}  \|  \widehat{\Phi}  (\varepsilon\xi ) \|  _{H}<\infty,\;\;\thinspace  \|  \Phi_{\varepsilon}  (s ) \|  _{H}\lesssim\frac{1}{  |s |^{d+1}},\thinspace\;  \forall\, s\in\mathbb{R}^{d}\setminus  \{ 0 \} \]
 and that
  \[  \|  \nabla\Phi_{\varepsilon}  (s ) \|  _{H}\lesssim\frac{1}{  |s |^{d+1}}, \thinspace\;  \forall\, s\in\mathbb{R}^{d}\setminus  \{ 0 \}.
\]
Thus, $\fk $ satisfies the
assumption of Corollary \ref{Cor C-Z}. By Remark \ref{rem: L_p H h_p}, we have, for any $1\leq p\leq 2$,
\begin{equation*}
  \|  \Phi_{\e}*f \|  _{L_{p}  (\mathcal{N};H^{c} )}  =   \|  g_{\Phi}^{c}  (f ) \|  _{p}\lesssim   \|  f \|  _{\h_{p}^{c}}.
\end{equation*}
The treatment of $s_{\Phi}^{c}$ is similar. In this case, we take the Hilbert
space $H=L_{2}  (\widetilde{\Gamma},\frac{dtd\varepsilon}{\varepsilon^{d+1}} )$.
On the other hand, $\widehat{\phi}\in H_2^\sigma(\R)$ implies $\phi\in L_1(\R)$, then $\|\phi* f\|_{L_p(\mathcal{N})}\lesssim \| f\|_{L_p(\mathcal{N})}\lesssim \|f\|_{\h^c_p}$.
Thus, combining the above estimates, we obtain
\begin{equation*}
  \|  g_{\Phi}^{c}  (f ) \|  _{p}+  \|  \phi*f \|  _{p}  \lesssim  \|  f \|  _{\h_{p}^{c} }
  \end{equation*}
  \begin{equation*}
 \|  s_{\Phi}^{c}  (f ) \|  _{p}+  \|  \phi*f \|  _{p} \lesssim  \|  f \|  _{\h_{p}^{c}}. 
  \end{equation*}
Then, a simple duality argument using \eqref{eq: basic2} and  Theorem \ref{cor: duality} gives the above inequalities for the case $p>2$. Moreover, it is obvious that if we replace $\Phi$ by $D^m \Phi$, the above two inequalities still hold for any $1\leq p<\infty$.

In the end, it remains to check the condition (3) for $\Psi$, $\psi$ obtained in \eqref{eq: basic} and \eqref{eq: basic2'}. This can be done by showing a Carleson measure characterization of $\bmo_q^c$ by general test functions. The proof of the following lemma has the same pattern with that of Lemma \ref{lem: general Carleson}, so is left to the reader.

\begin{lem}\label{lem: Carleson Phi}
Let $2<q\leq \infty $, $g\in{\bmo}_q^{c}  (\mathbb{R}^{d},\mathcal{M} )$ and
$d\mu_{g}=  |\Psi_{\varepsilon}*g  (s ) |^{2}\frac{dsd\varepsilon}{\varepsilon}$. Then $d\mu_{g}$
is an $\mathcal{M}$-valued $q$-Carleson measure on the strip $\mathbb{R}^d\times (0,1)$. Furthermore, let $\psi$ be any function on $\R$ such that
\begin{equation}\label{eq: psi}
\widehat{\psi}\in H_2^\sigma(\R) \quad \text{with} \quad \sigma>\frac{d}{2}.
\end{equation}
We have
\[
 \max\Big\{\big\|\underset{\substack{s\in Q\subset \R\\ |Q|<1}}{\sup{} ^+} \frac{1}{|Q|}\int_{T(Q)}
 d\mu _g \big\|_{\frac{q}{2}}^{\frac{1}{2}},\,  \|  \psi*g \|  _{q}\Big\} \lesssim  \|  g \|  _{{\bmo}_q^{c}}.
\]
\end{lem}

\begin{rmk}\label{rmk: Carleson Phi}
It is worthwhile to note that, if $\Psi$ and $\psi$ are determined by \eqref{eq: basic} and \eqref{eq: basic2'}, the opposite of the above lemma is also true. This can be deduced by a similar argument as that of Corollary \ref{cor: Carleson=Dbmo}; we omit the details. 
\end{rmk}

By the discussion above, we deduce the following corollary from Theorem \ref{thm main1}.
\begin{cor}
Let $\Phi$ be the Schwartz function on $\R$ satisfying \eqref{eq: Phi condition} and $\phi$ be the function given by \eqref{eq: basic2'}. Then for any $1\leq p<\infty$, we have 
\begin{equation}
 \|  f \|  _{\h_{p}^{c}}\approx  \|  s_{\Phi}^{c}  (f )  \|  _{p}+  \|  \phi*f \|  _{p}\approx  \|  g_{\Phi}^{c}  (f )  \|  _{p}+  \|  \phi*f \|  _{p} \label{eq: main}
\end{equation}
 with the relevant constants depending only on $d, p, \Phi$ and $\phi$.
 
 \end{cor}

\subsection{Discrete characterizations}

In this subsection, we give a discrete characterization for operator-valued local Hardy spaces. To this end, we need some modifications of the four conditions in the beginning of last subsection.
The square functions $s_{\Phi}^c(f)$ and $g_{\Phi}^c(f)$ can be discretized as follows:
\begin{equation*}
\begin{split}
 g_{\Phi}^{c,D}(f)(s)  &= \Big(\sum_{j\geq 1} |\Phi_j*f (s)|^2\Big)^{\frac 1 2},\\
 s_{\Phi}^{c,D}(f)(s) &= \Big(\sum_{j\geq 1} 2^{dj}\int_{B(s, 2^{-j})} |\Phi_j*f (t)|^2 dt\Big)^{\frac 1 2}.
\end{split}
\end{equation*}
Here $\Phi_j$ is the inverse Fourier transform of $\Phi(2^{-j}\cdot)$. This time, to get a resolvent of the unit on $\R$, we need to assume that $\Phi,\Psi, \phi, \psi$ satisfy
 \begin{equation}\label{eq: reproduceD 2}
  \sum_{j=1}^\infty\widehat{\Phi} (2^{-j}\xi)\, \overline{\widehat{\Psi} (2^{-j}\xi)}+\widehat{\phi}(\xi)\overline{\widehat{\psi}(\xi)}=1, \quad   \forall\, \xi \in \R.
 \end{equation}
In brief, the complex-valued infinitely differentiable function $\Phi$ considered in this subsection satisfies:
\begin{enumerate}[\rm(1)]
\item Every $D^m \Phi$  with $0\leq |m|_1 \leq d$ makes $f\mapsto s_{D^m \Phi}^{c,D}f$ and $f\mapsto g_{D^m \Phi}^{c,D}f$  Calder\'on-Zygmund singular integral operators in Corollary \ref{Cor C-Z};
\item There exist functions $\Psi, \psi$ and $\phi$ that fulfil \eqref{eq: reproduceD 2};
\item The above $\Psi$ and $\psi$ make $d\mu_f^D = \sum_{j\geq 1}  |\Psi_{j}*f(s)|^2 ds \times d\delta_{2^{-j}}(\e)$ (with $\delta_{2^{-j}}(\e)$ the unit Dirac mass at the point $2^{-j}$) and $\phi*f$ satisfy: 
\[
 \max\Big\{\big\|\underset{\substack{s\in Q\subset \R\\ |Q|<1}}{\sup{} ^+} \frac{1}{|Q|}\int_{T(Q)}
d\mu_f^D \big\|_{\frac{q}{2}}^{\frac{1}{2}},\,  \|  \psi*f \|  _{q}\Big\} \lesssim  \|  f \|  _{{\bmo}_q^{c}} \quad \text{for} \,\, q> 2;
\]
\item The above $\phi$ makes $f\mapsto \phi*f$  a Calder\'on-Zygmund singular integral operator in Corollary \ref{Cor C-Z}.
\end{enumerate}

\begin{rmk}
 Any  Schwartz function that has vanishing mean and is nondegenerate in the sense of
\eqref{eq: nondegenerate} satisfies all the four conditions above.
 \end{rmk}

  The following discrete version of Theorem \ref{thm main1} will play a crucial role in the study of operator-valued Triebel-Lizorkin spaces on $\R$ in our forthcoming paper \cite{XX18}.  Now we fix the pairs $(\Phi, \Psi)$ and $(\phi,\psi)$ satisfying  the above four conditions.

\begin{thm}\label{thm: equivalence hpD}
Let  $1\leq p<\infty$. Then for any $f\in L_{1}  (\mathcal{M};\mathrm R_{d}^{c} )+L_{\infty}  (\mathcal{M};\mathrm R_{d}^{c} )$,
 $f\in \h_p^c(\R,\M)$  if and only if  $s_{\Phi}^{c, D}(f)\in L_{p}(\N)$ and $\phi *f\in L_{p}(\N)$   if and only if $g_{\Phi}^{c, D}(f)\in L_{p}(\N)$ and $\phi *f\in L_{p}(\N)$. Moreover,
 $$
 \|f\|_{\h^c_p}\approx \|s_{\Phi}^{c, D}(f)\|_{p}+\|\phi *f\|_{p} \approx \|g_{\Phi}^{c, D}(f)\|_{p}+\|\phi *f\|_{p} $$
with the relevant constants depending only on $ d, p$, and the pairs $(\Phi, \Psi)$ and $(\phi ,\psi)$.
\end{thm}

The following paragraphs are devoted to the proof of Theorem \ref{thm: equivalence hpD} which is similar to that of Theorem \ref{thm main1}. We will just indicate the necessary modifications. We first prove
the discrete counterparts of Lemmas \ref{lem: main1} and \ref{lem:main2}.

\begin{lem}\label{S-dualD}
Let $1\leq p< 2$ and $q$ be the conjugate index of $p$. For any $f\in \h_p^c(\R,\M)\cap L_2(\N)$ and $g\in\bmo_q^c(\R,\M)$,
 $$\Big|\tau \int _{\R} f(s)g^*(s)ds\Big| \lesssim \big(\|s^{c,D}_\Phi(f)\|_{p}+\| \phi *f \|_{p}\big)\,\|g\|_{\bmo_q^c}\,.
 $$
 \end{lem}
\begin{proof}
First, note that by \eqref{eq: reproduceD 2}, we have
\begin{equation*}
\tau \int _{\R} f(s)g^*(s)ds =
\tau\int_{\R}\sum_{j\geq 1}\Phi_j*f(s) \big(\Psi_{j} *g(s)\big)^*\,ds+\tau\int_{\R} \phi *f(s)\big(\psi*g(s)\big)^*ds .
\end{equation*}
The second term on the right hand side of the above formula is exactly the same as the corresponding term $\mathrm {II}$  in the proof of Lemma \ref{lem: main1}.   Now we need the discrete versions of  $s_{\Phi}^{c}$ and $\overline{s}_{\Phi}^{c}$: For $j\geq 1$, $s\in \R$, let
  \begin{equation*}
\begin{split}
  s^{c,D}_\Phi(f)(s,j) &  = \Big(\sum_{1\leq k\leq j} 2^{dk}\int_{B(s,2^{-k}-2^{-j-1})} |\Phi_j*f(t)|^2 dt\Big)^{\frac12}\\
  \overline{s}^{c,D}_\Phi(f)(s,j)  &= \Big(\sum_{1\leq k\leq j}2^{dk}\int_{B(s,2^{-k-1})} |\Phi_j*f(t)|^2 dt\Big)^{\frac12}.
  \end{split}
\end{equation*}
Denote $s_{\Phi}^{c,D}(f)(s,j)$ and $\overline{s}_{\Phi}^{c,D}(f)(s,j)$ simply by $s(s,j)$ and $\overline{s}(s,j)$, respectively. By approximation, we may assume that $s(s,j)$  and $\overline{s}(s,j)$ are invertible for every $s\in\R$ and $j\geq 1$.  By the Cauchy-Schwarz inequality,
\begin{equation*}
\begin{split}
 &  \Big|\tau\int_{\R}\sum_{j\geq 1} \Phi_j*f(s)\big(\Psi_{j} *g(s)\big)^*\,ds \Big|^2\\
 & =  \Big|\frac{2^d}{c_d}\,\tau\int_{\R}\sum_j 2^{dj}\int_{B(s, 2^{-j-1})} \Phi_j*f(t)\big(\Psi_{j} *g(t)\big)^*\,dt\,ds  \Big|^2\\
 & \lesssim  \tau\int_{\R} \sum_j s(s,j)^{p-2}
 \Big( 2^{dj}\int_{B(s, 2^{-j-1}))} |\Phi_j*f(t)|^2 \,dt\Big)
  ds  \\
 &  \quad \, \,\,\cdot \tau\int_{\R} \sum_j s(s,j)^{2-p}
 \Big( 2^{dj}\int_{B(s, 2^{-j-1})} |\Psi_{j}*g (t)|^2 \,dt\Big)ds \\
 & \stackrel{\mathrm{def}}{=}  {\mathrm A} \cdot {\rm B}.
  \end{split}
\end{equation*}
The term $\mathrm A$ is less easy to estimate than the corresponding term $\mathrm A$  in the proof of Lemma \ref{lem: main1}. To deal with it we simply  set
$\overline{s}_j=\overline{s}(s,j)$ and $\overline{s}=\overline{s}(s,+\infty)\leq s^{c,D}(f)(s)$. Then
 \begin{equation*}
\begin{split}
  {\rm A} &= \tau\int_{\R} \sum_{j\geq 1} s_j^{p-2}(\overline{s}_j^2-\overline{s}_{j-1}^2)ds \\
 &\leq \tau\int_{\R} \sum_j \overline{s}_j^{p-2}(\overline{s}_j^2-\overline{s}_{j-1}^2)ds\\
 &= \tau\int_{\R} \sum_j (\overline{s}_j-\overline{s}_{j-1} )ds+\tau\int_{\R} \sum_j\overline{s}_j^{p-2}\overline{s}_{j-1}( \overline{s}_j-\overline{s}_{j-1})ds,
 \end{split}
\end{equation*}
where $\overline{s}_{0}=0$. Since $1\leq p < 2$, $\overline{s}_j^{p-1}\leq s^{p-1}$, we have
 $$ \tau\int_{\R} \sum_j \overline{s}_j^{p-1} (\overline{s}_j-\overline{s}_{j-1})ds\lesssim \tau\int_{\R}\overline{s}^p ds\leq \|s^{c,D}(f)\|^p_{p}.$$
On the other hand,
\[
\tau\int_{\R} \sum_j\overline{s}_j^{p-2}\overline{s}_{j-1}(\overline{s}_j-\overline{s}_{j-1})ds=\tau\int_{\R} \sum_j
\overline{s}^{\frac{1-p}{2}} \overline{s}_j^{p-2}\overline{s}_{j-1}\overline{s}^{\frac{1-p}{2}}\overline{s}^{\frac{p-1}{2}}(\overline{s}_j-\overline{s}_{j-1})\overline{s}^{\frac{p-1}{2}}ds,
\]
since $\overline{s}_j\geq \overline{s}_{j-1}$ for any $j\geq 1$, we have  $\overline{s}^{\frac{1-p}{2}} \overline{s}_j^{p-2}\overline{s}_{j-1}\overline{s}^{\frac{1-p}{2}}\leq 1$. Thus, by the H\"older inequality,
 $$
 \tau\int_{\R} \sum_j\overline{s}_j^{p-2}\overline{s}_{j-1}(\overline{s}_j-\overline{s}_{j-1})ds
 \leq\tau\int_{\R} \sum_j
\overline{s}^{\frac{p-1}{2}}(\overline{s}_j-\overline{s}_{j-1})\overline{s}^{\frac{p-1}{2}} ds
 =\tau\int_{\R} \overline{s}^p ds \leq \|s^{c,D}_\Phi(f)\|_p^p.
 $$
Combining the preceding inequalities, we get the desired estimate of A:
 $${\rm A}\leq 2\|s_\Phi^{c,D}(f)\|_{p}^p.$$

The estimate of the term $\rm B$ is, however,  almost identical to that of  $\rm B$ in the proof of Lemma \ref{lem: main1}. There are only two minor differences. The first one concerns the square function  $\mathbb S ^c(f)(s, j)$ in \eqref{eq: square function}: it is now replaced by
$$
 \mathbb S ^c(f)(s,j)=\Big(\sum_{1\leq k\leq j}2^{dk}\int_{B(c_{m,j}, 2^{-k})} |\Phi_j*f(t)|^2 \,dt\Big)^{\frac12}\;\;\textrm{ if }\;\;
 s\in Q_{m,j} . 
 $$
  Then we have $s(s, j)\le \mathbb S ^c(f)(s,j)$. The second difference is about the Carleson characterization of $\bmo_q^c$; we now use its discrete analogue, namely, $d\mu_g^D$. Apart from these two differences, the remainder of the argument is identical to that in the proof of Lemma \ref{lem: main1}.
 \end{proof}

 \begin{lem}\label{g-dualD}
  Let $1\leq p<2$ and $f\in \h^c_p(\R,\M)\cap L_2(\N)$,  $g\in \bmo_q^c(\R,\M)$. Then
 $$
 \Big|\tau \int _{\R} f(s)g^*(s)ds \Big|\lesssim \Big(\|g^{c,D}_\Phi(f)\|_p+\|\phi *f\|_p\Big)^{\frac{p}{2}}\,\|f\|_{\h^c_p}^{1-\frac p 2}\,\|g\|_{\bmo_q^c}\,.
 $$
 \end{lem}

\begin{proof}
 We use the truncated version of $g^{c,D}_\Phi(f)$:
 $$
 g^{c,D}_{\Phi}(f)(s,j) =\Big( \sum_{k\leq j} |\Phi_{k}*f (s)|^2\Big)^{\frac12}\,.
 $$
The proof of \cite[Lemma 4.3] {XXX17} is easily adapted to the present setting to ensure
 $$
 g^{c,D}_{\Phi}(f)(s,j)^2\lesssim \sum_{ |m|_1\leq d} s^{c,D}_{D^m {\Phi}}(f)(s, j)^2\,.
 $$
Then
  $$\Big|\tau \int _{\R} f(s)g^*(s)ds \Big| ^2\le {\rm I'} \cdot {\rm II'}+ \Big|\tau\int \phi *f(s)\big(\psi*g(s)\big)^*ds\Big|,$$
where
 \begin{equation*}
\begin{split}
  {\rm I'}&=  \tau\int_{\R}\sum_j g_{\Phi}^{c,D}(f)(s,j) ^{p-2} |\Phi_j*f(s)|^2  ds \,,\\
 {\rm II'}&=  \tau\int_{\R} \sum_j g_{\Phi}^{c,D}(f)(s,j)^{2-p} |\Psi_{j}*g(s)|^2 ds\,.
 \end{split}
\end{equation*}
Both terms ${\rm I}' $ and ${\rm II}' $ are estimated exactly as before, so we have
  $${\rm I}' \leq 2\| g_\Phi^c(f)\|^p_p\;\textrm{ and }\;
  {\rm II}'\lesssim \|f\|^{2-p}_{\h^c_p}\,\|g\|_{\bmo_q^c}^2\,.$$
This gives the announced assertion.  
\end{proof}

Armed with the previous two lemmas and the Calder\'on-Zygmund theory in section \ref{section-CZ}, we can prove Theorem \ref{thm: equivalence hpD} in the same way as Theorem \ref{thm main1}. Details are left to the reader.

We also include a discrete Carleson measure characterization of $\bmo_q^c$ by general test functions. Much as the characterization in Lemma \ref{lem: Carleson Phi} and Remark \ref{rmk: Carleson Phi}, it is a byproduct of the proof of Theorem \ref{thm: equivalence hpD}.

\begin{cor}
Let $2<q \leq \infty$, $\psi$ and $\Psi$ be given in \eqref{eq: reproduceD 2}. Assume further
\begin{equation*}
\widehat{\psi}\in H_2^\sigma(\R) \quad \text{with} \quad \sigma>\frac{d}{2}.
\end{equation*}
 Then for every $g\in \bmo^c_q$, we have
\begin{equation*}
\|  g\|  _{\bmo ^c_q}\approx\| \psi*g\|_q+ \Big\| \underset{\substack{s\in Q\subset \R\\ |Q|<1}}{\sup{} ^+} \frac{1}{|  Q| }\int_Q\sum_{j\geq -\log_2(l(Q))} |\Psi_j*g (s)| ^2ds\Big\| _{\frac q 2}^\frac{1}{2}.
\end{equation*}
\end{cor}

\subsection{The relation between $\h_{p}  (\mathbb{R}^{d},\mathcal{M} )$ and $\mathcal{H}_{p}(\mathbb{R}^{d},\mathcal{M} )$}

Due to the noncommutativity, for any $1< p<\infty$ and $p\neq 2$, the column operator-valued local Hardy space $\h_p^c(\R,\M)$ and the column operator-valued Hardy space $\mathcal H _p^c(\R,\M)$ are not equivalent. On the other hand, if we consider the  mixture spaces $\h_{p}  (\mathbb{R}^{d},\mathcal{M} )$ and $\mathcal{H}_{p}(\mathbb{R}^{d},\mathcal{M} )$, then we will have the same situation as in the classical case.

Since $\|  \mathrm{P} *f\| _p\lesssim \|  f\| _p \lesssim\|  f\| _{\mathcal{H}_{p}^{c}}$ for any $1\leq p\leq 2$, we deduce the inclusion
\begin{equation}\label{eq: inclu p<2}
\mathcal{H}_p^c(\R,\M) \subset \h_p^c(\R,\M) \quad \text{for}\quad 1\leq p\leq 2.
\end{equation}
Then by the duality obtained in Theorem  \ref{cor: duality}, we have
\begin{equation}\label{eq: inclu p>2}
\h_p^c(\R,\M)\subset \mathcal{H}_p^c(\R,\M) \quad \text{for}\quad  2< p <\infty.
\end{equation}

However, we can see from the following proposition that we do not have the inverse inclusion of \eqref{eq: inclu p<2} nor \eqref{eq: inclu p>2}. 
\begin{prop}
 Let $\phi $ be a function on $\R$ such that $\widehat{\phi }(0)\geq 0$ and $\widehat{\phi }\in H^\sigma_2(\R)$ with $\sigma>\frac{d}{2}$. Let $2< p<\infty$. If for any  $f\in \mathcal{H}_{p}^{c}(\R,\M)$,
\begin{equation}\label{eq: varphi Hp}
\|  \phi *f\| _p\lesssim \|  f\| _{\mathcal{H}_{p}^{c}},
\end{equation}
then we must have  $\widehat{\phi }(0)= 0$.

\end{prop}

\begin{proof}
We prove the assertion by contradiction.  Suppose that there exists $\phi$ such that $\widehat{\phi }(0)> 0$, $\widehat{\phi }\in H^\sigma_2(\R)$ and \eqref{eq: varphi Hp} holds for any $f\in \mathcal H _p^c(\R,\M)$. Since both $\mathcal H _p^c(\R,\M)$ and $L_p(\N)$ are homogeneous spaces, we have, for any $\e >0$,
\[
\|\phi* f(\e \cdot)\|_p=\|(\phi_\e* f)(\e \cdot)\|_p=\e^{-\frac{d}{p}}\|\phi_\e* f\|_p
\quad \text{and}\quad \|f(\e\cdot)\|_{\mathcal H _p^c}=\e^{-\frac{d}{p}}\| f\|_{\mathcal H _p^c}.
\]
This implies that
\begin{equation}\label{eq: for all e}
\|\phi_\e* f\|_p \lesssim \| f\|_{\mathcal H _p^c},
\end{equation}
 for any $\e>0$ with the relevant constant independent of $\e$. Now we consider a function $f \in L_p(\N)$ which takes values in $\mathcal{S}_\M^+$ and such that $\supp \widehat{f}$ is compact, i.e. there exists a positive real number $N$ such that $\supp \widehat{f}\subset \{\xi\in \R: |\xi|\leq N\}$. Since $\widehat{\phi }(0)> 0$, we can find $\e_0>0$ and $c>0$ such that $\widehat{\phi}(\e _0\xi)\geq c$ whenever $|\xi|\leq N$. Thus, in this case, $\|\phi_\e* f\|_p\geq c \| f\|_p$. Then by \eqref{eq: for all e}, we have
 \[
 \| f\|_p\lesssim \| f\|_{\mathcal H _p^c},
 \]
 which leads to a contradiction when $p>2$. Therefore, $\widehat{\phi}(0)=0$.
\end{proof}

By the definition of the $\h_p^c$-norm and the duality in Theorem \ref{cor: duality}, we get the following result:
\begin{cor}
Let $1\leq p<\infty$ and $p\neq 2$. $\h_p^c(\R,\M)$ and $\mathcal H _p^c(\R,\M)$ are not equivalent.
\end{cor}

Although $\h_p^c(\R,\M)$ and $\mathcal H _p^c(\R,\M)$ do not coincide when $p\neq 2$, for those functions whose Fourier transforms vanish at the origin, their $\h_p^c$-norms and $\mathcal H _p^c$-norms are still equivalent.

\begin{thm}
Let $\phi \in\mathcal{S}$ such that  $\int_{\mathbb{R}^{d}}\phi (s)ds=1$.
\begin{enumerate}[\rm(1)]
\item If $1\leq p\leq 2$ and $f\in \h_p^c(\R,\M)$, then $f- \phi * f \in \mathcal{H}_{p}^{c}(\mathbb{R}^{d},\mathcal{M})$ and $\left\Vert f- \phi * f  \right\Vert _{\mathcal{H}_{p}^{c}}\lesssim\left\Vert f\right\Vert _{\h_{p}^{c}}$.
\item If $2< p<\infty$ and $f\in \mathcal{H}_p^c(\R,\M)$, then $f- \phi * f \in \h_{p}^{c}(\mathbb{R}^{d},\mathcal{M})$ and $\left\Vert f- \phi * f  \right\Vert _{\h_{p}^{c}}\lesssim\left\Vert f\right\Vert _{\mathcal{H}_{p}^{c}}$.
\end{enumerate}
\end{thm}

\begin{proof}
(1) Let $f\in \h_{p}^{c}(\mathbb{R}^{d},\mathcal{M})$ and $\Phi$ be a nondegenerate Schwartz function with vanishing mean. By the general characterization of $\mathcal{H}_p^c(\R,\M)$ in Lemma \ref{equivalence Hp}, $\|f- \phi * f  \|_{\mathcal{H}_p^c(\R,\M)}\approx \|G^c_\Phi(f- \phi * f  )\|_p$. Let us split $\|G^c_\Phi(f- \phi * f  )\|_p$ into two parts:
\begin{equation*}
\begin{split}
 &   \| G^c_\Phi(f- \phi * f  )\|_{p}\\
 & \lesssim \Big\|\big(\int_{0}^{1}\left|\Phi_{\varepsilon}*(f- \phi * f  )\right|^{2}\frac{d\varepsilon}{\varepsilon}\big)^{\frac{1}{2}}\Big\|_{p}+\Big\|\big(\int_{1}^{\infty}\left|\Phi_{\varepsilon}*(f- \phi * f  )\right|^{2}\frac{d\varepsilon}{\varepsilon}\big)^{\frac{1}{2}}\Big\|_{p}\\
 & =  \Big\|\big(\int_{0}^{1}|\Phi_{\varepsilon}*(f- \phi * f )|^{2}\frac{d\varepsilon}{\varepsilon}\big)^{\frac{1}{2}}\Big\|_{p}+\Big\|\big(\int_{1}^{\infty}|(\Phi_{\varepsilon}-\Phi_{\varepsilon}*\phi )*f |^{2}\frac{d\varepsilon}{\varepsilon} \big)^{\frac{1}{2}}\Big\|_{p}.
\end{split}
\end{equation*}
In order to estimate the first term in the last equality, we notice that  $\phi *f\in \h_{p}^{c}(\mathbb{R}^{d},\mathcal{M})$, 
 thus we have $f- \phi * f  \in \h_{p}^{c}(\mathbb{R}^{d},\mathcal{M})$. Then by Theorem \ref{thm main1}, this term can be majorized from above by $\left\Vert f\right\Vert _{\h_{p}^{c}}$.

To deal with the second term, we express it
as a Calder\'on-Zygmund operator with Hilbert-valued kernel. Let $H=L_{2}((1,+\infty),\frac{d\varepsilon}{\varepsilon})$
and define the kernel $\fk:\mathbb{R}^{d}\rightarrow H$ by $\fk(s)=\Phi_{\cdot}(s)-\Phi_{\cdot}*\phi (s)$ ($\Phi_{\cdot}(s)$
being the function $\varepsilon\mapsto\Phi_{\varepsilon}(s)$).
Now we prove that $\fk$ satisfies the hypotheses
of Corollary \ref{Cor C-Z}. The condition (1) of that corollary  is easy to verify. So we only check the conditions (2) and (3) there. By the fact that  $\int_{\mathbb{R}^{d}}\phi (s)ds=1$ and the mean value theorem, we have
\begin{equation*}
\begin{split}
\big|(\Phi_{\varepsilon}-\Phi_{\varepsilon}*\phi )(s)\big|
 & = \Big|\int_{\mathbb{R}^{d}}\left[\Phi_{\varepsilon}(s)-\Phi_{\varepsilon}(s-t)\right]\phi (t)dt\Big|\\
 & \leq  \int_{\mathbb{R}^{d}}\left|t\right|\frac{1}{\varepsilon^{d+1}}\sup_{0<\theta <1}\big| \nabla\Phi\big(\frac{s-\theta t}{\varepsilon}\big)\big|\left|\phi (t)\right|dt .
\end{split}
\end{equation*}
Then we split the last integral into two parts:
\begin{equation*}
\begin{split}
\big\|(\Phi_{\cdot}-\Phi_{\cdot}*\phi )(s)\big\|_{H}
 &  \lesssim   \Big(\int_{1}^{\infty}\big(\int_{\left|t\right|<\frac{\left|s\right|}{2}}\left|t\right|\frac{1}{\varepsilon^{d+1}}\sup_{0<\theta <1}\big|\nabla\Phi\big(\frac{s-\theta t}{\varepsilon}\big)\big|\left|\phi (t)\right|dt\big)^{2}\frac{d\varepsilon}{\varepsilon}\Big)^{\frac{1}{2}}\\
 &\;\;\;\;+\Big(\int_{1}^{\infty}\big(\int_{\left|t\right|>\frac{\left|s\right|}{2}}\left|t\right|\frac{1}{\varepsilon^{d+1}}\sup_{0<\theta <1}\big|\nabla\Phi\big(\frac{s-\theta t}{\varepsilon}\big)\big|\left|\phi (t)\right|dt\big)^{2}\frac{d\varepsilon}{\varepsilon}\Big)^{\frac{1}{2}}\\
& \stackrel{\mathrm{def}}{=}  \rm I+\rm II.
\end{split}
\end{equation*}
If $|t |<\frac{|s |}{2}$, we have $|s-\theta t |\geq\frac{|s|}{2}$,
thus $|\nabla\Phi(\frac{s-\theta t}{\varepsilon})|\lesssim\frac{\varepsilon^{d+\frac{1}{2}}}{\left|s\right|^{d+\frac{1}{2}}}$ for any $0\leq \theta \leq 1$. Then
\begin{equation*}
{\rm I }    \lesssim \big(\int_{1}^{\infty}\frac{1}{\varepsilon^2}     d\e\big)^{\frac{1}{2}}\frac{1}{\left|s\right|^{d+\frac{1}{2}}} \lesssim  \frac{1}{\left|s\right|^{d+\frac{1}{2}}}.
\end{equation*}
When $| t |>\frac{|s|}{2}$, since $\phi \in\mathcal{S}$,
we have $\int_{\left|t\right|>\frac{\left|s\right|}{2}}\left|t\right|\left|\phi (t)\right|dt\lesssim\frac{1}{\left|s\right|^{d+\frac{1}{2}}}$.
Hence
\begin{equation*}
{\rm II}    \lesssim \big(\int_{1}^{\infty}\frac{1}{\varepsilon^{2d+2}}\frac{d\varepsilon}{\varepsilon}\big)^{\frac{1}{2}}\cdot\frac{1}{\left|s\right|^{d+\frac{1}{2}}} \lesssim  \frac{1}{\left|s\right|^{d+\frac{1}{2}}}.
\end{equation*}
The estimates of $\rm I$ and $\rm II$ imply
\[
\| (\Phi_{\varepsilon}-\Phi_{\varepsilon}*\phi )(s)\| _{H}\lesssim\frac{1}{\left|s\right|^{d+\frac{1}{2}}}.
\]
In a similar way, we obtain
\[
\| \nabla (\Phi_{\varepsilon}-\Phi_{\varepsilon}*\phi )(s)\| _{H}\lesssim\frac{1}{\left|s\right|^{d+1}}.
\]
Thus, it follows from Corollary \ref{Cor C-Z} that $\Big\|(\int_{1}^{\infty}\left|(\Phi_{\varepsilon}-\Phi_{\varepsilon}*\phi )*f\right|^{2}\frac{d\varepsilon}{\varepsilon})^{\frac{1}{2}}\Big\|_{p}$ is also majorized from above by $\left\Vert f\right\Vert _{\h_{p}^{c}}$.

(2) The case $p>2$ can be deduced from the duality between $\h_p^c$ and $\h_q^c$ (Theorem \ref{cor: duality}) and that between $\mathcal{H}_p^c$ and $\mathcal{H}_q^c$ ($q$ being the conjugate index of $p$). There exists  $g\in \h_q^c(\R,\M)$ with norm one such that 
\begin{equation*}
\begin{split}
\| f-\phi*f \|_{\h_p^c}& =\Big| \tau \int_{\R}(f-\phi*f)(s)g^*(s)ds \Big| \\ 
& =   \Big| \tau \int_{\R}f(s)(g^*-\phi * g^*)(s)ds\Big| \\
& \leq   \| f\|_{\mathcal{H}_p^c}  \| g-\overline{\phi} * g\|_{\mathcal{H}_q^c} \lesssim  \| f\|_{\mathcal{H}_p^c} \| g\|_{\h_q^c}=\| f\|_{\mathcal{H}_p^c} ,
\end{split}
\end{equation*}
which completes the proof.
\end{proof}

From the interpolation result of  mixture local hardy spaces in Proposition \ref{interp-mix}, we can deduce the equivalence between  mixture local Hardy spaces and $L_p$-spaces.

\begin{prop}\label{equi-Hp-hp-Lp}
 For any $1<p<\infty$, $\h_p(\R,\M)=\mathcal{H}_p(\R,\M)= L_p(\N)$ with equivalent norms.
 \end{prop}
\begin{proof}
It is  known that $\mathcal{H}_p(\R,\M)= L_p(\N)$ with equivalent norms. One can see \cite[Corollary 5.4]{Mei2007} for more details. One the other hand, since $L_\infty(\N)\subset \bmo^c(\R,\M)$, by duality, we get $\h_1^c(\R,\M)\subset L_1(\N)$. Combining \eqref{rem: h2=L2} and the interpolation result in Theorem \ref{thm: interpolation}, we deduce that $\h_p^c(\R,\M)\subset L_p(\N)$ for any $1< p\leq 2$ and $L_p(\N)\subset \h_p^c(\R,\M)$ for any $ 2<p< \infty$. Similarly, we also have $\h_p^r(\R,\M)\subset L_p(\N)$ for any $1< p\leq 2$ and $L_p(\N)\subset \h_p^r(\R,\M)$ for any $2<p< \infty$. Combined with \eqref{eq: inclu p<2} and \eqref{eq: inclu p>2}, we get
\begin{equation}\label{eq: hp subset Lp}
 \mathcal{H} _p(\R,\M) \subset \h_p(\R,\M)\subset L_p(\N) \quad \text{for}\quad 1< p\leq 2,
\end{equation}
and
\begin{equation}\label{eq: Lp subset hp}
L_p(\N)\subset \h_p(\R,\M) \subset  \mathcal{H} _p(\R,\M) \quad \text{for}\quad 2<p< \infty.
\end{equation}
Then \eqref{eq: hp subset Lp}, \eqref{eq: Lp subset hp} and  \cite[Corollary 5.4]{Mei2007} imply that 
 \[
\h_p(\R,\M)=\mathcal{H}_p(\R,\M)= L_p(\N) \quad \text{for}\quad 1< p< \infty,
 \]
which completes the proof.
\end{proof}

%%%%%%%%%%%%%%%%%%%%%%
\section{The atomic decomposition}

In this section, we give the atomic decomposition of $\h_1^c(\R,\M)$. The atomic decomposition of $\mathcal{H}_1^c(\R,\M)$ studied in \cite{Mei2007} and the characterizations obtained in the last section will be the main tools for us.
\begin{defn}\label{def: atom h1}
Let $Q$ be a cube in $\R$ with $| Q|\leq 1$. If $| Q|=1$, an $\h _1^c$-atom associated with $Q$ is a function $a\in L_{1}  (\mathcal{M};L_{2}^{c}  (\mathbb{R}^{d} ) )$ such that
\begin{itemize}
\item $\supp a\subset Q$;
\item $\tau\big(\int_{Q}  |a  (s ) |^{2}ds\big)^{\frac{1}{2}}\leq  |Q |^{-\frac{1}{2}}$.
\end{itemize}
If $|Q|<1$, we assume additionally: 
\begin{itemize}
\item $\int_{Q}a(s)ds=0.$
\end{itemize}
\end{defn}

Let $\h_{1,{\rm at}}^{c}  (\mathbb{R}^{d},\mathcal{M} )$ be the
space of all $f$ admitting a representation of the form
\[
f=\sum_{j=1}^\infty \lambda_ja_j,
\]
where the $a_{j}$'s are $\h_1^{c}$-atoms and $\lambda_{j}\in\mathbb{C}$ such that $\sum_{j=1}^\infty  |\lambda_{j} |<\infty$. The above
series  converges in the sense of distribution.
We equip $\h_{1,{\rm at}}^{c} (\mathbb{R}^{d},\mathcal{M} )$ with
the following norm:
\[
  \|  f \|  _{\h_{1,{\rm at}}^{c}}=\inf  \{ \sum_{j=1}^\infty  |\lambda_{j} |: f=\sum_{j=1}^\infty\lambda_{j}a_{j};\,\mbox{\ensuremath{a_{j}}'s are \ensuremath{\h_1^{c}} -atoms, }\mbox{\ensuremath{\lambda}}_{j}\in\mathbb{C} \} .
\]
Similarly, we define the row version $\h_{1,{\rm at}}^{r}  (\mathbb{R}^{d},\mathcal{M} )$.
Then we set
\[
\h_{1,{\rm at}}  (\mathbb{R}^{d},\mathcal{M} )=\h_{1,{\rm at}}^{c} (\mathbb{R}^{d},\mathcal{M} )+\h_{1,{\rm at}}^{r}  (\mathbb{R}^{d},\mathcal{M} ).
\]

\begin{thm}\label{thm:atomic h1}
We have $\h_{1,{\rm at}}^{c} (\mathbb{R}^{d},\mathcal{M} )=\h_{1}^{c}  (\mathbb{R}^{d},\mathcal{M} )$
with equivalent norms.
\end{thm}
\begin{proof}
First, we show the inclusion $\h_{1,{\rm at}}^c(\R,\M)\subset \h_{1}^c(\R,\M)$. To this end, it suffices to prove that for any atom $a$ in Definition \ref{def: atom h1}, we have
\begin{equation}\label{eq: a bounded}
\| a\|_{\h_1^c}\lesssim 1.
\end{equation}
Recall that the atomic decomposition of $\mathcal{H}_1^c(\R,\M)$ has been  considered in \cite{Mei2007}. An $\mathcal{H}_1^c$-atom is a function $b\in L_{1}  (\mathcal{M};L_{2}^{c}  (\mathbb{R}^{d} ) )$
such that, for some cube $Q$,
\begin{itemize}
\item $\supp b\subset Q$;
\item $\int_{Q}b(s)ds=0$;
\item $\tau\big(\int_{Q}  | b  (s ) |^{2}ds\big)^{\frac{1}{2}}\leq  |Q |^{-\frac{1}{2}}$.
\end{itemize}
If $a$ is supported in $Q$ with $|Q|<1$, then $a$ is also an $\mathcal{H}_1^c$-atom, so $\|a\|_{\h_1^c} \lesssim \|a\|_{\mathcal{H}_1^c}\lesssim 1$. Now assume that the supporting cube $Q$ of $a$ is of side  length one.
 We use the discrete characterization obtained in Theorem \ref{thm: equivalence hpD}, i.e.
\[
\| a \|_{\h_1^c}\approx \big\| (\sum_{j=1}^{\infty}|\Phi_j*a|^2)^{\frac{1}{2}} \big\|_{1}+\|\phi*a\|_{1}.
\]
Apart from the assumption on $\Phi$ and $\phi$ in Theorem \ref{thm main1}, we may take $\Phi$ and $\phi$ satisfying
$$\supp \Phi, \,\supp \phi \subset B_1=\{s\in\R: |s|\leq 1\}.$$
Then
\[
\supp \phi*a \subset 3Q \;\;\;\mbox{and}\,\;\;\supp \Phi_\e*a \subset 3Q \quad \text{for any}\quad 0<\e<1.
\]
By the Cauchy-Schwarz inequality we have
\[
\|\phi*a\|_{1}\leq \int_{3Q} \big(\int_{Q}|\phi (t-s)|^2ds\big)^\frac{1}{2} \cdot \tau \big( \int |a(s)|^2 ds\big)^\frac{1}{2}\,dt\lesssim  1.
\]
Similarly,
\begin{equation*}
\begin{split}
 \big\| (\sum_{j=1}^{\infty}|\Phi_j*a|^2)^{\frac{1}{2}} \big\|_{1} & =  \tau \int_{3Q} (\sum_{j=1}^{\infty}|\Phi_j*a(s)|^2)^{\frac{1}{2}}ds\\
&  \lesssim   \tau \Big( \int_{3Q}  \sum_{j=1}^{\infty}|\Phi_j*a(s)|^2 ds\Big)^\frac{1}{2}\\
&  =  \tau \Big( \int_{\R} \sum_{j=1}^{\infty}|\widehat{\Phi}(2^{-j}\xi)\widehat{a}(\xi)|^2 d\xi\Big)^\frac{1}{2}\\
&  \leq   \tau \big( \int |a(s)|^2 ds\big)^\frac{1}{2}\leq 1.
\end{split}
\end{equation*}
Therefore, $\h_{1,{\rm at}}^c(\R,\M)\subset \h_{1}^c(\R,\M)$.

Now we turn to proving the reverse inclusion. Observe that  $\mathcal{H}_1^c$-atoms are also $\h_1^c$-atoms. Then by the atomic decomposition of $\mathcal{H}_1^c(\R,\M)$ and the duality between $\mathcal{H}_1^c(\R,\M)$ and $\BMO^c(\R,\M)$, every continuous functional $\ell$ on
$\h_{1,{\rm at}}^{c} (\mathbb{R}^{d},\mathcal{M} )$ corresponds to a function $g\in \BMO^c(\R,\M)$. Moreover, since for any cube $Q$ with side length one, $L_{1}\big(\mathcal{M};L_{2}^{c}  (Q ) \big)\subset \h_{1,{\rm at}}^{c} (\mathbb{R}^{d},\mathcal{M} )$, $\ell $ induces a continuous
functional on $L_{1}\big(\mathcal{M};L_{2}^{c}  (Q ) \big)$
with norm less than or equal to $ \|  \ell \|  _{  (\h_{1,{\rm at}}^{c})^{*}}$.
Thus, the function $g$ satisfies the condition that 
\begin{equation}\label{eq: g bmo}
g\in \BMO^c(\R,\M)\quad \text{and}\quad\underset{\substack{Q\subset \R\\ |Q|=1}}\sup \| g|_Q\|_{L_\infty(\M;L_{2}^{c}(Q))}\leq \|  \ell \|  _{  (\h_{1,{\rm at}}^{c})^{*}}.
\end{equation}
Consequently, $g\in \bmo^c(\R,\M)$ and $$\ell  (f )=\tau\int_{\mathbb{R}^{d}}f  (s )g^{*}  (s )ds,\,  \forall\, f\in \h_{1,{\rm at}}^{c} (\mathbb{R}^{d},\mathcal{M} ).$$
Thus, $\h_{1,{\rm at}}^c(\R,\M)^*\subset \bmo ^c(\R,\M)$. On the other hand, by the previous result, we have $\bmo ^c(\R,\M)\subset \h_{1,{\rm at}}^c(\R,\M)^*$. Thus, $\h_{1,{\rm at}}^c(\R,\M)^*=\bmo ^c(\R,\M)$ with equivalent norms. Since $\h_{1,{\rm at}}^c(\R,\M)\subset \h_{1}^c(\R,\M)$ densely, we deduce that $\h_{1,{\rm at}}^c(\R,\M)=\h_{1}^{c}  (\mathbb{R}^{d},\mathcal{M} )$ with equivalent norms.
\end{proof}

\noindent{\bf Acknowledgements.} The authors are greatly indebted to Professor Quanhua Xu for having suggested to them the subject of this paper, for many helpful discussions and very careful reading of this paper. The authors are partially supported by the the National Natural Science Foundation of China (grant no. 11301401).

\end{document}